\documentclass[12pt,english]{article}
\usepackage{refstyle}

\usepackage[T1]{fontenc}
\usepackage[latin9]{inputenc}
\usepackage{geometry}
\usepackage{array}
\usepackage{verbatim}
\usepackage{float}
\usepackage{textcomp}
\usepackage{mathtools}
\usepackage{multirow}
\usepackage{amsmath}
\usepackage{amsthm}
\usepackage{amssymb}
\usepackage{stmaryrd}
\usepackage{graphicx}
\usepackage{setspace}

\makeatletter

\providecommand{\tabularnewline}{\\}

\theoremstyle{plain}
\newtheorem{thm}{\protect\theoremname}
\theoremstyle{remark}
\newtheorem{rem}[thm]{\protect\remarkname}
\theoremstyle{plain}
\newtheorem{prop}[thm]{\protect\propositionname}

\@ifundefined{date}{}{\date{}}
\usepackage{xcolor}
\usepackage{babel}
\usepackage{array}
\usepackage{verbatim}
\usepackage{float}
\usepackage{textcomp}
\usepackage{enumitem}
\usepackage{multirow}
\usepackage{amsthm}
\usepackage{graphicx}

\providecommand{\tabularnewline}{\\}
\RS@ifundefined{subsecref}{\newref{subsec}{name = \RSsectxt}}{}
\RS@ifundefined{thmref}{\def\RSthmtxt{theorem~}\newref{thm}{name = \RSthmtxt}}{}
\RS@ifundefined{lemref}{\def\RSlemtxt{lemma~}\newref{lem}{name = \RSlemtxt}}{}

\numberwithin{figure}{section}
\IfFileExists{candleshape.def}{%
 \input{candleshape.def}}{}
\IfFileExists{dropshape.def}{%
 \input{dropshape.def}}{}
\IfFileExists{TeXshape.def}{%
 \input{TeXshape.def}}{}
\IfFileExists{triangleshapes.def}{%
 \input{triangleshapes.def}}{}

\theoremstyle{plain}
\newcommand{\pp}{\rho}

\usepackage{etoolbox}
\usepackage{breqn}
\usepackage{diagbox}

\usepackage{tikz}
\usepackage{pgfplots}
\pgfplotsset{compat=newest}
\usepackage{pgf}
\usepackage{tikz}\usetikzlibrary{arrows}

\DeclareMathOperator{\sinc}{sinc}
\DeclareMathOperator{\ii}{Im}
\usepackage{mathtools}

\newcounter{example}
\newenvironment{example}[1][]{\refstepcounter{example}\par\medskip
   \noindent \textbf{Example~\theexample. #1} \rmfamily}{\medskip}

\makeatother

\usepackage{babel}
\providecommand{\propositionname}{Proposition}
\providecommand{\remarkname}{Remark}
\providecommand{\theoremname}{Theorem}

\begin{document}
\title{\textbf{Sinc method in spectrum completion and inverse Sturm-Liouville
problems }}
\maketitle
\begin{center}
Vladislav V. Kravchenko\textsuperscript{1}, L. Estefania Murcia-Lozano\textsuperscript{1} 
\par\end{center}

\begin{center}
\textsuperscript{1}Departamento de Matemáticas, Cinvestav, Unidad
Querétaro 
\par\end{center}

\begin{center}
Libramiento Norponiente \#2000, Fracc. Real de Juriquilla, Querétaro,
Qro. C.P. 76230 México 
\par\end{center}

\begin{center}
vkravchenko@math.cinvestav.edu.mx, emurcia@math.cinvestav.mx 
\par\end{center}
\begin{abstract}
Cardinal series representations for solutions of the Sturm-Liouville
equation $-y''+q(x)y=\rho^{2}y$, $x\in(0,L)$ with a complex valued
potential $q(x)$ are obtained, by using the corresponding transmutation
operator. Consequently, partial sums of the series approximate the
solutions uniformly with respect to $\rho$ in any strip $\left|\text{Im}\rho\right|<C$
of the complex plane. This property of the obtained series representations
leads to their applications in a variety of spectral problems. In
particular, we show their applicability to the spectrum completion
problem, consisting in computing large sets of the eigenvalues from
a reduced finite set of known eigenvalues, without any information
on the potential $q(x)$ as well as on the constants from boundary
conditions. Among other applications this leads to an efficient numerical
method for computing a Weyl function from two finite sets of the eigenvalues.
This possibility is explored in the present work and illustrated by
numerical tests. 

Finally, based on the cardinal series representations obtained, we
develop a method for the numerical solution of the inverse two-spectra
Sturm-Liouville problem and show its numerical efficiency.
\end{abstract}

\section{Introduction}

Let $q\in\mathcal{L}_{2}\left[0,L\right]$ be complex valued, $L>0$.
Consider the Sturm-Liouville equation
\begin{equation}
-y''+q(x)y=\rho^{2}y,\,\,0<x<L,\label{eq:PrincipalEq}
\end{equation}
where $\rho\in\mathbb{C}$ is a spectral parameter. 

Let $\phi(\rho,x)$ and $S(\rho,x)$ be the solutions of (1) satisfying
the initial conditions 
\[
\phi(\rho,0)=1,\quad\phi^{\prime}(\rho,0)=h,\quad h\in\mathbb{C},
\]
\[
S(\rho,0)=0,\quad S^{\prime}(\rho,0)=1
\]
for all $\rho\in\mathbb{C}$. For any $x>0$ fixed, and as functions
of $\rho$, the functions $\phi(\rho,x)-\cos(\rho x)$ and $S(\rho,x)-\frac{\sin(\rho x)}{\rho}$
are entire and, moreover, belong to the Paley-Wiener class $PW_{x}^{2}=\{f\in\mathcal{L}_{2}(\mathbb{R}):\exists g\in\mathcal{L}_{2}(-x,x),f(z)=\int_{-x}^{x}g(u)e^{iuz}du\}$.
This is because of the existence of the $\mathcal{L}_{2}$-functions
$\mathbf{G}(x,\cdot)$ and $\mathbf{S}(x,\cdot)$, such that 
\[
\phi(\rho,x)=\cos\left(\rho x\right)+\int_{0}^{x}\mathbf{G}(x,t)\cos\left(\rho t\right)\,dt
\]
and 
\[
S(\rho,x)=\frac{\sin\left(\rho x\right)}{\rho}+\int_{0}^{x}\mathbf{S}(x,t)\frac{\sin\left(\rho t\right)}{\rho}dt
\]
for all $\rho\in\mathbb{C}$. These Volterra integral operators of
the second kind are known as transmutation or transformation operators
and play a key role in spectral theory and theory of inverse problems.
For their theory we refer to \cite{Levitan}, \cite{MArchenko1},
\cite{Marchenko} and \cite{yurko}.

Several different series representations for the solutions $\phi(\rho,x)$
and $S(\rho,x)$ are known. The spectral parameter power series (SPPS)
representations (see \cite{Karapetyants}, \cite{KrCV08}, \cite{APFT},
\cite{PorterKrav}, \cite{KKRosu}, \cite{VK2020} and references
therein) give the solutions $\phi(\rho,x)$ and $S(\rho,x)$ in the
form of power series in terms of $\rho$. The coefficients of the
series are calculated following a simple recurrent integration procedure
and in fact, they are nothing but the images of the Taylor series
coefficients of the functions $\cos\left(\rho x\right)$ and $\sin\left(\rho x\right)/\rho$
under the action of the respective transmutation operators. The SPPS
converge uniformly in any compact set of the complex $\rho$-plane
and find numerous applications when solving direct spectral problems
(we refer to the review \cite{KKRosu} and references in \cite[p. 11]{VK2020}).

In \cite{Vk2017NSBF}, by expanding the kernels $\mathbf{G}(x,t)$
and $\mathbf{S}(x,t)$ into series of Legendre polynomials, the Neumann
series of Bessel functions (NSBF) representations for $\phi(\rho,x)$
and $S(\rho,x)$ were obtained (equalities (\ref{eq:NsbfS}) and (\ref{eq:NsbfPhi})
below). They proved to be extremely useful for solving both the direct
and the inverse spectral problems for (\ref{eq:PrincipalEq}) (see
\cite{Vk2017NSBF}, \cite{VK2020}, \cite{VkSpectralCompl2023}, \cite{VKComplex})
due to a couple of very convenient properties: they converge uniformly
with respect to $\rho\in\mathbb{R}$ (in fact, with respect to $\rho$
belonging to a strip $\left\vert \text{Im}\left(\rho\right)\right\vert \leq C$),
and the first coefficient of the series is sufficient to reconstruct
the potential $q(x)$. In \cite{VKComplex}, for continuously differentiable
potentials other NSBF representations for the solutions $\phi(\rho,x)$
and $S(\rho,x)$ were derived, which not only converge uniformly with
respect to $\rho\in\mathbb{R}$ but also admit estimates for the remainders
of the series which guarantee even the improvement of the convergence
for large values of $\left\vert \rho\right\vert $ ($\left\vert \text{Im}\left(\rho\right)\right\vert \leq C$).
They were used for solving a variety of inverse coefficient problems
for (\ref{eq:PrincipalEq}).

The belonging of the functions $\phi(\rho,x)-\cos(\rho x)$ and $S(\rho,x)-\frac{\sin(\rho x)}{\rho}$
to the Paley-Wiener class $PW_{x}^{2}$, due to the Whittaker-Shannon-Kotelnikov
sampling theorem, naturally leads to their representations in the
form of cardinal series (involving the $\sinc$ function, see e.g.,
\cite{Higgings}, \cite{Marks}, \cite{Whittaker} and \cite{Zayed}).
They were used for solving direct Sturm-Liouville problems in a number
of papers, e.g., \cite[Chapter 10]{Baumann}, \cite{BoumenirChananeEigenvalues1996},
\cite{Boumenir periodic}, \cite{Boumenir}, \cite{Chanane}, \cite{Eggert},
\cite{Gau}, \cite{Tharwat}, and associated with a general Sinc method
for ordinary and partial differential equations, e.g., \cite{Alkaheld},
\cite{Annaby2007}, \cite{Marks}, \cite{Parand}, \cite{Stenger},
\cite{Stenger 200} and references therein, \cite{Tharwat2}.

In the present paper we develop the Sinc method by obtaining new cardinal
series for the solutions $\phi(\rho,x)$ and $S(\rho,x)$ and applying
them to the spectrum completion, to the computation of Weyl's function
from two spectra as well as to the solution of the two-spectra inverse
Sturm-Liouville problem.

The spectrum completion refers to a technique which allows one to
compute a large set of the eigenvalues of a Sturm-Liouville problem
from several given eigenvalues, without knowing the potential $q(x)$.
This possibility was explored in \cite{VkSpectralCompl2023} for real-valued
potentials, and with the aid of the NSBF representations. It was shown
that this technique gives more accurate results than, e.g., implementing
well known asymptotic formulas for the eigenvalues. In the present
work we show that the cardinal series representations for the solutions
are also applicable to this problem, including the spectrum completion
for complex valued potentials. Comparison with the NSBF series shows
that the cardinal series are even preferable for dealing with noisy
data. The approach is based on converting the knowledge of several
eigenvalues into the series representation coefficients of the characteristic
function. Due to the uniform convergence of the series in strips ($\left\vert \text{Im}\left(\rho\right)\right\vert \leq C$),
this gives us the possibility to compute large sets of zeros of the
characteristic function, thus completing the spectrum. The same idea
is used in the present work for computing the Weyl function from the
corresponding two spectra. The Weyl function can be regarded as a
quotient of two characteristic functions of Sturm-Liouville problems,
and for its calculation there are known analytical formulas (see,
e.g., \cite{FreYurko}, \cite[Lemma 4.1]{Hati}) which unfortunately
are not of a practical applicability. In the present work we use the
known eigenvalues from two spectra for computing the coefficients
of the characteristic functions, that leads to a simple and efficient
method for computing the corresponding Weyl function.

Finally, we develop a method for solving the classical two-spectra
inverse Sturm-Liouville problem, which consists in the recovery of
a potential $q(x)$ and boundary conditions from the spectra of two
different Sturm-Liouville problems for (\ref{eq:PrincipalEq}).

This problem has been extensively studied (see, e.g., \cite{Borg},
\cite{Chadan}, \cite{VK2020}, \cite{Rundell}), even in the case
of a complex valued potential $q(x)$, (see, \cite{Bondarenko}, \cite{Buterin},
\cite{FreYurko}, \cite{Hor2}, \cite{VKComplex}, \cite{Marletta}).
It encounters practical applications across various fields (see, e.g.,
\cite{Barcilon}, \cite{Gladewell}). 

In a number of publications numerical methods for the two-spectra
inverse problems have been developed, see, e.g., \cite{Chadan}, \cite{Gao},
\cite{Kammanee}, \cite{VK2021}, \cite{KKCMAthematics}, \cite{VkSpectralCompl2023},
\cite{Neamaty}, \cite{R=0000F6hrl} and \cite{Rundell}, however,
for real-valued potentials. It is worth mentioning that \cite{VK2021}
was the first work where a numerical method developed allowed recovering
the boundary conditions. Here we present a method for the numerical
solution of the two-spectra inverse problems for complex potentials,
developing a recent result from \cite{VKComplex}, where an approach
to solving inverse coefficient problems for complex potentials was
proposed. Moreover, to our best knowledge the present work is the
first one, in which the application of the Sinc method to the inverse
spectral problems is developed.

The paper is structured as follows. In Section 2 we introduce necessary
notations, definitions and properties concerning the transmutation
operators. In Section 3 we introduce the NSBF representations together
with their basic properties. In Section 4, Fourier expansions of the
transmutation kernels and the Whittaker--Shannon--Kotelnikov sampling
theorem are used to obtain one of the main results of this work, the
cardinal series representations, Theorem \ref{thm:Let-.-The OldSIn}
and Theorem \ref{thm:Let-.-TheNewSinc}. Additionally, analogous results
are obtained for the derivatives of solutions of (\ref{eq:PrincipalEq}).
Section 5 presents an approach for solving the spectrum completion
problem together with numerical results and discussion. In Section
6 an approximation of the Weyl function from two spectra of (\ref{eq:PrincipalEq})
using Sinc methods is discussed. Section 7 presents an approach for
solving the two-spectra inverse Sturm-Liouville problem together with
numerical illustrations. Finally, Section 8 presents some concluding
remarks.

\section{Preliminaries }

By $S\left(\rho,x\right)$, $T\left(\rho,x\right)$, $\phi(\rho,x)$
and $\psi\left(\rho,x\right)$ we denote the solutions of (\ref{eq:PrincipalEq})
satisfying the initial conditions 
\begin{align*}
S(\rho,0) & =0,S'(\rho,0)=1,\,T(\rho,L)=0,T'(\rho,L)=1,\\
\phi(\rho,0) & =1,\phi'(\rho,0)=h,\,\psi(\rho,L)=1,\psi'(\rho,L)=-H,
\end{align*}
respectively, where $h$ and $H$ are some complex constants. 

Denote
\[
\omega(x)=\frac{1}{2}\int_{0}^{x}q(s)ds.
\]

We will use the notation $f_{1}(x,t)$ and $f_{2}(x,t)$ for the partial
derivatives with respect to the first or second argument, respectively
($f_{11}(x,t)$ and $f_{22}(x,t)$ for the second derivatives). 

We recall the integral representations for $S\left(\rho,x\right)$
and $\phi(\rho,x)$ arising in the theory of transmutation operators
(see, e.g., \cite{VK2020}, \cite{Marchenko}, \cite{Sitnik}, \cite{yurko}). 
\begin{thm}
\label{thm: FourierRepTransKernel}Let $q\in\mathcal{L}_{2}\left[0,L\right]$.
There exist functions $\boldsymbol{\mathcal{S}}(x,t)$ and $\boldsymbol{\mathcal{G}}(x,t)$
defined in the domain $0\leq t\leq x\leq L,$ such that
\begin{align}
S(\rho,x) & =\frac{\sin(\rho x)}{\rho}+\frac{1}{\rho}\int_{0}^{x}\boldsymbol{\mathcal{S}}(x,t)\sin(\rho t)dt,\label{eq:TransOpS}\\
\phi(\rho,x) & =\cos(\rho x)+\int_{0}^{x}\boldsymbol{\mathcal{G}}(x,t)\cos(\rho t)dt,\label{eq: TransOpPhi}
\end{align}
for all $\rho\in\mathbb{C}$. The functions $\boldsymbol{\mathcal{S}}(x,\cdot)$
and $\boldsymbol{\mathcal{G}}(x,\cdot)$ possess the same regularity
as $\int_{0}^{x}q(t)dt$, see \cite[p. 22]{FreYurko}. In particular,
$\boldsymbol{\mathcal{S}}(x,\cdot),\boldsymbol{\mathcal{G}}(x,\cdot)\in W_{2}^{1}[0,x]$
($W_{2}^{1}\left[0,L\right]=\left\{ f\in\mathcal{L}_{2}\left[0,L\right]\,|\,f'(x)\in\mathcal{L}_{2}\left[0,L\right]\right\} $),
see the proof of Proposition 19 in \cite{Campos}. 

The derivatives of $S\left(\rho,x\right)$ and $\phi(\rho,x)$ with
respect to $x$ admit the integral representations
\begin{align}
S'(\rho,x) & =\cos(\rho x)+\frac{\omega(x)}{\rho}\sin(\rho x)+\frac{1}{\rho}\int_{0}^{x}\boldsymbol{\mathcal{S}}_{1}(x,t)\sin(\rho t)dt,\label{eq:TransOpS-1}\\
\phi'(\rho,x) & =-\rho\sin(\rho x)+\left(h+\omega(x)\right)\cos(\rho x)+\int_{0}^{x}\boldsymbol{\mathcal{G}}_{1}(x,t)\cos(\rho t)dt,\label{eq: TransOpPhi-1}
\end{align}
for $0\leq t\leq x\leq L$, where $\boldsymbol{\mathcal{S}}_{1}(x,\cdot)$,
$\boldsymbol{\boldsymbol{\mathcal{G}}}_{1}(x,\cdot)\in\mathcal{L}_{2}\left[0,x\right]$. 
\end{thm}

Assume, $q\in W_{2}^{1}\left[0,L\right]$. The transmutation kernels
$\boldsymbol{\mathcal{S}}(x,t)$ and $\boldsymbol{\mathcal{G}}(x,t)$
defined on the domain $0\leq t\leq x\leq L$ are twice differentiable
with respect to each of the arguments and satisfy the equalities \cite{Kravchenko2016}
\begin{align}
\boldsymbol{\mathcal{S}}(x,x) & =\omega(x),\boldsymbol{\mathcal{\:S}}(x,0)=0,\label{eq:TransOpSxx}\\
\boldsymbol{\mathcal{G}}(x,x) & =h+\omega(x),\:\boldsymbol{\mathcal{G}}_{2}(x,0)=0,\label{eq:TransOpGxx}
\end{align}
\begin{equation}
\boldsymbol{\mathcal{S}}_{2}(x,x)=\frac{1}{4}\left(q(x)-2\omega^{2}(x)+q(0)\right),\,\,\boldsymbol{\mathcal{S}}_{2}(x,0)=\frac{1}{2}q\left(\frac{x}{2}\right),\label{eq:DerivativesKernels}
\end{equation}
\begin{equation}
\boldsymbol{\mathcal{\mathcal{G}}}_{2}(x,x)=\frac{1}{4}\left(q(x)-4h\omega(x)-2\omega^{2}(x)-q(0)\right),\label{eq:DerKernelG}
\end{equation}

\[
\boldsymbol{\mathcal{S}}_{22}(x,x)=\frac{1}{8}\left(q'(x)-q'(0)-2q(x)\omega(x)-2q(0)\omega(x)-\int_{0}^{x}q^{2}(s)ds+\frac{8}{6}\omega^{3}(x)\right)
\]

and 

\[
\boldsymbol{\mathcal{G}}_{22}(x,x)=\frac{1}{8}\left(q'(x)+q'(0)-2q(x)\omega(x)+2q(0)\omega(x)-\int_{0}^{x}q^{2}(s)ds+\frac{8}{6}\omega^{3}(x)\right),
\]
where $\boldsymbol{\mathcal{S}}_{22}(x,\cdot),$ $\boldsymbol{\mathcal{G}}_{22}(x,\cdot)\in\mathcal{L}_{2}\left[0,x\right]$.
\begin{thm}
\label{thm:qC1SolutionsInegralRepSPhi}Let $q\in W_{2}^{1}\left[0,L\right]$.
The solutions $S(\rho,x)$ and $\phi(\rho,x)$ admit the integral
representations
\begin{align}
S(\rho,x) & =\frac{\sin(\rho x)}{\rho}-\frac{\cos(\rho x)}{\rho^{2}}\omega(x)+\frac{\text{\ensuremath{\sin}}(\rho x)}{4\rho^{3}}\left(q(x)-2\omega^{2}(x)+q(0)\right)\label{eq:NewRepIntegralS}\\
 & -\frac{1}{\rho^{3}}\int_{0}^{x}\boldsymbol{\mathcal{S}}_{22}(x,t)\text{\ensuremath{\sin}}(\rho t)dt,\,\rho\in\mathbb{C\setminus}\left\{ 0\right\} ,\nonumber \\
\phi(\rho,x) & =\cos(\rho x)+\frac{\sin(\rho x)}{\rho}(h+\omega(x))+\frac{\cos(\rho x)}{4\rho^{2}}\left(q(x)-4h\omega(x)-2\omega^{2}(x)-q(0)\right)\label{eq:NewRepIntegralPhi}\\
 & -\frac{1}{\rho^{2}}\int_{0}^{x}\mathcal{\boldsymbol{G}}_{22}(x,t)\cos(\rho t)dt,\,\rho\in\mathbb{C\setminus}\left\{ 0\right\} .\nonumber 
\end{align}
\end{thm}

\begin{proof}
The proof is analogous to the proof of Proposition 5 in \cite{Vk2017}.
Integrating by parts in (\ref{eq:TransOpS}) and using (\ref{eq:TransOpSxx})
we obtain

\begin{equation}
S(\rho,x)=\frac{\sin(\rho x)}{\rho}-\frac{\cos(\rho x)}{\rho^{2}}\omega(x)+\frac{1}{\rho{{}^2}}\int_{0}^{x}\boldsymbol{\mathcal{S}}_{2}(x,t)\cos(\rho t)dt.\label{eq:AuxS}
\end{equation}
Similarly, integration by parts applied to the integral in (\ref{eq:AuxS})
and taking into account (\ref{eq:DerivativesKernels}) leads to the
equalities 
\begin{align}
\int_{0}^{x}\boldsymbol{\mathcal{S}}_{2}(x,t)\cos(\rho t)dt & =\left.\boldsymbol{\mathcal{S}}_{2}(x,t)\frac{\sin(\rho t)}{\rho}\right|_{t=x}-\int_{0}^{x}\boldsymbol{\mathcal{S}}_{22}(x,t)\frac{\sin(\rho t)}{\rho}dt\nonumber \\
 & =\frac{\sin(\rho x)}{4\rho}\left(q(x)-2\omega^{2}(x)+q(0)\right)-\frac{1}{\rho}\int_{0}^{x}\boldsymbol{\mathcal{S}}_{22}(x,t)\sin(\rho t)dt.\label{eq:intS2}
\end{align}

Substitution of (\ref{eq:intS2}) into (\ref{eq:AuxS}) leads to (\ref{eq:NewRepIntegralS}).

Integrating by parts in (\ref{eq: TransOpPhi}) and using (\ref{eq:TransOpGxx})
we obtain
\begin{equation}
\phi(\rho,x)=\cos(\rho x)+\frac{\sin(\rho x)}{\rho}(h+\omega(x))-\frac{1}{\rho}\int_{0}^{x}\boldsymbol{\mathcal{\mathcal{G}}}_{2}(x,t)\sin(\rho t)dt.\label{eq:AuxG}
\end{equation}

Integration by parts applied to the integral in (\ref{eq:AuxG}) and
taking into account (\ref{eq:TransOpGxx}) and (\ref{eq:DerKernelG})
gives 
\begin{align}
\int_{0}^{x}\boldsymbol{\mathcal{G}}_{2}(x,t)\sin(\rho t)dt & =\left.-\mathcal{\boldsymbol{G}}_{2}(x,t)\frac{\cos(\rho t)}{\rho}\right|_{t=x}+\int_{0}^{x}\mathcal{\boldsymbol{G}}_{22}(x,t)\frac{\cos(\rho t)}{\rho}dt\nonumber \\
 & =-\frac{\cos(\rho x)}{4\rho}\left(q(x)-4h\omega(x)-2\omega^{2}(x)-q(0)\right)+\frac{1}{\rho}\int_{0}^{x}\mathcal{\boldsymbol{G}}_{22}(x,t)\cos(\rho t)dt.\label{eq:intS2-1}
\end{align}

Finally, substitution of (\ref{eq:intS2-1}) into (\ref{eq:AuxG})
leads to (\ref{eq:NewRepIntegralPhi}). 
\end{proof}

\section{Neumann series of Bessel functions representations\label{sec:NSBF-representations-for}}

In \cite{Vk2017NSBF} the following Neumann series of Bessel functions
(NSBF) representations for solutions $S\left(\rho,x\right)$ and $\phi(\rho,x)$
of (\ref{eq:PrincipalEq}) were obtained:

\begin{align}
S(\rho,x) & =\frac{\sin(\rho x)}{\rho}+\frac{1}{\rho}\sum_{n=0}^{\infty}(-1)^{n}\boldsymbol{s}_{n}(x)j_{2n+1}(\rho x),\label{eq:NsbfS}\\
\phi(\rho,x) & =\cos(\rho x)+\sum_{n=0}^{\infty}(-1)^{n}g_{n}(x)j_{2n}(\rho x),\label{eq:NsbfPhi}
\end{align}
where $j_{k}(z)$ stands for the spherical Bessel function of order
$k$ (see, e.g., \cite[p. 437]{Abramo}). The coefficients $\boldsymbol{s}_{n}(x)$
and $g_{n}(x)$ can be calculated following a simple recurrent integration
procedure (see \cite{Vk2017NSBF} or \cite[Sect. 9.4]{VK2020}). For
every $\rho\in\mathbb{C}$ the series converge pointwise, and for
every $x\in[0,L]$ the convergence is uniform on any compact set of
the complex plane of the variable $\rho$.

Since in this work we deal with several series representations for
solutions of (\ref{eq:PrincipalEq}), we need to introduce notations
for the corresponding partial sums. The partial sums of (\ref{eq:NsbfS})
and (\ref{eq:NsbfPhi}) are denoted by
\begin{align}
\hat{S}_{N}(\rho,x) & =\frac{\sin(\rho x)}{\rho}+\frac{1}{\rho}\sum_{n=0}^{N-1}(-1)^{n}\boldsymbol{s}_{n}(x)j_{2n+1}(\rho x),\label{eq:NsbfS-1}\\
\hat{\phi}_{N}(\rho,x) & =\cos(\rho x)+\sum_{n=0}^{N-1}(-1)^{n}g_{n}(x)j_{2n}(\rho x).\label{eq:NsbfPhi-1}
\end{align}

For any $\rho\in\mathbb{C\setminus}\left\{ 0\right\} $ belonging
to the strip $|\ii\rho|<C$, $C\geq0$ the inequalities hold \cite{Vk2017NSBF}
\[
\left|\rho S(\rho,x)-\rho\hat{S}_{N}(\rho,x)\right|\leq\frac{\varepsilon_{N}(x)\sinh\left(Cx\right)}{C}\,\,\text{and}\,\,\text{\ensuremath{\left|\phi(\rho,x)-\hat{\phi}_{N}(\rho,x)\right|\leq\frac{\varepsilon_{N}(x)\sinh\left(Cx\right)}{C}}.}
\]

Here and everywhere below the same symbol $\varepsilon_{N}(x)$ stands
for a positive function tending to zero when $N\rightarrow\infty$,
which is independent of $\rho$. Note that $S(0,x)=\hat{S}_{0}(0,x)$
and $\phi(0,x)=\hat{\phi}_{0}(0,x)$.

Thus, the remainders of the partial sums admit estimates, which are
independent of the real part of\textrm{ $\rho$. }This is extremely
useful when using the partial sums as approximate solutions for solving
direct or inverse problems on large ranges of the spectral parameter.

For smooth potentials ($q\in C^{1}\left[0,L\right]$) another NSBF
representation for $S(\pp,x)$ and \textrm{$\text{\ensuremath{\phi(\rho,x)}}$}
was obtained in \cite{VKComplex} with the aid of results from \cite{Vk2017}.

Hereinafter, denote 
\begin{align*}
q^{+}(x) & :=\frac{q(x)+q(0)}{4}-\frac{\omega^{2}(x)}{2},\,\,q^{h}(x):=\frac{q(x)-q(0)}{4}-\frac{\omega^{2}(x)}{2}-h\omega(x),\,\,\omega_{L}(x):=\frac{1}{2}\int_{x}^{L}q(s)ds,\\
 & q_{L}(x):=\frac{q(x)+q(L)}{4}-\frac{\omega_{L}^{2}(x)}{2},\,\,q_{L}^{H}(x):=\frac{q(x)-q(L)}{4}-\frac{\omega_{L}^{2}(x)}{2}-H\omega_{L}(x).
\end{align*}

\begin{thm}
Let $q\in C^{1}\left[0,L\right]$. The solutions $S(\pp,x)$ and $\phi(\rho,x)$
of equation (\ref{eq:PrincipalEq}) admit the following representations

\begin{align}
S\left(\rho,x\right) & =\frac{\sin(\rho x)}{\rho}-\frac{\cos(\rho x)}{\rho^{2}}\omega(x)+\frac{q^{+}(x)}{\rho^{3}}\left(\sin(\rho x)-3j_{1}(\rho x)\right)-\frac{1}{\rho^{3}}\sum_{n=1}^{\infty}(-1)^{n}\beta_{n}(x)j_{2n+1}(\rho x)\label{eq:NewSNSBF}
\end{align}

and 
\begin{align}
\phi(\rho,x) & =\cos(\rho x)+\frac{\sin(\rho x)}{\rho}(h+\omega(x))-\frac{xj_{1}(\rho x)}{\rho}q^{h}(x)-\frac{1}{\rho^{2}}\sum_{n=1}^{\infty}(-1)^{n}\alpha_{n}(x)j_{2n}(\rho x).\label{eq:NewPhiNSBF-1}
\end{align}

For every $\rho\in\mathbb{C}$ the series in (\ref{eq:NewSNSBF})
and (\ref{eq:NewPhiNSBF-1}) converge pointwise, and for every $x\in[0,L]$
the convergence is uniform on any compact set of the complex plane
of the variable $\rho$. Moreover, for any $\rho\in\mathbb{C\setminus}\left\{ 0\right\} $
belonging to the strip $|\ii\rho|<C$, the estimates hold

\begin{equation}
\left|S(\rho,x)-\hat{\mathbb{\mathtt{s}}}_{N}(\rho,x)\right|\leq\frac{\varepsilon_{N}(x)}{\left|\rho\right|^{3}}\sqrt{\frac{\sinh(2Cx)}{C}}\,\,\text{and}\,\,\text{\ensuremath{\left|\phi(\rho,x)-\hat{\Phi}_{N}(\rho,x)\right|\leq\frac{\varepsilon_{N}(x)}{\left|\rho\right|^{2}}\sqrt{\frac{\sinh(2Cx)}{C},}}}\label{eq:Estimates NSBF new}
\end{equation}
where 
\begin{align}
\hat{\mathbb{\mathtt{s}}}_{N}\left(\rho,x\right) & :=\frac{\sin(\rho x)}{\rho}-\frac{\cos(\rho x)}{\rho^{2}}\omega(x)+\frac{q^{+}(x)}{\rho^{3}}\left(\sin(\rho x)-3j_{1}(\rho x)\right)-\frac{1}{\rho^{3}}\sum_{n=1}^{N}(-1)^{n}\beta_{n}(x)j_{2n+1}(\rho x),\label{eq:PartialNewSNSBF}
\end{align}

\begin{align*}
\hat{\Phi}_{N}(\rho,x) & :=\cos(\rho x)+\frac{\sin(\rho x)}{\rho}(h+\omega(x))-\frac{xj_{1}(\rho x)}{\rho}q^{h}(x)-\frac{1}{\rho^{2}}\sum_{n=1}^{N}(-1)^{n}\alpha_{n}(x)j_{2n}(\rho x),
\end{align*}
and $\varepsilon_{N}(x)$ is a positive function tending to zero when
$N\rightarrow\infty$.
\end{thm}

We notice that the estimates (\ref{eq:Estimates NSBF new}) guarantee
an ever-improving accuracy of the approximation by the partial sums
for larger values of $\left|\text{Re}\rho\right|$.
\begin{rem}
Let $q\in C^{1}\left[0,L\right]$. The solution $\psi(\pp,x)$ of
equation (\ref{eq:PrincipalEq}) admits the NSBF representation

\begin{align}
\psi(\rho,x) & =\cos(\rho(L-x))+\frac{\sin(\rho(L-x))}{\rho}(\omega_{L}(x)+H)\nonumber \\
 & -\frac{(L-x)j_{1}(\rho(L-x))}{\rho}q_{L}^{H}(x)-\frac{1}{\rho^{2}}\sum_{n=1}^{\infty}(-1)^{n}\theta_{n}(x)j_{2n}(\rho(L-x)).\label{eq:NewNsbfPsi}
\end{align}
\end{rem}

The NSBF representations (\ref{eq:NewSNSBF}), (\ref{eq:NewPhiNSBF-1})
and (\ref{eq:NewNsbfPsi}) remain valid for potentials from $W_{2}^{1}\left[0,L\right]$. 

\section{Cardinal series representations for solutions\label{sec:Cardinal-series-representation}}

To obtain the cardinal series representations used in the present
work, we start with the Fourier series expansions of the transmutation
kernels $\mathcal{\boldsymbol{S}}(x,t)$ and $\mathcal{\boldsymbol{G}}(x,t)$.
\begin{prop}
\label{thm:The-transmutations-kernelExpan}The transmutation kernel
$\mathcal{\boldsymbol{S}}(x,t)$ admits the series representation
\begin{equation}
\mathcal{\boldsymbol{S}}(x,t)=\sum_{n=0}^{\infty}\frac{b_{n}(x)}{x}\sin\left(\frac{2n+1}{2x}\pi t\right),\label{eq:Series Kernel S}
\end{equation}
where for every $x\in(0,L]$ the series converges absolutely and uniformly
with respect to $t\in[0,x]$.
\end{prop}

\begin{proof}
For any $x\in[0,L]$ fixed, formula (\ref{eq:Series Kernel S}) is
nothing but a Fourier expansion in terms of the orthonormal basis
$\left\{ \sqrt{2/x}\sin\left(\frac{2n+1}{2x}\pi t\right)\right\} _{n=0}^{\infty}$
of $\mathcal{L}_{2}[0,x]$, see \cite[p. 80, Ex. 6]{Folland}. In
fact, this expansion can be seen as a classical Fourier sine series.
For this, extend $\mathcal{\boldsymbol{S}}(x,\cdot)$ onto $[0,2x]$
by making it symmetric with respect to the line $t=x$, that is, define
the extension $\tilde{\mathcal{\boldsymbol{S}}}$ as $\tilde{\mathcal{\boldsymbol{S}}}(x,t)=\tilde{\mathcal{\boldsymbol{S}}}(x,2x-t)=\boldsymbol{S}(x,t)$
for $t\in[0,x]$. Then consider the odd extension of $\boldsymbol{\tilde{S}}(x,t)\in AC[0,2x]$
onto $[-2x,2x]$. Thus, we obtain a function $\boldsymbol{\tilde{S}}(x,t)\in AC[-2x,2x]$.
Since $\boldsymbol{\tilde{S}}_{t}(x,t)$ is square integrable, its
Fourier series converges absolutely and uniformly, see \cite[Chap.1. §26]{Bary}.
\end{proof}
\begin{prop}
The transmutation kernel $\mathcal{\boldsymbol{G}}(x,t)$ admits the
series representation
\begin{equation}
\mathcal{\boldsymbol{G}}(x,t)=\sum_{n=0}^{\infty}\frac{\mathfrak{g}_{n}(x)}{x}\cos\left(\frac{n\pi t}{x}\right),\label{eq:Series Kernel G}
\end{equation}
where for every $x\in(0,L]$ the series converges absolutely and uniformly
with respect to $t\in[0,x]$.
\end{prop}

\begin{proof}
For any $x\in[0,L]$, $\mathcal{\boldsymbol{G}}(x,\cdot)\in AC[0,x]$
with $\mathcal{\boldsymbol{G}}_{t}(x,t)$ being square integrable.
Let us consider the even extension of $\mathcal{\boldsymbol{G}}(x,t)$
onto $[-x,x]$ and its cosine Fourier expansion in the orthonormal
basis $\left\{ \sqrt{2/x}\cos\left(\frac{n\pi t}{x}\right)\right\} _{n=0}^{\infty}$.
Since $\mathcal{\boldsymbol{G}}_{t}(x,t)$ is square integrable, its
Fourier series converges absolutely and uniformly, see \cite[Chap.1. §26]{Bary}.
\end{proof}
\begin{rem}
It is worth to mentioning the compatibility of the Fourier series
representation (\ref{eq:Series Kernel S}) with the condition (\ref{eq:TransOpSxx})
at the point $t=0$, i.e., direct substitution of $t=0$ into (\ref{eq:Series Kernel S})
gives us $\mathcal{\boldsymbol{S}}(x,0)=0$. Moreover, the condition
at $t=x$ from (\ref{eq:TransOpSxx}) implies
\[
\mathcal{\boldsymbol{S}}(x,x)=\sum_{n=0}^{\infty}(-1)^{n}\frac{b_{n}(x)}{x}=\omega(x).
\]

Analogously, 
\[
\mathcal{\boldsymbol{G}}(x,x)=\sum_{n=0}^{\infty}(-1)^{n}\frac{\mathfrak{g}_{n}(x)}{x}=h+\omega(x).
\]
\end{rem}

\subsection{\label{subsec:Cardinal-series-representation}Cardinal series representation
for solutions of (\ref{eq:PrincipalEq}) with $q\in\mathcal{L}_{2}[0,L]$}

Recall the Paley-Wiener class $PW_{x}^{2}=\left\{ f\in\mathcal{L}_{2}(\mathbb{R}):\exists g\in\mathcal{L}_{2}(-x,x),f(z)=\int_{-x}^{x}g(u)e^{iuz}du\right\} $
and its orthonormal basis $\left\{ \sqrt{\frac{x}{\pi}}\sinc\left(\frac{\rho x}{\pi}-n\right)\right\} _{n\in\mathbb{Z}}$.
The function $\sinc(z)$ stands for the normalized $\text{sinc}$
function defined as 
\[
\sinc(z):=\begin{cases}
\frac{\sin(\pi z)}{\pi z}, & z\neq0,\\
1, & z=0.
\end{cases}
\]
The Whittaker-Shannon-Kotel'nikov (W.S.K) sampling theorem asserts
that a function $f(\rho)$ belonging to the Paley-Wiener class $PW_{x}^{2}$
admits the series representation 
\begin{equation}
f(\rho)=\sum_{n\in\mathbb{Z}}f\left(\frac{n\pi}{x}\right)\sinc\left(\frac{\rho x}{\pi}-n\right).\label{eq:SincSamplingTheorem}
\end{equation}
 The series (\ref{eq:SincSamplingTheorem}) is known as the cardinal
series of $f(\rho)$ and converges uniformly on compact subsets of
$\mathbb{C}$. For $\rho$ on the real line, the convergence of this
series holds in the norm of $\mathcal{L}_{2}(\mathbb{R})$, absolutely
and uniformly, see \cite[p. 59]{Higgings}.
\begin{thm}
\label{thm:Let-.-The OldSIn}Let $q\in\mathcal{L}_{2}[0,L]$. The
solutions $S(\pp,x)$ and $\phi(\rho,x)$ of equation (\ref{eq:PrincipalEq})
admit the following series representations
\begin{equation}
S\left(\rho,x\right)=\frac{\sin(\pp x)}{\rho}+\frac{1}{2\rho\sin(\pp x)}\sum_{n=0}^{\infty}s_{n}(x)\left(\sinc\left(\frac{2\rho x}{\pi}-2n-1\right)+\sinc\left(\frac{2\rho x}{\pi}+2n+1\right)\right),\label{eq: SSeriesSinc}
\end{equation}
where
\begin{equation}
s_{n}(x)=\frac{(-1)^{n}(2n+1)\pi}{x}S\left(\frac{2n+1}{2x}\pi,x\right)-2,\,n=0,1,\ldots,\label{eq:Coeffs s}
\end{equation}
and
\begin{align}
\phi\left(\rho,x\right) & =-\rho xj_{1}(\rho x)+\sum_{n=0}^{\infty}\phi_{n}(x)\left(\sinc\left(\frac{\rho x}{\pi}-n\right)+\sinc\left(\frac{\rho x}{\pi}+n\right)\right),\label{eq: PhiSeriesSinc-1}
\end{align}
where
\begin{equation}
\phi_{0}(x)=\frac{\phi(0,x)}{2},\,\,\phi_{n}(x)=\phi\left(\frac{n\pi}{x},x\right)-(-1)^{n},\,n=1,2,\ldots.\label{eq:FormulasPsin-1}
\end{equation}
The series in (\ref{eq: SSeriesSinc}) and (\ref{eq: PhiSeriesSinc-1})
converge uniformly on any compact subset of the complex plane of the
variable $\rho$. For $\rho$ on the real line, the convergence of
(\ref{eq: SSeriesSinc}) and (\ref{eq: PhiSeriesSinc-1}) holds in
the norm of $\mathcal{L}_{2}(\mathbb{R})$, absolutely and uniformly.
\end{thm}

\begin{proof}
According to (\ref{eq: TransOpPhi}), for any fixed value of $x$,
the function 
\begin{equation}
\hat{\phi}(\rho,x):=\phi(\rho,x)-\cos(\pp x)\label{eq:AuxPhiFun}
\end{equation}
 belongs to the Paley-Wiener class $PW_{x}^{2}$. By the W.S.K theorem,
it is represented by the cardinal series
\begin{equation}
\hat{\phi}(\rho,x)=\sinc\left(\frac{\rho x}{\pi}\right)\hat{\phi}(0,x)+\sum_{n=1}^{\infty}\phi_{n}(x)\left(\sinc\left(\frac{\rho x}{\pi}-n\right)+\sinc\left(\frac{\rho x}{\pi}+n\right)\right)\cdot\label{eq:Phi tilde}
\end{equation}
 The coefficients $\phi_{n}(x)$, $n=1,2,\dots$ are obtained as the
values of the function $\hat{\phi}(\rho,x)$ at the sample points
$\rho=\frac{n\pi}{x},$ i.e., $\phi_{n}(x)=\phi\left(\frac{n\pi}{x},x\right)-(-1)^{n}.$ 

Substitution of (\ref{eq:AuxPhiFun}) into (\ref{eq:Phi tilde}) yields
into the representation
\[
\phi(\rho,x)=\cos(\rho x)-\sinc\left(\frac{\rho x}{\pi}\right)+\sum_{n=0}^{\infty}\phi_{n}(x)\left(\sinc\left(\frac{\rho x}{\pi}-n\right)+\sinc\left(\frac{\rho x}{\pi}+n\right)\right).
\]

Finally, we obtain (\ref{eq: PhiSeriesSinc-1}) by using the identity
\begin{equation}
zj_{1}(z)=\sinc\left(\frac{z}{\pi}\right)-\cos z,\,z\in\mathbb{C}.\label{eq:Besselj}
\end{equation}

For the solution $S(\rho,x)$, substitution of (\ref{eq:Series Kernel S})
into (\ref{eq:TransOpS}) leads to the equality
\[
S(\rho,x)=\frac{\sin(\rho x)}{\rho}+\frac{1}{\rho}\sum_{n=0}^{\infty}\frac{b_{n}(x)}{x}\int_{0}^{x}\sin\left(\frac{2n+1}{2x}t\pi\right)\sin(\rho t)dt.
\]

We have
\begin{equation}
\int_{0}^{x}\sin\left(\frac{2n+1}{2x}t\pi\right)\sin(\rho t)dt=(-1)^{n}\frac{\sinc\left(\frac{2\rho x}{\pi}-2n-1\right)+\sinc\left(\frac{2\rho x}{\pi}+2n+1\right)}{2\sin(\pp x)}x.\label{eq:FormulaOFIntegral}
\end{equation}

Thus,
\begin{equation}
S\left(\rho,x\right)=\frac{\sin(\pp x)}{\rho}+\frac{1}{2\rho\sin(\pp x)}\sum_{n=0}^{\infty}(-1)^{n}b_{n}(x)\left(\sinc\left(\frac{2\rho x}{\pi}-2n-1\right)+\sinc\left(\frac{2\rho x}{\pi}+2n+1\right)\right).\label{eq:CardinalS1}
\end{equation}

Considering $\rho=\frac{2m+1}{2x}\pi$, $m\in\mathbb{N}\cup\left\{ 0\right\} $
in (\ref{eq:CardinalS1}), we obtain 
\[
b_{m}(x)=\frac{(2m+1)\pi}{x}S\left(\frac{2m+1}{2x}\pi,x\right)-2(-1)^{m}.
\]

Thus, denoting $s_{n}(x)=(-1)^{n}b_{n}(x)$ the desired cardinal series
representation for the solution $S(\rho,x)$ is obtained. 

From the convergence principle (see \cite[p. 57]{Higgings}) and W.S.K
theorem, it follows that for $\rho$ on the real line and on any compact
subset of the complex plane, the convergence of (\ref{eq: SSeriesSinc})
and (\ref{eq: PhiSeriesSinc-1}) holds uniformly. Additionally, the
absolute convergence follows from Hölder's inequality, see Theorem
6.16 in \cite[p. 53]{Higgings}.
\end{proof}
\begin{rem}
The cardinal series representation (\ref{eq: PhiSeriesSinc-1}) can
be obtained also by using the Fourier series expansion (\ref{eq:Series Kernel G})
of the transmutation kernel $\mathcal{\boldsymbol{G}}(x,t)$, and
following a procedure analogous to the one presented in the case of
the cardinal series for $S(\rho,x)$, i.e., by substituting (\ref{eq:Series Kernel G})
into (\ref{eq: TransOpPhi}).
\end{rem}

\begin{rem}
In \cite{BoumenirChananeEigenvalues1996}, another cardinal series
representation for the function $S(\rho,L)$ is obtained by applying
the W.S.K theorem directly to the function $S(\rho,L)-\frac{\sin(\rho L)}{\rho}$.
It can also be derived by expanding the transmutation kernel into
the Fourier series
\[
\boldsymbol{S}(x,t)=\sum_{n=1}^{\infty}\frac{d_{n}(x)}{x}\sin\left(\frac{n\pi t}{x}\right)
\]
with respect to the orthogonal basis$\left\{ \sin\left(\frac{n\pi t}{x}\right)\right\} _{n=1}^{\infty}$
of $\mathcal{L}_{2}\left[0,x\right]$. However, the pointwise convergence
may fail in this case. Indeed, at the point $t=x$ the series is $0$,
which is not generally true for $\boldsymbol{S}(x,x)$, see (\ref{eq:TransOpSxx}).
\end{rem}

\begin{thm}
Let $q\in\mathcal{L}_{2}[0,L]$. The solutions $T(\pp,x)$ and $\psi(\rho,x)$
of equation (\ref{eq:PrincipalEq}) admit the following series representations
\begin{align}
T\left(\rho,x\right) & =\frac{\sin(\pp(x-L))}{\rho}\label{eq: TSeriesSinc}\\
 & +\frac{1}{2\rho\sin(\pp(x-L))}\sum_{n=0}^{\infty}t_{n}(x)\left(\sinc\left(\frac{2\rho(x-L)}{\pi}-2n-1\right)+\sinc\left(\frac{2\rho(x-L)}{\pi}+2n+1\right)\right),\nonumber 
\end{align}
where
\[
t_{n}(x)=\frac{(-1)^{n}(2n+1)\pi}{x-L}T\left(\frac{2n+1}{2(x-L)}\pi,x\right)-2,\,n=0,1,\ldots,
\]
and
\begin{align}
\psi\left(\rho,x\right) & =\rho(x-L)j_{1}(\rho(L-x))\label{eq: PsiSeriesSinc}\\
 & +\sum_{n=0}^{\infty}\psi_{n}(x)\left(\sinc\left(\frac{\rho(x-L)}{\pi}-n\right)+\sinc\left(\frac{\rho(x-L)}{\pi}+n\right)\right)\nonumber 
\end{align}
where
\begin{align}
\psi_{0}(x) & =\frac{\psi\left(0,x\right)}{2},\,\psi_{n}(x)=\psi\left(\frac{n\pi}{x-L},x\right)-(-1)^{n},\,n=1,2,\ldots.\label{eq:CoeffsPsi}
\end{align}
The series in (\ref{eq: TSeriesSinc}) and (\ref{eq: PsiSeriesSinc})
converge uniformly on any compact subset of the complex plane of the
variable $\rho$. For $\rho$ on the real line, the convergence holds
in the norm of $\mathcal{L}_{2}(\mathbb{R})$, absolutely and uniformly.
\end{thm}

\begin{proof}
A change of variable in the series representations (\ref{eq: SSeriesSinc})
and (\ref{eq: PhiSeriesSinc-1}) gives us the representations (\ref{eq: TSeriesSinc})
and (\ref{eq: PsiSeriesSinc}), respectively.
\end{proof}
\begin{rem}
Notice that from (\ref{eq:FormulasPsin-1}) and (\ref{eq:CoeffsPsi}),
we have
\begin{equation}
q(x)=\frac{\psi_{0}^{''}(x)}{\psi_{0}(x)}=\frac{\phi_{0}^{''}(x)}{\phi_{0}(x)}.\label{eq:Potential First Coeffs}
\end{equation}
\end{rem}

The following relations between the coefficients of the NSBF representations
and cardinal series representations hold.
\begin{rem}
Substitution of (\ref{eq:NsbfPhi}) into (\ref{eq:FormulasPsin-1})
results in
\begin{align*}
\phi_{0}(x) & =\frac{g_{0}(x)+1}{2},\\
\phi_{n}(x) & =\sum_{m=0}^{\infty}(-1)^{m}g_{m}(x)j_{2m}(n\pi),\,n=1,2,\ldots.
\end{align*}
\end{rem}

Analogously, substituting the NSBF series representation (\ref{eq:NsbfS})
into (\ref{eq:Coeffs s}) we obtain
\[
s_{m}(x)=\frac{2(-1)^{m}x}{(2m+1)\pi}\sum_{n=0}^{\infty}(-1)^{n}\boldsymbol{s}_{n}(x)j_{2n+1}\left(\frac{2m+1}{2}\pi\right),\,m=0,1,\ldots.
\]

\subsection{Representation for derivatives of solutions\label{subsec:Cardinal-series-representation-Derivatives} }
\begin{thm}
\label{thm:sincDerivatives}Let $q\in\mathcal{L}_{2}[0,L]$. The functions
$S'\left(\rho,x\right)$ and $\phi'\left(\rho,x\right)$ admit the
following series representations
\begin{align}
S'\left(\rho,x\right) & =\cos(\rho x)+\frac{\omega(x)}{\rho}\sin(\rho x)\label{eq: SSeriesSinc-1}\\
 & +\frac{1}{2\rho\sin(\pp x)}\sum_{n=0}^{\infty}\mathring{s}_{n}(x)\left(\sinc\left(\frac{2\rho x}{\pi}-2n-1\right)+\sinc\left(\frac{2\rho x}{\pi}+2n+1\right)\right),\nonumber 
\end{align}
where
\begin{equation}
\mathring{s}_{n}(x)=\frac{(-1)^{n}(2n+1)\pi}{x}\left.s'\left(\rho,x\right)\right|_{\rho=\frac{2n+1}{2x}\pi}-2\omega(x),\,n=0,1,\ldots,\label{eq:CoefsDErS}
\end{equation}
and
\begin{align}
\phi'\left(\rho,x\right) & =-\rho\sin(\rho x)-(h+\omega(x))\rho xj_{1}(\rho x)\label{eq: PhiSeriesSinc-1-1}\\
 & +\sum_{n=0}^{\infty}\mathring{\phi}_{n}(x)\left(\sinc\left(\frac{\rho x}{\pi}-n\right)+\sinc\left(\frac{\rho x}{\pi}+n\right)\right)\nonumber 
\end{align}
where
\[
\mathring{\phi}_{0}(x)=\frac{\phi'\left(0,x\right)}{2},\,\,\mathring{\phi}_{n}(x)=\left.\phi'\left(\rho,x\right)\right|_{\rho=\frac{n\pi}{x}}-(-1)^{n}(h+\omega(x)),\,\,n=1,2,\ldots.
\]
The series in (\ref{eq: SSeriesSinc-1}) and (\ref{eq: PhiSeriesSinc-1-1})
converge uniformly on any compact subset of the complex plane of the
variable $\rho$. For $\rho$ on the real line, the convergence holds
in the norm of $\mathcal{L}_{2}(\mathbb{R})$, absolutely and uniformly.
\end{thm}

\begin{proof}
The proof is analogous to the proof of Theorem \ref{thm:Let-.-The OldSIn}.
First, consider the kernel $\boldsymbol{\boldsymbol{\mathcal{S}_{1}}}(x,t)$
from (\ref{eq:TransOpS-1}), and its Fourier expansion in terms of
the orthogonal basis $\left\{ \sin\left(\frac{2n+1}{2x}t\pi\right)\right\} _{n=0}^{\infty}$
in $\mathcal{L}_{2}[0,x]$ as follows
\begin{equation}
\boldsymbol{\boldsymbol{\boldsymbol{\mathcal{S}_{1}}}}(x,t)=\sum_{n=0}^{\infty}\frac{\mathring{b}_{n}(x)}{x}\sin\left(\frac{2n+1}{2x}t\pi\right),\label{eq:KernelSRep-1}
\end{equation}

where for every $x\in(0,L]$ the series converges in $\mathcal{L}_{2}[0,x]$.
Substitution of (\ref{eq:KernelSRep-1}) into (\ref{eq:TransOpS-1})
and application of formula (\ref{eq:FormulaOFIntegral}) leads to
the cardinal series representation (\ref{eq: SSeriesSinc-1}) where
$\mathring{s}_{n}(x)=(-1)^{n}\mathring{b}_{n}(x)$. Here we change
the order of summation and integration due to Parseval's identity
\cite[p. 16]{Glazman}. The formula (\ref{eq:CoefsDErS}) for the
coefficients can be obtained applying the same procedure as for $S(\rho,x)$
in Theorem \ref{thm:Let-.-The OldSIn}.

Now, consider the function $F(\rho,x):=\phi'\left(\rho,x\right)+\rho\sin(\rho x)-\cos(\rho x)(h+\omega(x))$
which belongs to $PW_{x}^{2}$ with respect to the variable $\rho$,
because of (\ref{eq: TransOpPhi-1}). By the W.S.K theorem it admits
the cardinal series representation
\begin{equation}
F(\rho,x)=\sinc\left(\frac{\rho x}{\pi}\right)F(0,x)+\sum_{n=1}^{\infty}F_{n}(x)\left(\sinc\left(\frac{\rho x}{\pi}-n\right)+\sinc\left(\frac{\rho x}{\pi}+n\right)\right)\label{eq:Aux1}
\end{equation}

with the coefficients $F_{n}(x)$ obtained as the values of the function
$F(\rho,x)$ at the sample points $\rho=\frac{n\pi}{x},$ i.e.,
\[
F_{n}(x)=F\left(\frac{n\pi}{x},x\right)=\left.\phi'\left(\rho,x\right)\right|_{\rho=\frac{n\pi}{x}}-(-1)^{n}(h+\omega(x)),\,\,n=1,2,\ldots.
\]

Thus, by using (\ref{eq:Besselj}) in (\ref{eq:Aux1}) the cardinal
series representation (\ref{eq: PhiSeriesSinc-1-1}) is obtained with
the coefficients $\mathring{\phi}_{n}(x)=F_{n}(x)$. The results on
the convergence of the series in (\ref{eq: SSeriesSinc-1}) and (\ref{eq: PhiSeriesSinc-1-1})
are obtained following the same arguments from the proof of Theorem
\ref{thm:Let-.-The OldSIn}.
\end{proof}
\begin{thm}
Let $q\in\mathcal{L}_{2}\left[0,L\right]$. The functions $T'\left(\rho,x\right)$
and $\psi'\left(\rho,x\right)$ admit the following representations

\begin{align}
T'\left(\rho,x\right) & =\cos(\rho(L-x))+\frac{\sin(\rho(L-x))}{\rho}\omega_{L}(x)\label{eq: DTSeries}\\
 & +\frac{1}{2\rho\sin(\rho(L-x))}\sum_{n=0}^{\infty}\mathring{t}_{n}(x)\left(\sinc\left(\frac{2\rho(L-x)}{\pi}-2n-1\right)+\sinc\left(\frac{2\rho(L-x)}{\pi}+2n+1\right)\right)\nonumber 
\end{align}
~with 
\[
\mathring{t}_{n}(x)=\left.T'\left(\rho,x\right)\right|_{\rho=\frac{2n+1}{2(L-x)}\pi}-2(-1)^{n}\omega_{L}(x),\,n=0,1,\ldots,
\]
and
\begin{align}
\psi'\left(\rho,x\right) & =\rho\sin(\rho(L-x))+(H+\omega_{L}(x))\rho(L-x)j_{1}(\rho(L-x))\label{eq: DPsiSeries}\\
 & +\sum_{n=0}^{\infty}\mathring{\psi}_{n}(x)\left(\sinc\left(\frac{\rho(x-L)}{\pi}-n\right)+\sinc\left(\frac{\rho(x-L)}{\pi}+n\right)\right)\nonumber 
\end{align}
with 
\[
\mathring{\psi}_{0}(x)=\frac{\psi'\left(0,x\right)}{2},\,\,\mathring{\psi}_{n}(x)=\left.\psi'\left(\rho,x\right)\right|_{\rho=\frac{n\pi}{L-x}}+(-1)^{n}(H+\omega_{L}(x)),\,n=1,2,\ldots.
\]
The series in (\ref{eq: DTSeries}) and (\ref{eq: DPsiSeries}) converge
uniformly on any compact subset of the complex plane of the variable
$\rho$. For $\rho$ on the real line, the convergence holds in the
norm of $\mathcal{L}_{2}(\mathbb{R})$, absolutely and uniformly.
\end{thm}

\begin{proof}
The series representations are obtained from those of Theorem \ref{thm:sincDerivatives}
by flipping the interval.
\end{proof}
\begin{rem}
The obtained series representations for the solutions are not limited
to potentials from $\mathcal{L}_{2}\left[0,L\right]$. In \cite{Bondarenko},
\cite{Hryniv2003} and \cite{Hryniv2004}, definitions and construction
of transmutation operators for potentials in $W_{2}^{-1}\left[0,L\right]$
are presented. Similar series representations for the integral transmutation
kernels of the corresponding Volterra integral operators can be obtained
in this case as well, that leads to similar cardinal series representations
for the solutions. 
\end{rem}

\subsection{Cardinal series representation for solutions when $q\in W_{2}^{1}\left[0,L\right]$\label{subsec:NewCardinal-series-}}

Let $q\in W_{2}^{1}\left[0,L\right]$. We obtain another cardinal
series representation for the solutions $S(\pp,x)$, $T\left(\rho,x\right)$,
$\phi(\rho,x)$ and $\psi(\rho,x)$ of equation (\ref{eq:PrincipalEq})
with the aid of the integral representations from Theorem \ref{thm:qC1SolutionsInegralRepSPhi}.
First, consider the following Fourier expansion of the kernel $\mathcal{\boldsymbol{S}}_{22}(x,t)$. 
\begin{prop}
\label{thm:The-transmutations-kernelExpan-1}The kernel $\mathcal{\boldsymbol{S}}_{22}(x,t)$
from (\ref{eq:NewRepIntegralS}) admits the series representation
\begin{equation}
\mathcal{\boldsymbol{S}}_{22}(x,t)=\sum_{n=0}^{\infty}\frac{\tilde{b}_{n}(x)}{x}\sin\left(\frac{2n+1}{2x}\pi t\right),\label{eq:Series Kernel S-1}
\end{equation}
where for every $x\in[0,L]$ the series converges in $\mathcal{L}_{2}[0,x]$.
\end{prop}

\begin{proof}
For any $x\in[0,L]$, formula (\ref{eq:Series Kernel S-1}) is a Fourier
expansion of $\mathcal{\boldsymbol{S}}_{22}(x,\cdot)\in\mathcal{L}_{2}[0,x]$
in terms of the orthonormal basis $\left\{ \sqrt{2/x}\sin\left(\frac{2n+1}{2x}\pi t\right)\right\} _{n=0}^{\infty}$
of $\mathcal{L}_{2}[0,x]$. Hence the series (\ref{eq:Series Kernel S-1})
converges in $\mathcal{L}_{2}[0,x]$.
\end{proof}
\begin{thm}
\label{thm:Let-.-TheNewSinc}Let $q\in W_{2}^{1}\left[0,L\right]$.
The solutions $S(\pp,x)$ and $\phi(\rho,x)$ of equation (\ref{eq:PrincipalEq})
admit the following series representations
\begin{align}
S\left(\rho,x\right) & =\frac{\sin(\rho x)}{\rho}-\frac{\cos(\rho x)}{\rho^{2}}\omega(x)+\frac{\sin(\rho x)}{\rho^{3}}q^{+}(x)\label{eq:NewCardinalS}\\
 & -\frac{1}{2\rho^{3}\sin(\rho x)}\sum_{n=0}^{\infty}\mathfrak{s_{n}}(x)\left(\sinc\left(\frac{2\rho x}{\pi}-2n-1)\right)+\sinc\left(\frac{2\rho x}{\pi}+2n+1\right)\right),\nonumber 
\end{align}

with
\begin{equation}
\mathfrak{s_{n}}(x)=(-1)^{n+1}S\left(\frac{2n+1}{2x}\pi,x\right)\frac{((2n+1)\pi)^{3}}{4x^{3}}+2\left(\frac{2n+1}{2x}\pi\right)^{2}+2q^{+}(x),\,n=0,1,\ldots,\label{eq:NewCoeffs}
\end{equation}
and
\begin{align}
\phi(\rho,x) & =\cos(\rho x)+\frac{\sin(\rho x)}{\rho}(h+\omega(x))-\frac{x}{\rho}j_{1}(\rho x)q^{h}(x)\label{eq:NewCardinalPhi}\\
 & -\frac{1}{\rho^{2}}\sum_{n=1}^{\infty}\Phi_{n}(x)\left(\sinc\left(\frac{\rho x}{\pi}-n\right)+\sinc\left(\frac{\rho x}{\pi}+n\right)\right).\nonumber 
\end{align}
with
\[
\Phi_{n}(x)=\left((-1)^{n}-\phi\left(\frac{n\pi}{x},x\right)\right)\left(\frac{n\pi}{x}\right)^{2}+(-1)^{n}q^{h}(x),\,n=1,2,\ldots.
\]
The series in (\ref{eq:NewCardinalS}) and (\ref{eq:NewCardinalPhi})
converge uniformly on any compact subset of the complex plane of the
variable $\rho$. For $\rho$ on the real line, the convergence holds
in the norm of $\mathcal{L}_{2}(\mathbb{R})$, absolutely and uniformly.
\end{thm}

\begin{proof}
According to (\ref{eq:NewRepIntegralPhi}), 
\[
\hat{F}(\rho,x):=\cos(\rho x)\rho^{2}-\phi(\rho,x)\rho^{2}+\sin(\rho x)\rho(h+\omega(x))+\cos(\rho x)q^{h}(x)
\]
 belongs to $PW_{x}^{2}$ with respect to the variable $\rho$. Application
of the W.S.K theorem leads to the cardinal series expansion
\[
\hat{F}(\rho,x)=\sinc\left(\frac{\rho x}{\pi}\right)\hat{F}(0,x)+\sum_{n=1}^{\infty}\hat{F}\left(\frac{n\pi}{x},x\right)\left(\sinc\left(\frac{\rho x}{\pi}+n\right)+\sinc\left(\frac{\rho x}{\pi}-n\right)\right).
\]
By rewriting the series representation for $\hat{F}(\rho,x)$ in terms
of the solution $\phi(\rho,x)$ and noting that $\hat{F}(0,x)=q^{h}(x)$,
we obtain
\begin{align}
\phi(\rho,x) & =\cos(\rho x)+\frac{\sin(\rho x)}{\rho}(h+\omega(x))+\left(\frac{\cos(\rho x)}{\rho^{2}}-\frac{\sin(\rho x)}{\rho^{3}x}\right)q^{h}(x)\nonumber \\
 & -\frac{1}{\rho^{2}}\sum_{n=1}^{\infty}\Phi_{n}(x)\left(\sinc\left(\frac{\rho x}{\pi}-n\right)+\sinc\left(\frac{\rho x}{\pi}+n\right)\right).\label{eq:NwePhiAuxiliar=0000B7}
\end{align}

Finally, by using (\ref{eq:Besselj}) in (\ref{eq:NwePhiAuxiliar=0000B7})
we obtain (\ref{eq:NewCardinalPhi}). 

Analogous to the proof of the series representation (\ref{eq: SSeriesSinc}),
substitution of the Fourier series (\ref{eq:Series Kernel S-1}) into
(\ref{eq:NewRepIntegralS}) leads to (\ref{eq:NewCardinalS}), where
$\mathfrak{s_{n}}(x)=(-1)^{n}\tilde{b}_{n}(x)$.

Additionally, substitution of $\rho=\frac{2m+1}{2x}\pi$, $m\in\mathbb{N}\cup\left\{ 0\right\} $
in (\ref{eq:NewCardinalS}) leads to (\ref{eq:NewCoeffs}).

From the convergence principle (see \cite[p. 57]{Higgings}) and the
W.S.K theorem, it follows that for $\rho$ on the real line and on
any compact subset of the complex plane, the convergence of (\ref{eq:NewCardinalS})
and (\ref{eq:NewCardinalPhi}) holds uniformly. Additionally, the
absolute convergence follows from Hölder's inequality, see Theorem
6.16 in \cite[p. 53]{Higgings}.
\end{proof}
\begin{rem}
Notice that the potential can be written in terms of the coefficients
of the cardinal series representations (\ref{eq:NewCardinalS}) and
(\ref{eq:NewCardinalPhi})
\begin{equation}
q(x)=2(q^{+}(x)+q^{h}(x)+\omega^{2}(x)+h\omega(x)).\label{eq:op2-1}
\end{equation}
\end{rem}

\begin{thm}
Let $q\in W_{2}^{1}\left[0,L\right]$. The solutions $T(\pp,x)$ and
$\psi(\rho,x)$ of equation (\ref{eq:PrincipalEq}) admit the following
series representations{\small{}
\begin{align}
T\left(\rho,x\right) & =\frac{\sin(\rho(x-L))}{\rho}+\frac{\cos(\rho(x-L))}{\rho^{2}}\omega_{L}(x)+\frac{\sin(\rho(x-L))}{\rho^{3}}q_{L}(x)\label{eq:NewCardinalS-1}\\
 & -\frac{1}{2\rho^{3}\sin(\rho(x-L))}\sum_{n=0}^{\infty}\tilde{t}_{n}(x)\left(\sinc\left(\frac{2\rho(x-L)}{\pi}-2n-1)\right)+\sinc\left(\frac{2\rho(x-L)}{\pi}+2n+1\right)\right),\nonumber 
\end{align}
}where
\begin{align*}
\tilde{t}_{n}(x) & =(-1)^{n+1}T\left(\frac{2n+1}{2(x-L)}\pi,x\right)\frac{((2n+1)\pi)^{3}}{4(x-L)^{3}}+\frac{1}{2}\left(\frac{2n+1}{x-L}\pi\right)^{2}+2q_{L}(x),\,n=0,1,\ldots,
\end{align*}
and 
\begin{align}
\psi(\rho,x) & =\cos(\rho(L-x))+\frac{\sin(\rho(L-x))}{\rho}(H+\omega_{L}(x))-\frac{L-x}{\rho}j_{1}(\rho(L-x))q_{L}^{H}(x)\nonumber \\
 & -\frac{1}{\rho^{2}}\sum_{n=1}^{\infty}\tilde{\Psi}_{n}(x)\left(\sinc\left(\frac{\rho(x-L)}{\pi}-n\right)+\sinc\left(\frac{\rho(x-L)}{\pi}+n\right)\right),\label{eq:NewCardinalPsi}
\end{align}
where
\[
\tilde{\Psi}_{n}(x)=\psi\left(\frac{n\pi}{x-L},x\right)\left(\frac{n\pi}{x-L}\right)^{2}-(-1)^{n}\left(\frac{n\pi}{x-L}\right)^{2}-(-1)^{n}q_{L}^{H}(x),\,n=1,2,\ldots.
\]
The series in (\ref{eq:NewCardinalS-1}) and (\ref{eq:NewCardinalPsi})
converge uniformly on any compact subset of the complex plane of the
variable $\rho$. For $\rho$ on the real line, the convergence of
(\ref{eq:NewCardinalS-1}) and (\ref{eq:NewCardinalPsi}) holds in
the norm of $\mathcal{L}_{2}(\mathbb{R})$, absolutely and uniformly.
\end{thm}

\begin{proof}
A change of variable in the series representations (\ref{eq:NewCardinalS})
and (\ref{eq:NewCardinalPhi}) gives us (\ref{eq:NewCardinalS-1})
and (\ref{eq:NewCardinalPsi}), respectively.
\end{proof}
\begin{rem}
The cardinal series representation of $\psi(\rho,x)$ at $\rho=0$
(as well as the other cardinal series representations of Section \ref{sec:Cardinal-series-representation})
is obtained by taking the corresponding limits:
\begin{align*}
\psi(0,x) & =1+(L-x)(H+\omega_{L}(x))-\frac{(L-x)^{2}}{3}q_{L}^{H}(x)+\frac{2(L-x)^{2}}{\pi^{2}}\sum_{n=1}^{\infty}(-1)^{n}\frac{\tilde{\Psi}_{n}(x)}{n^{2}}.
\end{align*}
 
\end{rem}

\subsection{Remainder estimates\label{subsec:Error-estimates}}

This section presents remainder estimates of the cardinal series representations
for solutions $S(\rho,x)$, $\phi(\rho,x)$ and their derivatives.
Similar results are obtained for remainder estimates of the cardinal
series representations for solutions $T(\rho,x)$, $\psi(\rho,x)$
and their derivatives. 

Consider the partial sums of the representations (\ref{eq: SSeriesSinc}),
(\ref{eq: SSeriesSinc-1}), (\ref{eq: PhiSeriesSinc-1}) and (\ref{eq: PhiSeriesSinc-1-1})
{\small{}
\begin{align}
S_{N}\left(\rho,x\right) & =\frac{\sin(\pp x)}{\rho}+\frac{1}{2\rho\sin(\pp x)}\sum_{n=0}^{N-1}s_{n}(x)\left(\sinc\left(\frac{2\rho x}{\pi}-2n-1\right)+\sinc\left(\frac{2\rho x}{\pi}+2n+1\right)\right),\label{eq: TruncatedS}\\
\mathring{S}_{N}\left(\rho,x\right) & =\cos(\rho x)+\frac{\omega(x)}{\rho}\sin(\rho x)+\frac{1}{2\rho\sin(\pp x)}\sum_{n=0}^{N-1}\mathring{s}_{n}(x)\left(\sinc\left(\frac{2\rho x}{\pi}-2n-1\right)+\sinc\left(\frac{2\rho x}{\pi}+2n+1\right)\right),\nonumber \\
\varphi_{N}\left(\rho,x\right) & =-\rho xj_{1}(\rho x)+\sum_{n=0}^{N-1}\phi_{n}(x)\left(\sinc\left(\frac{\rho x}{\pi}-n\right)+\sinc\left(\frac{\rho x}{\pi}+n\right)\right)\nonumber 
\end{align}
}{\small\par}

and
\begin{align*}
\mathring{\varphi}_{N}\left(\rho,x\right) & =-\rho\sin(\rho x)-(h+\omega(x))\rho xj_{1}(\rho x)+\sum_{n=0}^{N-1}\mathring{\phi}_{n}(x)\left(\sinc\left(\frac{\rho x}{\pi}-n\right)+\sinc\left(\frac{\rho x}{\pi}+n\right)\right),
\end{align*}
 respectively.

For any $x\in(0,L]$ the following remainder estimates hold
\begin{equation}
\left|S(\rho,x)-S_{N}(\rho,x)\right|\leq\frac{\varepsilon_{N}(x)}{\left|\rho\right|},\,\,\text{\ensuremath{\left|\phi(\rho,x)-\varphi_{N}(\rho,x)\right|\leq\varepsilon_{N}(x)}},\label{eq:EStimateS}
\end{equation}

\begin{equation}
\left|S'(\rho,x)-\mathring{S}_{N}(\rho,x)\right|\leq\frac{\varepsilon_{N}(x)}{\left|\rho\right|}\,\,\text{and}\,\,\text{\ensuremath{\left|\phi'(\rho,x)-\mathring{\varphi}_{N}(\rho,x)\right|\leq\varepsilon_{N}(x)}}\label{eq:DErSEStimate}
\end{equation}

for all $\rho\in\mathbb{R},\rho\neq0,$ and
\begin{equation}
\left|S(\rho,x)-S_{N}(\rho,x)\right|\leq\frac{\varepsilon_{N}(x)}{\left|\rho\right|}\sqrt{\frac{\sinh\left(2Cx\right)}{C}},\,\,\text{\ensuremath{\left|\phi(\rho,x)-\varphi_{N}(\rho,x)\right|\leq\varepsilon_{N}(x)\sqrt{\frac{\sinh\left(2Cx\right)}{C}}}},\label{eq:EstimateS2}
\end{equation}

\begin{equation}
\left|S'(\rho,x)-\mathring{S}_{N}(\rho,x)\right|\leq\frac{\varepsilon_{N}(x)}{\left|\rho\right|}\sqrt{\frac{\sinh\left(2Cx\right)}{C}}\,\,\text{and}\,\,\text{\ensuremath{\left|\phi'(\rho,x)-\mathring{\varphi}_{N}(\rho,x)\right|\leq\varepsilon_{N}(x)\sqrt{\frac{\sinh\left(2Cx\right)}{C}}}}\label{eq:EstimateS2-1}
\end{equation}

for all $\rho\in\mathbb{C}\setminus\left\{ 0\right\} $ belonging
to the strip $|\ii\rho|\leq C$, $C\geq0$ where $\varepsilon_{N}(x)$
is a positive function tending to zero when $N\rightarrow\infty$,
i.e., the series in (\ref{eq: SSeriesSinc}) and (\ref{eq: PhiSeriesSinc-1})
converge uniformly respect to the variable $\rho$ in any strip $|\ii\rho|\leq C$,
$C\geq0$.

Denote the partial sums of (\ref{eq:Series Kernel S}) and (\ref{eq:Series Kernel G})
by
\[
\boldsymbol{S}_{N}(x,t)=\sum_{n=0}^{N}\frac{b_{n}(x)}{x}\sin\left(\frac{2n+1}{x}\pi t\right)\,\,\,\text{and}\,\,\,\boldsymbol{G}_{N}(x,t)=\sum_{n=0}^{N}\frac{\mathfrak{g}_{n}(x)}{x}\cos\left(\frac{n\pi t}{x}\right).
\]

The second estimate in (\ref{eq:EStimateS}) follows from {\small{}
\begin{equation}
\ensuremath{\left|\phi(\rho,x)-\varphi_{N}(\rho,x)\right|=\left|\int_{0}^{x}\left(\boldsymbol{G}(x,t)-\boldsymbol{G}_{N}(x,t)\right)\cos(\rho t)dt\right|\leq\varepsilon_{1,N}(x)\int_{0}^{x}\left|\cos(\rho t)\right|dt}\leq x\varepsilon_{1,N}(x)=:\varepsilon_{N}(x),\label{eq:proof real}
\end{equation}
}where $\varepsilon_{1,N}(x)=\left|\boldsymbol{G}(x,t)-\boldsymbol{G}_{N}(x,t)\right|$.

The first estimate in (\ref{eq:EstimateS2}) follows from the Cauchy--Bunyakovsky--Schwarz
inequality,

\begin{align*}
\left|S(\rho,x)-S_{N}(\rho,x)\right| & =\frac{1}{\left|\rho\right|}\left|\int_{0}^{x}\left(\boldsymbol{S}(x,t)-\boldsymbol{S}_{N}(x,t)\right)\sin(\rho t)dt\right|\\
 & \leq\frac{1}{\left|\rho\right|}\left\Vert \boldsymbol{S}(x,t)-\boldsymbol{S}_{N}(x,t)\right\Vert _{\mathcal{L}_{2}(0,x)}\left\Vert \sin(\rho t)\right\Vert _{\mathcal{L}_{2}(0,x)}.
\end{align*}

Since
\[
\left\Vert \sin(\rho t)\right\Vert _{\mathcal{L}_{2}(0,x)}^{2}\leq\frac{\left\Vert e^{i\rho t}\right\Vert _{\mathcal{L}_{2}(0,x)}^{2}+\left\Vert e^{-i\rho t}\right\Vert _{\mathcal{L}_{2}(0,x)}^{2}}{2}=\int_{0}^{x}\cosh(2\ii\rho t)dt=\frac{\sinh(2\ii\rho x)}{2\ii\rho},
\]

we obtain the inequality 
\begin{equation}
\left|S(\rho,x)-S_{N}(\rho,x)\right|\leq\frac{\varepsilon_{N}(x)}{\left|\rho\right|}\sqrt{\frac{\sinh\left(2Cx\right)}{C}}\label{eq:proof}
\end{equation}

with $\varepsilon_{N}(x):=\frac{\left\Vert \boldsymbol{S}(x,t)-\boldsymbol{S}_{N}(x,t)\right\Vert _{\mathcal{L}_{2}[0,x]}}{\sqrt{2}}$,
see Remark 9.1 in \cite{VK2020}. Similarly, the other estimates (\ref{eq:EStimateS})-(\ref{eq:EstimateS2-1})
are obtained.

Consider the partial sums of the representations (\ref{eq:NewCardinalS})
and (\ref{eq:NewCardinalPhi}) 
\begin{align}
\mathtt{s}_{N}\left(\rho,x\right) & =\frac{\sin(\rho x)}{\rho}-\frac{\cos(\rho x)}{\rho^{2}}\omega(x)+\frac{\sin(\rho x)}{\rho^{3}}q^{+}(x)\label{eq:TruncatedNewCardinalS}\\
 & -\frac{1}{2\rho^{3}\sin(\rho x)}\sum_{n=0}^{N-1}\mathfrak{s}_{n}(x)\left(\sinc\left(\frac{2\rho x}{\pi}-2n-1)\right)+\sinc\left(\frac{2\rho x}{\pi}+2n+1\right)\right)\nonumber 
\end{align}
and
\begin{align}
\Phi_{N}(\rho,x) & =\cos(\rho x)+\frac{\sin(\rho x)}{\rho}(h+\omega(x))-\frac{x}{\rho}j_{1}(\rho x)q^{h}(x)\label{eq:TruncatedNewCardinalPhi-1}\\
 & -\frac{1}{\rho^{2}}\sum_{n=1}^{N}\Phi_{n}(x)\left(\sinc\left(\frac{\rho x}{\pi}-n\right)+\sinc\left(\frac{\rho x}{\pi}+n\right)\right).\nonumber 
\end{align}

The following inequalities are valid
\[
\left|S(\rho,x)-\mathtt{s}_{N}(\rho,x)\right|\leq\frac{\varepsilon_{N}(x)}{\left|\rho\right|^{3}}\,\,\text{and}\,\,\text{\ensuremath{\left|\phi(\rho,x)-\Phi_{N}(\rho,x)\right|\leq\frac{\varepsilon_{N}(x)}{\left|\rho\right|^{2}}}},
\]

for all $\rho\in\mathbb{R},\rho\neq0,$ and
\[
\left|S(\rho,x)-\mathtt{s}_{N}(\rho,x)\right|\leq\frac{\varepsilon_{N}(x)}{\left|\rho\right|^{3}}\sqrt{\frac{\sinh(2Cx)}{C}}\,\,\text{and}\,\,\text{\ensuremath{\left|\phi(\rho,x)-\Phi_{N}(\rho,x)\right|\leq\frac{\varepsilon_{N}(x)}{\left|\rho\right|^{2}}\sqrt{\frac{\sinh(2Cx)}{C}}}}
\]

for all $\rho\in\mathbb{C\setminus}\left\{ 0\right\} $ belonging
to a strip $|\ii\rho|\leq C$, $C\geq0$ where $\varepsilon_{N}(x)$
is a positive function tending to zero when $N\rightarrow\infty$.
The proof of the estimate for $\rho\in\mathbb{R},\rho\neq0,$ is analogous
to the proof (\ref{eq:proof real}). For $\rho\in\mathbb{C\setminus}\left\{ 0\right\} $,
the proof of the estimate follows from the Cauchy--Bunyakovsky--Schwarz
inequality, analogous to the proof of (\ref{eq:proof}) (see Theorem
2.2 in \cite{VKComplex}).

\begin{rem}
\label{rem:The-error-analysis}The error analysis of cardinal series
has been extensively investigated. For instance, in \cite[p. 258]{Baumann}
the results from \cite[Section 3.5]{Butzer} were applied to describe
the amplitude error of cardinal series for functions from the class
$PW_{x}^{2}$, see also \cite{Annaby2008}, \cite{Jaggerman}, \cite{jERRY}
and references therein. Recall the amplitude error 
\[
\mathcal{A}_{\varepsilon}[f](\rho)=\sum_{n=-\infty}^{\infty}\left\{ f\left(\frac{n\pi}{x}\right)-\tilde{f}\left(\frac{n\pi}{x}\right)\right\} \sinc\left(\frac{\rho x}{\pi}-n\right),\,\rho\in\mathbb{C}.
\]

The values $\tilde{f}\left(\frac{n\pi}{x}\right)$ differ from the
exact coefficients $f\left(\frac{n\pi}{x}\right)$ by local errors
$\varepsilon_{n}:=f\left(\frac{n\pi}{x}\right)-\tilde{f}\left(\frac{n\pi}{x}\right)$,
such that $\left|\varepsilon_{n}\right|\leq\varepsilon$, for $n\in\mathbb{Z}$,
and there exists a constant $M_{f}>0$, such that 
\begin{equation}
\left|f(\rho)\right|\leq\frac{M_{f}}{\left|\rho\right|}\,\,\left(\left|\rho\right|\geq1\right),\rho\in\mathbb{R}.\label{eq:AssumtionErrorAmplitude}
\end{equation}

Then for $0<\varepsilon\leq\min\left\{ x/\pi,\pi/x,1/\sqrt{e}\right\} $
the estimate for the amplitude error is valid
\[
\left|\mathcal{A}_{\varepsilon}[f](\rho)\right|\leq4\left(\sqrt{3}e+\sqrt{2}M_{f}\exp(1/4)\right)\varepsilon\log(1/\varepsilon)\exp(x\left|\ii\rho\right|),\,\rho\in\mathbb{C}.
\]

Very often the coefficients of the cardinal series are not available
in the exact form, so the amplitude error estimates allow us to estimate
the error caused by working with noisy data.
\end{rem}

\begin{prop}
Consider the approximation 
\[
\widetilde{\phi}\left(\rho,x\right)=-\rho xj_{1}(\rho x)+\sum_{n=0}^{\infty}\tilde{\phi}_{n}(x)\left(\sinc\left(\frac{\rho x}{\pi}-n\right)+\sinc\left(\frac{\rho x}{\pi}+n\right)\right)
\]

of the cardinal series representation (\ref{eq: PhiSeriesSinc-1}).

Let $\tilde{\phi_{n}}\left(x\right)$ represent values differing from
the exact coefficients $\phi_{n}(x)$ by local errors $\varepsilon_{n}:=\tilde{\phi_{n}}\left(x\right)-\phi_{n}\left(x\right)$,
$n\in\mathbb{N}\cup\left\{ 0\right\} $, which are uniformly bounded,
$\left|\varepsilon_{n}\right|\leq\varepsilon$, $n\in\mathbb{N}\cup\left\{ 0\right\} $. 

Then 
\[
\left|\sum_{n=0}^{\infty}\left(\phi_{n}(x)-\tilde{\phi_{n}}\left(x\right)\right)\left(\sinc\left(\frac{\rho x}{\pi}-n\right)+\sinc\left(\frac{\rho x}{\pi}+n\right)\right)\right|\leq4\left(\sqrt{3}e+\sqrt{2}M_{\phi}e^{\frac{1}{4}}\right)\varepsilon\log\left(\frac{1}{\varepsilon}\right)e^{Cx}
\]
for $\left|\ii\rho\right|\leq C$, $C\geq0$, $0<\varepsilon\leq\min\left\{ x/\pi,\pi/x,1/\sqrt{e}\right\} $,
where $M_{\phi}>0$ is a constant.
\end{prop}

\begin{proof}
Since (see \cite[Lemma 1.1.2]{FreYurko})
\[
\phi(\rho,x)=\cos(\rho x)+O\left(\frac{1}{\left|\rho\right|}\exp\left(\left|\ii\rho\right|x\right)\right),\,\left|\rho\right|\rightarrow\infty,
\]

\textrm{the function $\phi(\rho,x)-\cos(\rho x)$} fulfills condition
(\ref{eq:AssumtionErrorAmplitude}) with $M_{\phi}=K\exp(Cx)$, where
$K$ is a positive constant. Thus, from Remark \ref{rem:The-error-analysis}
the result is obtained.
\end{proof}
An analogous result can be obtained for the cardinal series representation
of the solution $S(\rho,x)$ since the proof of the amplitude estimate
in \cite[Section 3.5]{Butzer} can be reproduced in this case as well. 

Additionally to the previous notations we denote the partial sums
of (\ref{eq:NewCardinalPsi}), (\ref{eq: PsiSeriesSinc}) and (\ref{eq: DPsiSeries})
by

\begin{align}
\psi_{N}\left(\rho,x\right) & =\cos(\rho(L-x))+\frac{\sin(\rho(L-x))}{\rho}(H+\omega_{L}(x))-\frac{L-x}{\rho}j_{1}(\rho(L-x))q_{L}^{H}(x)\nonumber \\
 & -\frac{1}{\rho^{2}}\sum_{n=1}^{N}\tilde{\Psi}_{n}(x)\left(\sinc\left(\frac{\rho(x-L)}{\pi}+n\right)+\sinc\left(\frac{\rho(x-L)}{\pi}-n\right)\right),\label{eq:PArtialPPSI}
\end{align}

\begin{equation}
\Psi_{N}\left(\rho,x\right)=\rho(x-L)j_{1}(\rho(L-x))+\sum_{n=0}^{N-1}\psi_{n}(x)\left(\sinc\left(\frac{\rho(x-L)}{\pi}-n\right)+\sinc\left(\frac{\rho(x-L)}{\pi}+n\right)\right)\label{eq:PartialPsiOldSinc}
\end{equation}

and 

\begin{align}
\mathring{\psi}_{N}\left(\rho,x\right) & =\rho\sin(\rho(L-x))+(H+\omega_{L}(x))\rho(L-x)j_{1}(\rho(L-x))\label{eq:PartialDpsi}\\
 & +\sum_{n=0}^{N-1}\mathring{\psi}_{n}(x)\left(\sinc\left(\frac{\rho(x-L)}{\pi}-n\right)+\sinc\left(\frac{\rho(x-L)}{\pi}+n\right)\right),\nonumber 
\end{align}
respectively.

\section{Spectrum completion\label{sec:Spectrum-completion}}

The approach to the spectrum completion problem is analogous to that
in \cite{VkSpectralCompl2023}, where Neumann series of Bessel functions
representations for solutions of (\ref{eq:PrincipalEq}) were used.
Given a finite set of the eigenvalues of a Sturm-Liouville problem,
the algorithm consists in computing several first coefficients of
a series representation for a characteristic function of the problem,
thus approximating the characteristic function and consequently locating
its zeros for computing further eigenvalues.

Namely, let us assume that several eigenvalues $\left\{ \xi_{k}^{2}\right\} _{k=1}^{N_{1}}$
of a Sturm-Liouville problem
\begin{equation}
\begin{array}{c}
-y''(x)+q(x)y(x)=\rho^{2}y(x)\\
y'(0)-hy(0)=0,\,y'(L)+Hy(L)=0
\end{array}\label{eq:SLProblemCompletion}
\end{equation}
are given, where $q$ is the unknown complex valued potential and
$h$ and $H$ are unknown complex constants. The solution $\psi(\rho,x)$
fulfills the second boundary condition of (\ref{eq:SLProblemCompletion}).
Thus, the characteristic function of problem (\ref{eq:SLProblemCompletion})
is defined by the first boundary condition 
\begin{equation}
\Delta(\lambda):=h\psi(\rho,0)-\psi^{'}(\rho,0),\label{eq:DeltaCharfun}
\end{equation}
where $\lambda=\rho^{2}$, see \cite[p. 6]{FreYurko}.

The spectrum completion can be performed following the algorithm.

1. Approximate the characteristic function $\Delta(\lambda)$ of problem
(\ref{eq:SLProblemCompletion}) by using the truncated cardinal series
representations (\ref{eq:PArtialPPSI}) (corresponding to potentials
$q(x)\in W_{2}^{1}\left[0,L\right]$) and (\ref{eq:PartialDpsi}),
i.e.,
\begin{align*}
\Delta_{N}(\lambda)=h\psi_{N}\left(\rho,0\right)-\mathring{\psi}_{N}\left(\rho,0\right) & =h\cos(\rho L)+\frac{h\sin(\rho L)}{\rho}\omega_{H}-\frac{hL}{\rho}j_{1}(\rho L)q_{L}^{H}(0)\\
 & -\frac{h}{\rho^{2}}\sum_{n=1}^{N}\tilde{\Psi}_{n}(0)\left(\sinc\left(\frac{\rho L}{\pi}+n\right)+\sinc\left(\frac{\rho L}{\pi}-n\right)\right)-\rho\sin(\rho L)\\
 & -\rho Lj_{1}(\rho L)\omega_{H}-\sum_{n=0}^{N-1}\mathring{\psi}_{n}(0)\left(\sinc\left(\frac{\rho L}{\pi}-n\right)+\sinc\left(\frac{\rho L}{\pi}+n\right)\right),
\end{align*}
where 
\begin{equation}
\omega:=\omega_{L}(0)=\frac{1}{2}\int_{0}^{L}q(s)ds,\,\,\omega_{H}(x):=\omega_{L}(x)+H\,\,\text{and}\,\,\,\,\omega_{H}=\omega_{H}(0)=\omega+H.\label{eq:omega}
\end{equation}

2. Consider a system of $N_{1}$ linear algebraic equations, which
is obtained from the $N_{1}$ equalities $\Delta_{N}(\lambda_{k})=0$
($\lambda_{k}=\xi_{k}^{2}$), 
\begin{align*}
 & h\cos(\xi_{k}L)+\frac{h\sin(\xi_{k}L)}{\xi_{k}}\omega_{H}-\frac{hL}{\xi_{k}}j_{1}(\xi_{k}L)q_{L}^{H}(0)\\
 & -\frac{h}{\xi_{k}^{2}}\sum_{n=1}^{N}\tilde{\Psi}_{n}(0)\left(\sinc\left(\frac{\xi_{k}L}{\pi}+n\right)+\sinc\left(\frac{\xi_{k}L}{\pi}-n\right)\right)-\xi_{k}Lj_{1}(\xi_{k}L)\omega_{H}\\
 & -\sum_{n=0}^{N-1}\mathring{\psi}_{n}(0)\left(\sinc\left(\frac{\xi_{k}L}{\pi}-n\right)+\sinc\left(\frac{\xi_{k}L}{\pi}+n\right)\right)=\xi_{k}\sin(\xi_{k}L),
\end{align*}
with $k=1,\ldots,N_{1},\,2N+4\leq N_{1}.$

Solving this system of equations, we obtain the coefficients 
\begin{equation}
\omega_{H},\,h,\,q_{L}^{H}(0),\,\left\{ \tilde{\Psi}_{n}(0)\right\} _{n=1}^{N}\,\,\,\text{and}\,\,\left\{ \mathring{\psi}_{n}(0)\right\} _{n=0}^{N-1}.\label{eq:CoeffsSpectralComNN}
\end{equation}

3. Consider the approximate characteristic function $h\psi_{N}\left(\rho,0\right)-\mathring{\psi}_{N}\left(\rho,0\right)$,
for any $\rho\in\mathbb{C}$. This is possible because the coefficients
(\ref{eq:CoeffsSpectralComNN}) found in the previous step are independent
of $\rho$.

4. Finally, compute the zeros of $h\psi_{N}\left(\rho,0\right)-\mathring{\psi}_{N}\left(\rho,0\right)$.
Their squares approximate the eigenvalues of (\ref{eq:SLProblemCompletion}).

Notice that for completing the spectrum no knowledge of $q$, $h$
and $H$ is required.
\begin{rem}
If it is known that $q\in\mathcal{L}_{2}[0,L]$ and not necessarily
$q\in W_{2}^{1}[0,L]$, the function $\psi(\rho,L)$ is approximated
by (\ref{eq:PartialPsiOldSinc}) (instead of (\ref{eq:PArtialPPSI}))
in step 1 of the spectrum completion algorithm. In this case, the
system of $N_{1}$ linear algebraic equations in step 2 is derived
from the $N_{1}$ equalities: 
\begin{align*}
-h\xi_{k}Lj_{1}(\xi_{k}L)-\xi_{k}Lj_{1}(\xi_{k}L)\omega_{H}+\sum_{n=0}^{N-1}\sigma_{n}\left(\sinc\left(\frac{\xi_{k}L}{\pi}-n\right)+\sinc\left(\frac{\xi_{k}L}{\pi}+n\right)\right)=\xi_{k}\sin(\xi_{k}L),
\end{align*}
where $k=1,\ldots,N_{1},\,N+2\leq N_{1}$ and $\sigma_{n}:=h\psi_{n}(0)-\mathring{\psi}_{n}(0)$.
Solving this system of equations is sufficient to reconstruct $\Delta_{N}(\lambda)$
and proceed to step 4 of the spectrum completion algorithm. 
\end{rem}

\begin{rem}
\begin{singlespace}
\label{rem:In-the-caseDD}In the case of a Sturm-Liouville problem
involving a Dirichlet condition at $x=0$ or at $x=L$ the procedure
is analogous. For example, suppose there is given a finite set $\left\{ \lambda_{k}=\rho_{k}^{2}\right\} _{k=1}^{N_{1}}$
of eigenvalues of the Dirichlet-Dirichlet Sturm-Liouville problem,
i.e., equation (\ref{eq:PrincipalEq}) subject to the boundary conditions
\[
y(0)=y(L)=0.
\]
The square roots of the eigenvalues coincide with zeros of the function
$S(\rho,L)$. The linear system (analogous to that mentioned in step
2) of the spectrum completion algorithm can be constructed in four
different forms by considering one of the four truncated series representations
(\ref{eq:NsbfS-1}), (\ref{eq:PartialNewSNSBF}), (\ref{eq: TruncatedS})
and (\ref{eq:TruncatedNewCardinalS}). In particular, by considering
$\mathtt{s}_{N}\left(\rho_{k},L\right)=0$ the system is constructed
from the equations{\footnotesize{}
\begin{align}
 & \frac{\cos(\rho_{k}L)}{\rho_{k}^{2}}\omega-\frac{\sin(\rho_{k}L)}{\rho_{k}^{3}}q^{+}(L)+\frac{1}{2\rho_{k}^{3}\sin(\rho_{k}L)}\sum_{n=0}^{N-1}\mathfrak{s}_{n}(L)\left(\sinc\left(\frac{2\rho_{k}L}{\pi}-2n-1)\right)+\sinc\left(\frac{2\rho_{k}L}{\pi}+2n+1\right)\right)\nonumber \\
 & =\frac{\sin(\rho_{k}L)}{\rho_{k}},\,k=1,\ldots,N_{1}.\label{eq:SystemDD}
\end{align}
}{\footnotesize\par}
\end{singlespace}

Notice that the parameter $\omega$ is one of the components of the
vector solution.

Similarly, by considering $S_{N}(\rho_{k},L)=0$ the system is obtained
from the equations{\small{}
\[
\frac{1}{2\rho_{k}\sin(\rho_{k}L)}\sum_{n=0}^{N-1}s_{n}(L)\left(\sinc\left(\frac{2\rho_{k}L}{\pi}-2n-1\right)+\sinc\left(\frac{2\rho_{k}L}{\pi}+2n+1\right)\right)=-\frac{\sin(\rho_{k}L)}{\rho_{k}},\,k=1,\ldots,N_{1}.
\]
}The other two options $\mathtt{\hat{s}}_{N}\left(\rho_{k},L\right)=0$
or $\hat{S}_{N}(\rho_{k},L)=0$ are analogous (NSBF representations
are considered).

Finally, once the characteristic function $\mathtt{s}_{N}\left(\rho,L\right)$
(or $S_{N}(\rho,L)$, or $\mathtt{\hat{s}}_{N}\left(\rho,L\right)$,
or $\hat{S}_{N}(\rho,L)$) is obtained, compute its zeros. Their squares
provide approximations to the eigenvalues of the Dirichlet-Dirichlet
problem.
\end{rem}

Let us discuss the computation of the parameter $\omega$. For this,
recall that $\left\{ \rho_{k}\right\} _{k=1}^{\infty}$, the square
roots of the eigenvalues, satisfy the asymptotic relation 
\begin{equation}
\rho_{k}=\frac{k\pi}{L}+\frac{\omega}{k\pi}+\frac{\varsigma_{_{k}}}{k},\label{eq:Asymtptic DD}
\end{equation}
 where $\left\{ \varsigma_{_{k}}\right\} \in l_{2}$. 

A frequently used technique for computing $\omega$ consists in minimizing
the $l_{2}$-norm of the sequence 
\begin{equation}
c_{k}=\pi k\left(\rho_{k}-\frac{k\pi}{L}\right)-\omega,\,k=1,2,\ldots,\label{eq:DDAymo2}
\end{equation}
constructed from (\ref{eq:Asymtptic DD}), see, e.g., \cite{Ignatiev},
\cite{Rundell}. In Section \ref{subsec:Dirichlet-Dirichlet-spectrum-com},
thus obtained approximation of $\omega$ is compared with the value
obtained as a part of the solution of system (\ref{eq:SystemDD}). 
\begin{thm}
For any $\varepsilon>0$ there exists such $N\in\mathbb{N}$ that
all zeros of the function $S(\rho,L)$ are approximated by corresponding
zeros of the function $S_{N}(\rho,L)$ with errors uniformly bounded
by $\varepsilon$, and $S_{N}(\rho,L)$ has no other zeros. 
\end{thm}

\begin{proof}
The proof is analogous to the proof of Proposition 7.1 in \cite{Kravchenko2016}
where the Rouché theorem is applied to the functions $\rho S(\rho,L)$
and $\rho S_{N}(\rho,L)$.
\end{proof}
\begin{rem}
An analogous statement can be proved for the characteristic function
$\Delta(\lambda)$ (see (\ref{eq:DeltaCharfun})) or the characteristic
function associated to the problem of equation (\ref{eq:PrincipalEq})
with boundary conditions$\begin{array}{c}
y(0)=0,\end{array}$ $y'(L)+Hy(L)=0$. Properties of these characteristic functions can
be found in \cite[§1.3]{Marchenko} (see also \cite[§1.1]{FreYurko}).
\end{rem}

We present illustrative examples performed in Matlab2021a. In the
final step of the spectrum completion algorithm, when locating complex
zeros of the approximate characteristic function, the algorithm of
the argument principle theorem described in \cite{Vk14} is applied.
When $q(x)$, $h$ and $H$ are real, to locate real zeros, the 'fnzeros'
routine in Matlab is used, applying it to a 6th order spline interpolation
of the approximate characteristic function. 

\subsection{Dirichlet-Dirichlet spectrum completion\label{subsec:Dirichlet-Dirichlet-spectrum-com}}

The Dirichlet-Dirichlet spectrum completion is performed following
Remark \ref{rem:In-the-caseDD}. The first two examples of this section
deal with the spectrum completion by using the four truncated series
representations (\ref{eq:NsbfS-1}), (\ref{eq:PartialNewSNSBF}),
(\ref{eq: TruncatedS}) and (\ref{eq:TruncatedNewCardinalS}) in the
case of real-valued potentials. In the last example, a complex Dirichlet-Dirichlet
spectrum is considered, and the spectrum completion algorithm is performed
by using the truncated cardinal series representations (\ref{eq: TruncatedS})
and (\ref{eq:TruncatedNewCardinalS}).

\begin{example} \label{ExamplePAine}Consider\textbf{ }the potential
$q(x)=\exp(x)$, $x\in[0,\pi]$. We compute the input data (a set
of the first eigenvalues) in a high precision arithmetics in Wolfram
Mathematica12 in order to study different situations under the presence
of a controlled noise. In Figures \ref{fig:ExpSquareNoise} and \ref{fig:DDUndeterminedNoiseC1Pot-1},
we present the absolute errors of the Dirichlet-Dirichlet spectrum
completed by using the four different approximations of $S(\rho,\pi)$,
see Remark \ref{rem:In-the-caseDD}. 

Three cases are considered, corresponding to 15, 10 and 5 eigenvalues
given as input data, presented by three columns in Figure \ref{fig:ExpSquareNoise}.
In each case, the considered linear system of algebraic equations
is square. In the first column, series representations $\mathtt{\hat{s}}_{13}(\rho,\pi)$,
$\hat{S}_{15}(\rho,\pi)$, $\mathtt{s}_{13}(\rho,\pi)$ and $S_{15}(\rho,\pi)$,
each with 15 unknown coefficients, are used. In the second column,
series representations $\mathtt{\hat{s}}_{8}(\rho,\pi)$, $\hat{S}_{10}(\rho,\pi)$,
$\mathtt{s}_{8}(\rho,\pi)$ and $S_{10}(\rho,\pi)$, each with 10
unknown coefficients are considered. Finally, in the third column,
series representations $\mathtt{\hat{s}}_{3}(\rho,\pi)$, $\hat{S}_{5}(\rho,\pi)$,
$\mathtt{s}_{3}(\rho,\pi)$ and $S_{5}(\rho,\pi)$, each with 5 unknown
coefficients, are considered. Furthermore, in the first row of Figure
\ref{fig:ExpSquareNoise}, we present the absolute errors of the computed
zeros of the characteristic function $S(\rho,\pi)$ when exact input
data are used. The other two rows of the grid present the absolute
errors dealing with randomly noisy data, as indicated at the top of
each figure. In all the cases we complete the sequence of the first
300 eigenvalues.

For convenience we use the following notation. $\text{NSBF}_{1}$
indicates the use of expression (\ref{eq:NsbfS-1}), $\text{NSBF}_{2}$
refers to the use of (\ref{eq:PartialNewSNSBF}), $\text{Cardinal}_{1}$
indicates the use of (\ref{eq: TruncatedS}) and $\text{Cardinal}_{2}$
refers to the use of (\ref{eq:TruncatedNewCardinalS}). 
\begin{flushleft}
\begin{figure}[H]
\begin{raggedright}
\includegraphics[width=0.9\paperwidth,height=0.9\textwidth]{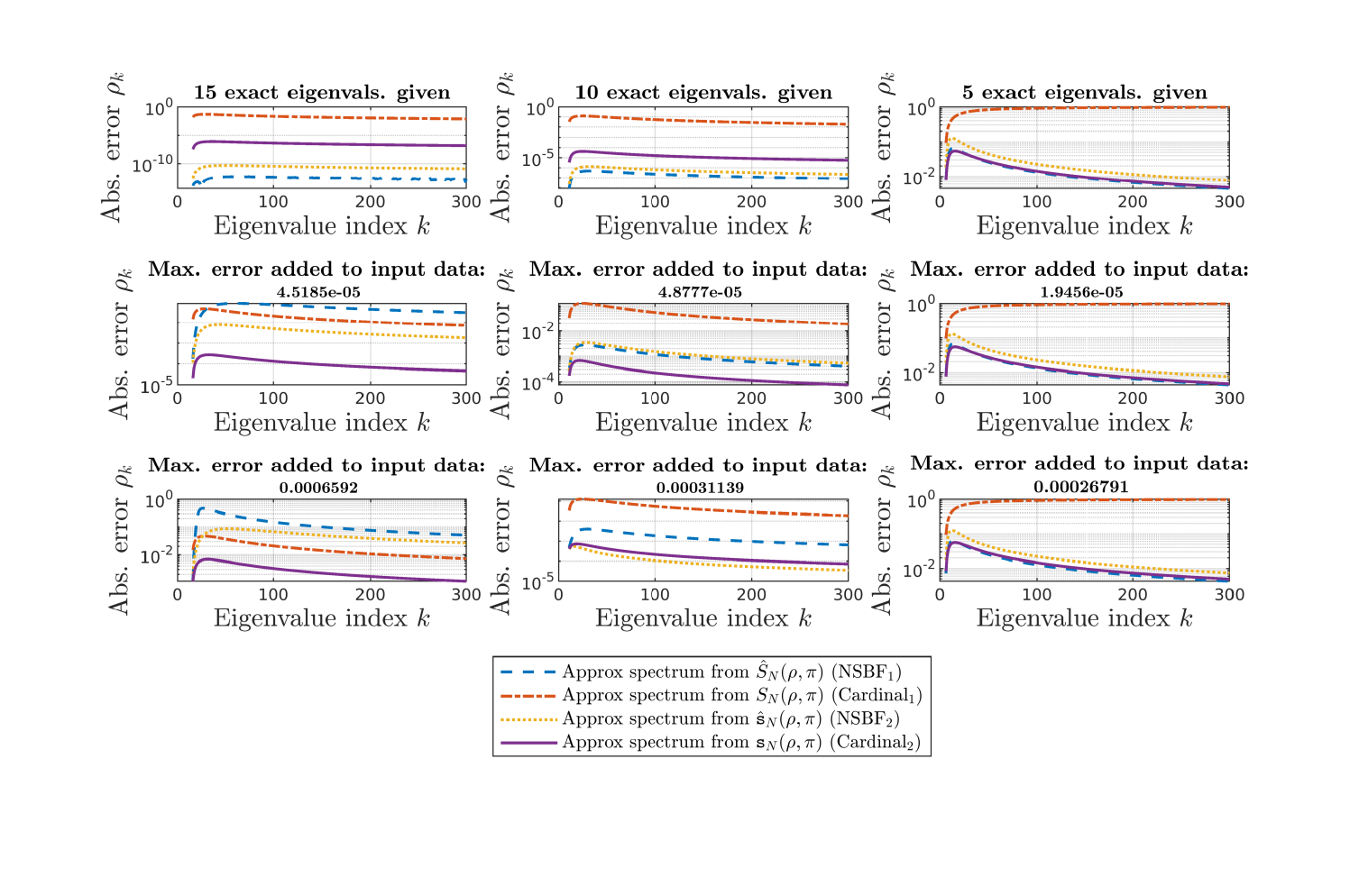}
\par\end{raggedright}
\caption{\label{fig:ExpSquareNoise}Example \ref{ExamplePAine}. Absolute errors
of the Dirichlet-Dirichlet spectrum completed from 15, 10 and 5 eigenvalues
(columns in the figure). In the top row the exact eigenvalues were
used as the input data, while in the other two rows the spectrum is
completed from randomly noisy data. The maximum noise level is indicated
on the top of each figure. }
\end{figure}
\par\end{flushleft}

The results in the first row of Figure \ref{fig:ExpSquareNoise} indicate
a better accuracy attained by both NSBF representations and by the
cardinal series representation (\ref{eq:TruncatedNewCardinalS}).
When the input data are not exact, the cardinal series representation
(\ref{eq:TruncatedNewCardinalS}) turns to be a preferable technique
(see the second and the third rows of figure \ref{fig:ExpSquareNoise}).

In Figure \ref{fig:DDUndeterminedNoiseC1Pot-1}, three numerical tests
are reported. Here in all tests 15 eigenvalues are given, however,
we deal with the partial sums of the series representations containing
less coefficients, thus solving overdetermined systems of linear algebraic
equations. The results presented in the first column correspond to
the partial sums $\mathtt{\hat{s}}_{11}(\rho,\pi)$, $\hat{S}_{13}(\rho,\pi)$,
$\mathtt{s}_{11}(\rho,\pi)$ and $S_{13}(\rho,\pi)$. For each of
these partial sums 13 unknown coefficients are computed from a $15\times13$
system. Analogously, the results of the second column correspond to
the partial sums $\mathtt{\hat{s}}_{8}(\rho,\pi)$, $\hat{S}_{10}(\rho,\pi)$,
$\mathtt{s}_{8}(\rho,\pi)$ and $S_{10}(\rho,\pi)$, while the results
of the third column correspond to $\mathtt{\hat{s}}_{4}(\rho,\pi)$,
$\hat{S}_{6}(\rho,\pi)$, $\mathtt{s}_{4}(\rho,\pi)$ and $S_{6}(\rho,\pi)$.
In the first row, the absolute errors presented correspond to exact
input data. The other two rows of the grid present the absolute errors
corresponding to noisy data as indicated at the top of each figure. 
\begin{flushleft}
\begin{figure}[H]
\begin{raggedright}
\includegraphics[width=0.9\paperwidth,height=0.9\textwidth]{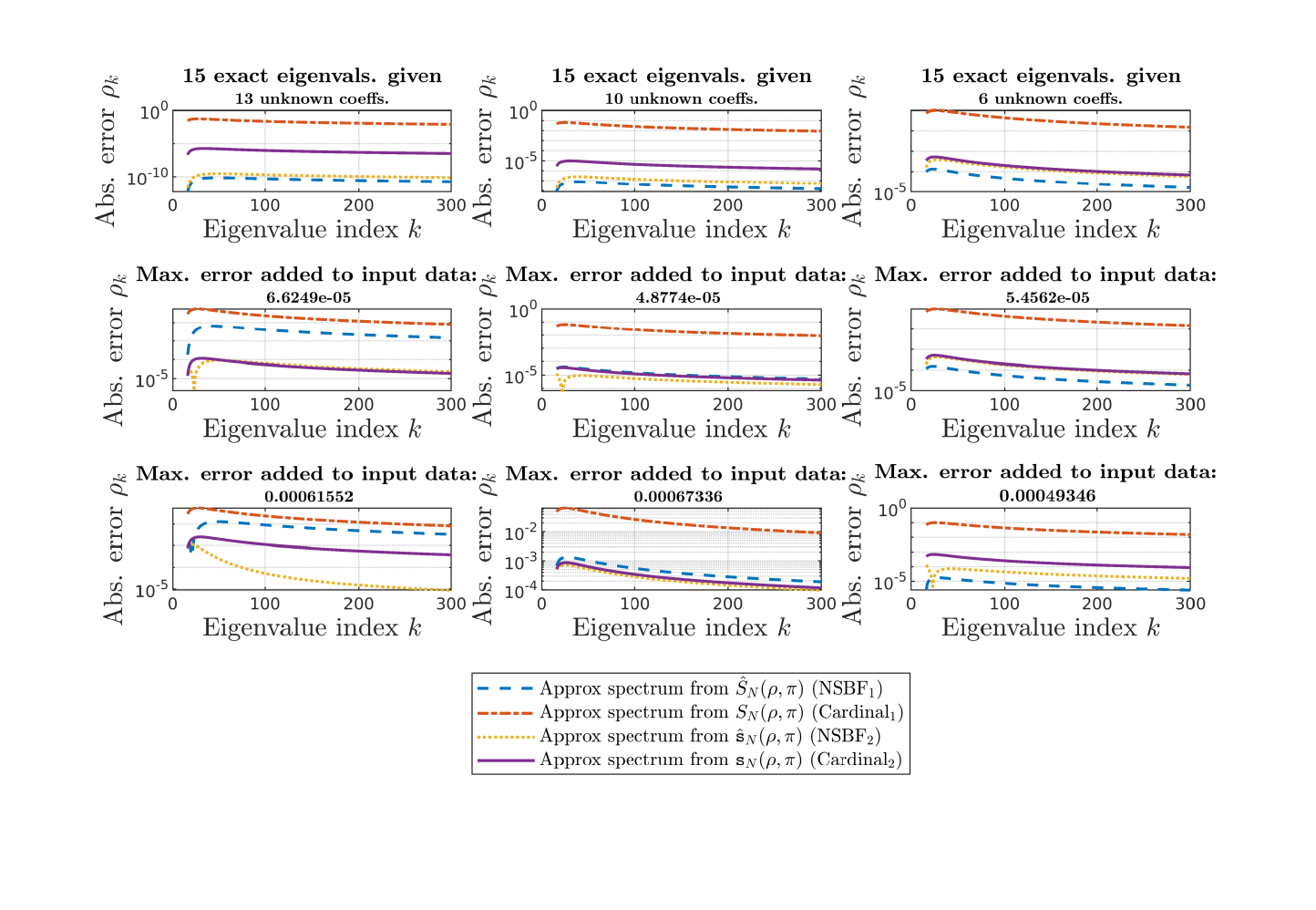}
\par\end{raggedright}
\caption{\label{fig:DDUndeterminedNoiseC1Pot-1}Example \ref{ExamplePAine}.
Absolute errors of the Dirichlet-Dirichlet spectrum completed from
given 15 eigenvalues and considering overdetermined systems in the
spectrum completion procedure. The maximum noise level is indicated
on the top of each figure. }
\end{figure}
\par\end{flushleft}

Figure \ref{fig:DDUndeterminedNoiseC1Pot-1} shows that both NSBF
representations and the cardinal series representation (\ref{eq:TruncatedNewCardinalS})
provide comparable accuracy when dealing with fewer coefficients and
noisy data.

Given 15 exact eigenvalues, the value of $\omega$ was obtained by
using $\mathtt{s}_{13}(\rho,\pi)$ and $\hat{\mathtt{s}}_{13}(\rho,\pi)$
as shown in Table \ref{tab:OmegaTableC1-1}.

\begin{table}[H]
\begin{centering}
\begin{tabular}{|c|c|c|c|}
\hline 
\multirow{2}{*}{} & Minimizing $l_{2}$-norm of & NSBF series & Cardinal series\tabularnewline
 & the sequence$\left\{ c_{k}\right\} _{k=1}^{15}$, see (\ref{eq:DDAymo2}) & $\hat{\mathtt{s}}_{13}(\rho,\pi)$ & $\mathtt{s}_{13}(\rho,\pi)$\tabularnewline
\hline 
\hline 
Abs. error $\omega$ & $8.34\times10^{-2}$ & $1.18\times10^{-8}$ & $1.50\times10^{-4}$\tabularnewline
\hline 
Rel. error $\omega$ & $7.53\times10^{-3}$ & $1.06\times10^{-9}$ & $1.35\times10^{-5}$\tabularnewline
\hline 
\end{tabular}
\par\end{centering}
\caption{\label{tab:OmegaTableC1-1}Absolute and relative errors of $\omega$
from Example \ref{ExamplePAine}.}
\end{table}

Note that the spectrum completion algorithm demonstrates a remarkable
accuracy in approximating the parameter $\omega$ as compared to the
minimization of the $l_{2}$-norm of $\left\{ c_{k}\right\} _{k=1}^{15}$,
see (\ref{eq:DDAymo2}).

\end{example}

\begin{example} \label{ExampleC1Pot}Consider\textbf{ }the potential
$q(x)=\left|x-1\right|(x-1)$, $x\in[0,\pi]$. The ``exact'' eigenvalues
were computed by the Matslise package \cite{Ledoux}. In Figures \ref{fig:DDSquareNoiseC1Pot}
and \ref{fig:DDUndeterminedNoiseC1Pot}, the absolute error of the
Dirichlet-Dirichlet spectrum completed by using the four different
approximations of $S(\rho,\pi)$ is presented. 

Three cases are considered, corresponding to 15, 10 and 5 eigenvalues
given as input data, presented by three columns in Figure \ref{fig:DDSquareNoiseC1Pot}.
In each case, the considered linear system of algebraic equations
is square. In the first column, series representations $\mathtt{\hat{s}}_{13}(\rho,\pi)$,
$\hat{S}_{15}(\rho,\pi)$, $\mathtt{s}_{13}(\rho,\pi)$ and $S_{15}(\rho,\pi)$,
each with 15 unknown coefficients, are used. In the second column,
series representations $\mathtt{\hat{s}}_{8}(\rho,\pi)$, $\hat{S}_{10}(\rho,\pi)$,
$\mathtt{s}_{8}(\rho,\pi)$ and $S_{10}(\rho,\pi)$, each with 10
unknown coefficients, are considered. Finally, in the third column,
series representations $\mathtt{\hat{s}}_{3}(\rho,\pi)$, $\hat{S}_{5}(\rho,\pi)$,
$\mathtt{s}_{3}(\rho,\pi)$ and $S_{5}(\rho,\pi)$, each with 5 unknown
coefficients, are considered. Furthermore, in the first row of Figure
\ref{fig:DDSquareNoiseC1Pot}, we present the absolute errors of the
computed zeros of the characteristic function $S(\rho,\pi)$ when
exact input data are used. The other two rows of the grid present
the absolute errors dealing with randomly noisy data, as indicated
at the top of each figure. In all the cases we complete the sequence
of the first 300 eigenvalues.
\begin{flushleft}
\begin{figure}[H]
\begin{raggedright}
\includegraphics[width=0.9\paperwidth,height=0.9\textwidth]{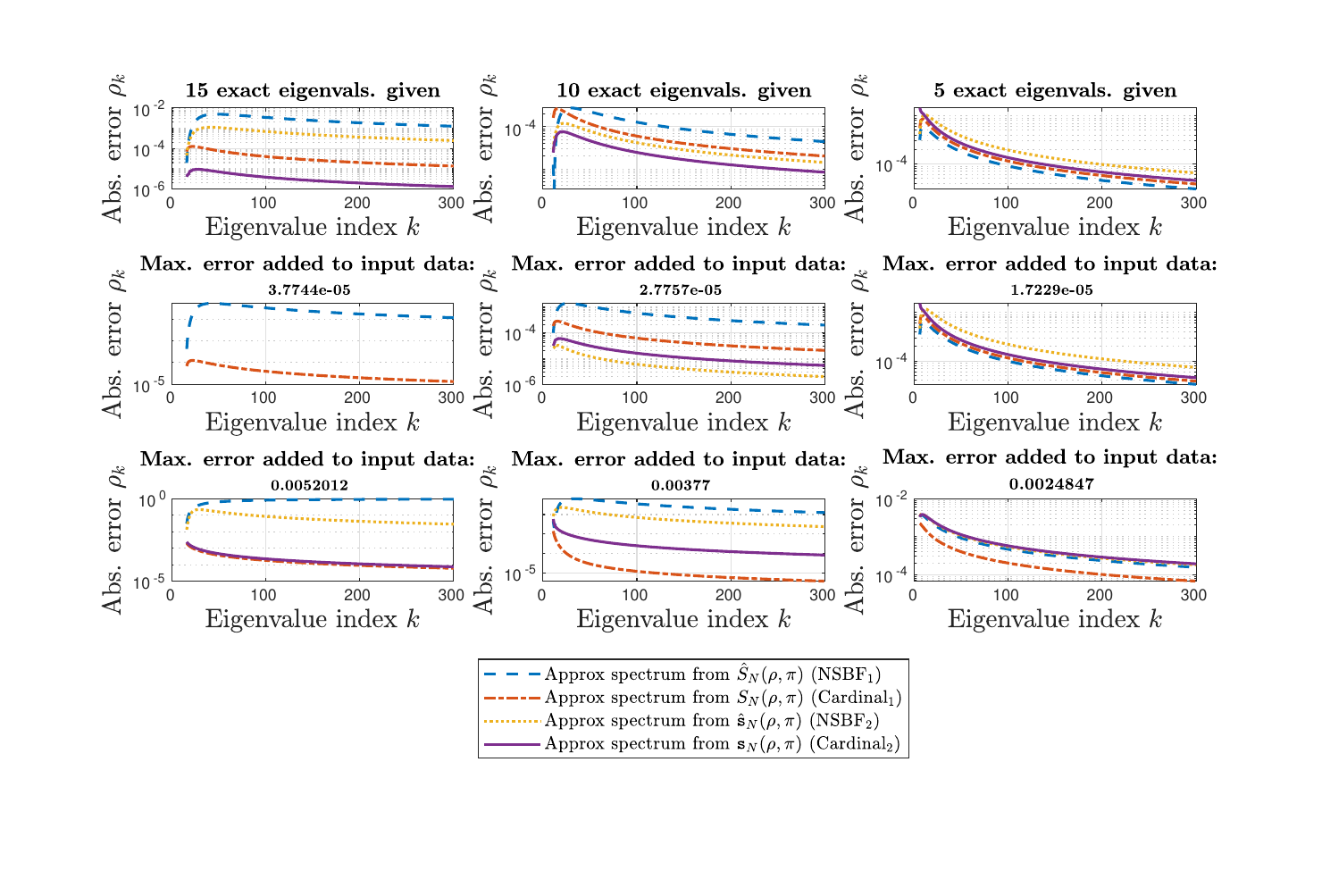}
\par\end{raggedright}
\caption{\label{fig:DDSquareNoiseC1Pot}Example \ref{ExampleC1Pot}. Absolute
errors of the Dirichlet-Dirichlet spectrum completed from given 15,
10 and 5 eigenvalues (columns in the figure). In the top row the exact
eigenvalues were used as the input data, while in the other two rows
the spectrum is completed from randomly noisy data. The maximum noise
level is indicated on the top of each figure.}
\end{figure}
\par\end{flushleft}

Figure \ref{fig:DDSquareNoiseC1Pot} shows that both the NSBF and
cardinal series representations provide comparable accuracy when dealing
with fewer given eigenvalues. However, from given 15 eigenvalues (first
column), better results are attained by using the cardinal series
representations. Note that in the second row of the first column,
the results corresponding to NSBF$_{2}$ and Cardinal$_{2}$ are omitted
because, in both cases, in the sequence of the roots of the eigenvalues
completed, one term was not found by the argument principle theorem
algorithm.

In Figure \ref{fig:DDUndeterminedNoiseC1Pot}, three numerical tests
are reported. Here in all tests 15 eigenvalues are given, however,
we deal with the partial sums of the series representations containing
less coefficients, thus solving overdetermined systems of linear algebraic
equations. The results presented in the first column correspond to
the partial sums $\mathtt{\hat{s}}_{11}(\rho,\pi)$, $\hat{S}_{13}(\rho,\pi)$,
$\mathtt{s}_{11}(\rho,\pi)$ and $S_{13}(\rho,\pi)$. For each of
these partial sums 13 unknown coefficients are computed from a $15\times13$
system. Analogously, the results of the second column correspond to
the partial sums $\mathtt{\hat{s}}_{8}(\rho,\pi)$, $\hat{S}_{10}(\rho,\pi)$,
$\mathtt{s}_{8}(\rho,\pi)$ and $S_{10}(\rho,\pi)$, while the results
of the third column correspond to $\mathtt{\hat{s}}_{4}(\rho,\pi)$,
$\hat{S}_{6}(\rho,\pi)$, $\mathtt{s}_{4}(\rho,\pi)$ and $S_{6}(\rho,\pi)$.
In the first row, the absolute errors presented correspond to exact
input data. The other two rows of the grid present the absolute errors
corresponding to noisy data as indicated at the top of each figure. 
\begin{flushleft}
\begin{figure}[H]
\begin{raggedright}
\includegraphics[width=0.9\paperwidth,height=0.9\textwidth]{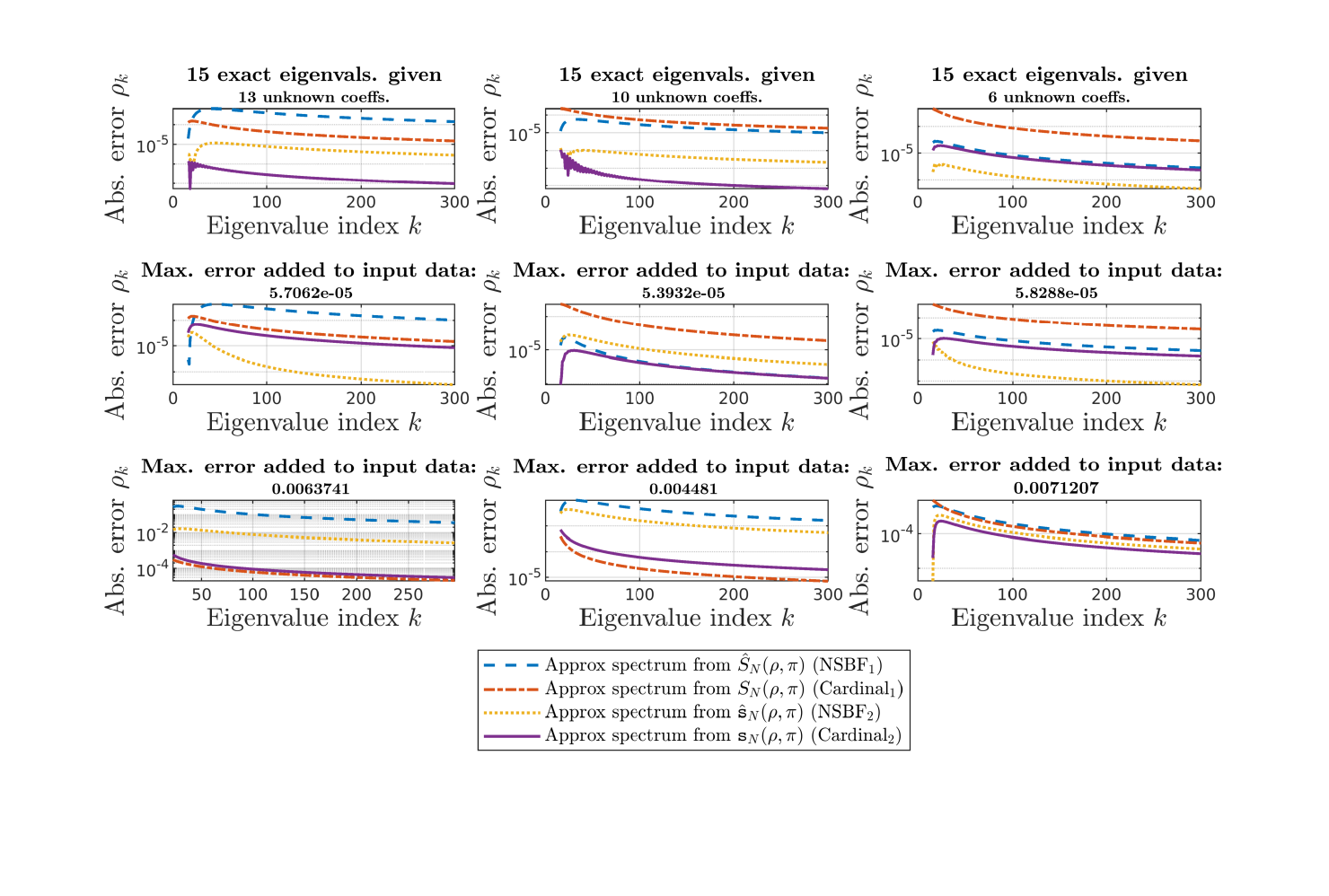}
\par\end{raggedright}
\caption{\label{fig:DDUndeterminedNoiseC1Pot}Example \ref{ExampleC1Pot}.
Absolute errors of the Dirichlet-Dirichlet spectrum completed given
15 eigenvalues and considering overdetermined systems in the spectrum
completion procedure. The maximum noise level is indicated on the
top of each figure. }
\end{figure}
\par\end{flushleft}

Figure \ref{fig:DDUndeterminedNoiseC1Pot} shows that both NSBF representations
and the cardinal series representations provide comparable accuracy
when dealing with fewer coefficients.

Given 15 exact eigenvalues, the value of $\omega$ was obtained by
using $\mathtt{s}_{13}(\rho,\pi)$ and $\hat{\mathtt{s}}_{13}(\rho,\pi)$
as shown in Table \ref{tab:OmegaTableC1}.

\begin{table}[H]
\begin{centering}
\begin{tabular}{|c|c|c|c|}
\hline 
\multirow{2}{*}{} & Minimizing $l_{2}$-norm of & NSBF series & Cardinal series\tabularnewline
 & the sequence$\left\{ c_{k}\right\} _{k=1}^{15}$ & $\hat{\mathtt{s}}_{13}(\rho,\pi)$ & $\mathtt{s}_{13}(\rho,\pi)$ \tabularnewline
\hline 
\hline 
Abs. error $\omega$ & $1.36\times10^{-3}$  & $2.67\times10^{-3}$  & $9.19\times10^{-5}$\tabularnewline
\hline 
Rel. error $\omega$ & $9.27\times10^{-4}$ & $1.82\times10^{-3}$  & $6.25\times10^{-5}$\tabularnewline
\hline 
\end{tabular}
\par\end{centering}
\caption{\label{tab:OmegaTableC1}Absolute and relative errors of $\omega$
from Example \ref{ExampleC1Pot}.}
\end{table}

The spectrum completion algorithm using NSBF representations attained
accuracy comparable with that of minimizing the $l_{2}$-norm of $\left\{ c_{k}\right\} _{k=1}^{15}$,
see (\ref{eq:DDAymo2}). However, in contrast to the previous example,
the use of cardinal representations attained a higher accuracy in
approximating $\omega$.

\end{example}
\begin{rem}
It is worth mentioning that for the numerical implementation of the
spectrum completion algorithm in the case of a Sturm-Liouville problem
with a real-valued potential, the input data comprise a set of several
first eigenvalues. Another choice of this set, e.g., excluding some
of the first eigenvalues, may affect the accuracy of the result. In
examples provided below, which involve sets of complex $\rho$ values,
the set $\left\{ \rho_{k}\right\} $ was arranged in ascending order
according to their real parts and again, another choice of the set
of the input data different from the lower eigenvalues ordered in
this sense, may considerably affect the accuracy of the spectrum completion
results. 
\end{rem}

\begin{example} \label{Example-2.-Razavy} Consider\textbf{ }the
potential {\small{}
\[
q(x)=\frac{\xi^{2}}{8}\cosh\left(4\left(x-\frac{\pi}{2}\right)\right)-\xi\cosh\left(2\left(x-\frac{\pi}{2}\right)\right)-\frac{\xi^{2}}{8}+i\left(-2\beta\text{\ensuremath{\cos}}\left(2\left(x-\frac{\pi}{2}\right)\right)+\beta^{2}\sin^{2}\left(2\left(x-\frac{\pi}{2}\right)\right)\right),
\]
}\textrm{where $x\in[0,\pi]$.} The real part of $q(x)$ is a Razavy
potential, see \cite{Razavy} (also \cite{Donj2}), while the imaginary
part is a Coffey--Evans potential, see \cite{Child}. Consider the
parameters $\xi=1$ and $\beta=2.5$. Then $95$ exact roots of the
eigenvalues of the Dirichlet-Dirichlet spectral problem were computed,
i.e., $\left\{ \rho_{k}\right\} _{k=1}^{95}$, applying the argument
principle algorithm to the truncated NSBF representation $\hat{S}_{70}(\rho,\pi)$,
see Table \ref{tab:EigenvalsRazavy}.

{\small{}}
\begin{table}[H]
\begin{centering}
{\small{}}%
\begin{tabular}{|c|c|c|c|}
\hline 
{\small{}$k$} & {\small{}$\rho_{k}$} & {\small{}$k$} & {\small{}$\rho_{k}$}\tabularnewline
\hline 
\hline 
{\small{}1} & {\small{}$3.12617926311455+0.572266443980943i$} & {\small{}16} & {\small{}$18.041095650974853+0.085877686528993i$}\tabularnewline
\hline 
{\small{}2} & {\small{}$4.088547511197311+0.389615377318618i$} & {\small{}17} & {\small{}$19.039065664759466+0.081442639937180i$}\tabularnewline
\hline 
{\small{}3} & {\small{}$5.084080945656062+0.302775401034403i$} & {\small{}18} & {\small{}$20.037218815706300+0.077440297693835i$}\tabularnewline
\hline 
{\small{}4} & {\small{}$6.085195457473403+0.250574546648726i$} & {\small{}19} & {\small{}$21.035532722327530+0.073810537153015i$}\tabularnewline
\hline 
{\small{}5} & {\small{}$7.082466993436702+0.214529415008048i$} & {\small{}20} & {\small{}$22.033988214663513+0.070503864631646i$}\tabularnewline
\hline 
{\small{}6} & {\small{}$8.077896364511830+0.188191144658540i$} & {\small{}21} & {\small{}$23.032568858538518+0.067479162553785i$}\tabularnewline
\hline 
{\small{}7} & {\small{}$9.072794574641620+0.167946782507620i$} & {\small{}22} & {\small{}$24.031260534710880+0.064701988966698i$}\tabularnewline
\hline 
{\small{}8} & {\small{}$10.067783652205982+0.151772831343148i$} & {\small{}23} & {\small{}$25.030051077609170+0.062143276372243i$}\tabularnewline
\hline 
{\small{}9} & {\small{}$11.063116297188756+0.138491798721120i$} & {\small{}24} & {\small{}$26.028929970169504+0.059778323971854i$}\tabularnewline
\hline 
{\small{}10} & {\small{}$12.058872589066544+0.127363162077596i$} & {\small{}25} & {\small{}$27.027888088230295+0.057586008621454i$}\tabularnewline
\hline 
{\small{}11} & {\small{}$13.055054839227326+0.117890683181366i$} & {\small{}26} & {\small{}$28.026917487279782+0.055548160905673i$}\tabularnewline
\hline 
{\small{}12} & {\small{}$14.051633124058206+0.109725199358610i$} & {\small{}27} & {\small{}$29.026011224748824+0.053649067311458i$}\tabularnewline
\hline 
{\small{}13} & {\small{}$15.048566694201764+0.102611765605969i$} & {\small{}28} & {\small{}$30.025163211839477+0.051875069711426i$}\tabularnewline
\hline 
{\small{}14} & {\small{}$16.045813628934138+0.096358838899685i$} & {\small{}29} & {\small{}$31.024368089770650+0.050214240659765i$}\tabularnewline
\hline 
{\small{}15} & {\small{}$17.043334874534832+0.090819220540106i$} & {\small{}30} & {\small{}$32.023621126166745+0.048656118273636i$}\tabularnewline
\hline 
\end{tabular}{\small\par}
\par\end{centering}
{\small{}\caption{\label{tab:EigenvalsRazavy}Thirty square roots $\left\{ \rho_{k}\right\} _{k=1}^{30}$
of eigenvalues of the Dirichlet-Dirichlet Sturm-Liouville problem
in Example \ref{Example-2.-Razavy}.}
}{\small\par}
\end{table}
{\small\par}

In Figure \ref{fig:RzavyCompletion}, we present the absolute and
relative errors of the values $\rho_{k}$ computed until $k=95$ (a
sequence of the first 95 eigenvalues were completed), from given $5$,
$15$, $25$ and $35$ eigenvalues. The same number of the unknowns
in the linear system for the coefficients of the approximate characteristic
function is considered. Table \ref{tab:Abs.-and-rel.Razavy} presents
the maximum absolute and relative errors of the sequence of the $\rho_{k}$
completed, by using both cardinal series representations (\ref{eq: TruncatedS})
and (\ref{eq:TruncatedNewCardinalS}).

\begin{figure}[H]
\begin{minipage}[c][1\totalheight][t]{0.45\textwidth}%
\begin{center}
\includegraphics[scale=0.55]{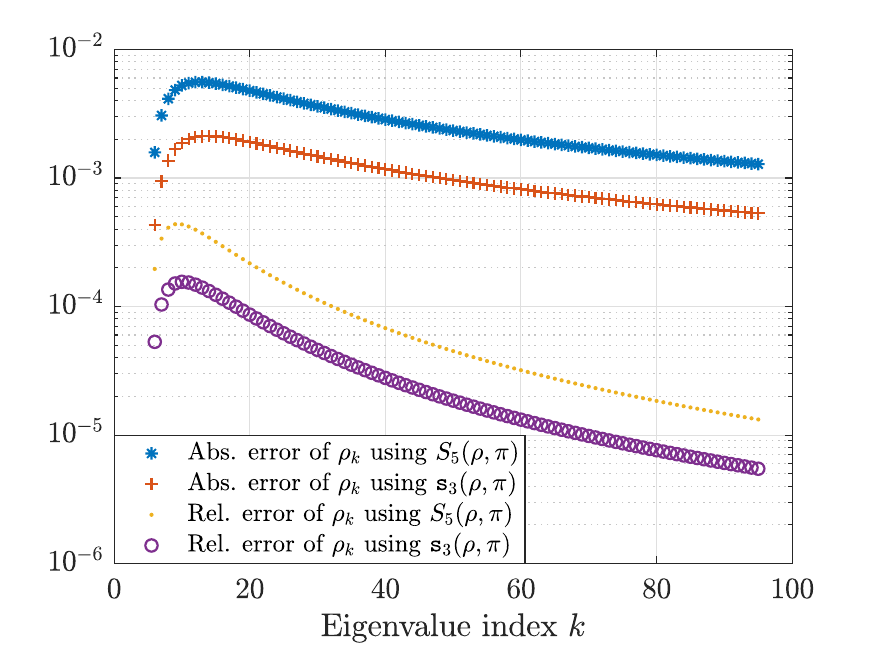}
\par\end{center}
\begin{center}
\textbf{(A)}
\par\end{center}%
\end{minipage}%
\begin{minipage}[c][1\totalheight][t]{0.45\textwidth}%
\begin{center}
\includegraphics[scale=0.55]{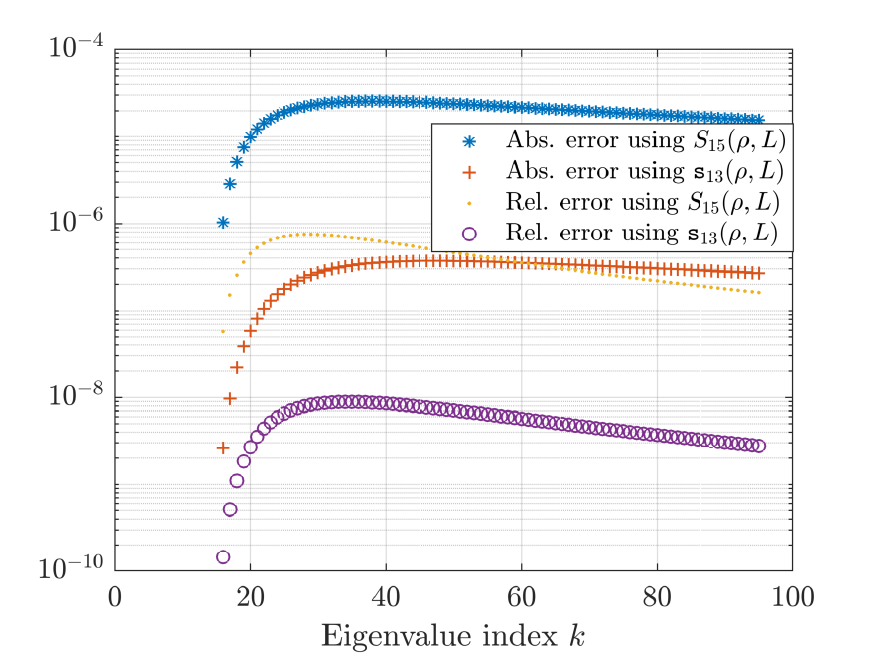}
\par\end{center}
\begin{center}
\textbf{(B)}
\par\end{center}%
\end{minipage}

\begin{minipage}[c][1\totalheight][t]{0.45\textwidth}%
\begin{center}
\includegraphics[scale=0.55]{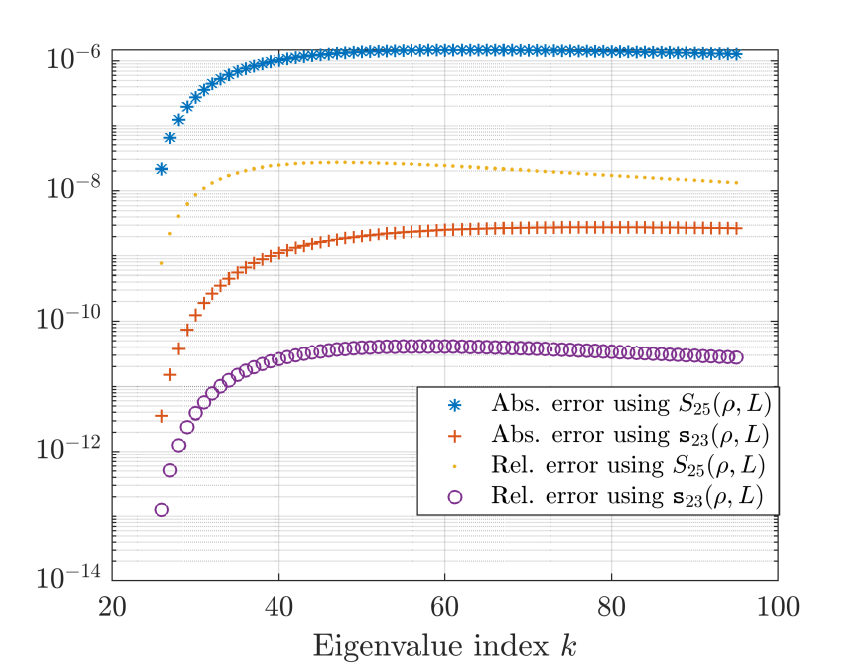}
\par\end{center}
\begin{center}
\textbf{(C)}
\par\end{center}%
\end{minipage}%
\begin{minipage}[c][1\totalheight][t]{0.45\textwidth}%
\begin{center}
\includegraphics[scale=0.55]{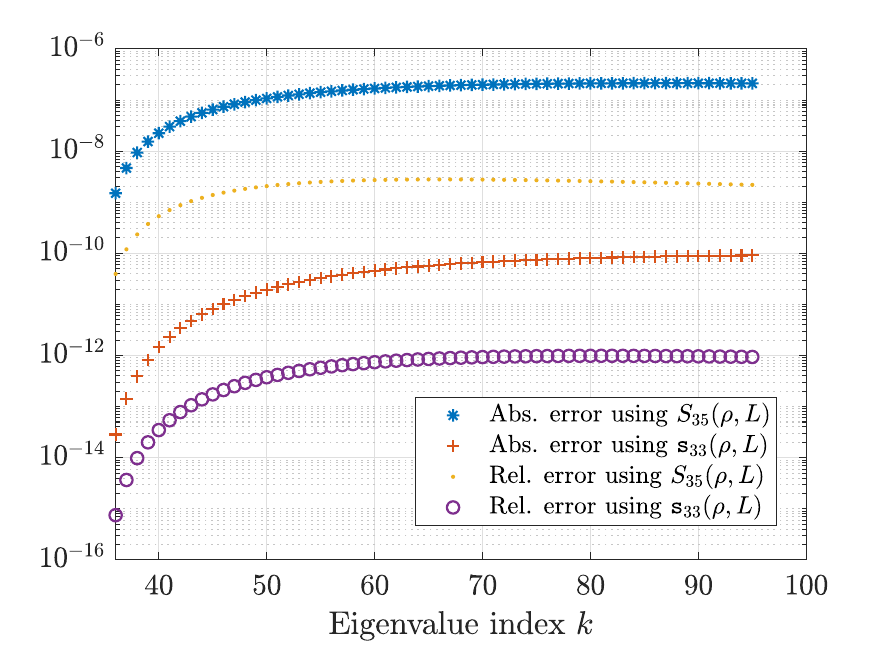}
\par\end{center}
\begin{center}
\textbf{(D)}
\par\end{center}%
\end{minipage}

\caption{\label{fig:RzavyCompletion}Example \ref{Example-2.-Razavy}. Abs.
and rel. errors of $\rho_{k}$ computed from 5 eigenvalues (in (A)),
from 15 eigenvalues (in (B)), from 25 eigenvalues (in (C)) and from
35 eigenvalues (in (D)).}
\end{figure}

\begin{table}[H]
\begin{centering}
{\footnotesize{}}%
\begin{tabular}{|c|c|c|c|c|}
\hline 
\multirow{2}{*}{} & {\footnotesize{}Given 5 eigenvals. and} & {\footnotesize{}Given 5 eigenvals. and} & {\footnotesize{}Given 15 eigenvals. and} & {\footnotesize{}Given 15 eigenvals. and}\tabularnewline
 & {\footnotesize{}computing $S_{5}(\rho,L)$} & {\footnotesize{}computing $\mathtt{s}_{3}(\rho,L)$} & {\footnotesize{}computing $S_{15}(\rho,L)$} & {\footnotesize{}computing $\mathtt{s}_{13}(\rho,L)$}\tabularnewline
\hline 
\hline 
{\footnotesize{}Abs. error } & {\footnotesize{}$5.57\times10^{-3}$} & {\footnotesize{}$2.12\times10^{-3}$} & {\footnotesize{}$2.58\times10^{-5}$} & {\footnotesize{}$3.72\times10^{-7}$}\tabularnewline
\hline 
{\footnotesize{}Rel. error } & {\footnotesize{}$4.37\times10^{-4}$} & {\footnotesize{}$1.56\times10^{-4}$} & {\footnotesize{}$7.5\times10^{-7}$} & {\footnotesize{}$8.95\times10^{-9}$}\tabularnewline
\hline 
\end{tabular}{\footnotesize\par}
\par\end{centering}
\begin{centering}
{\footnotesize{}}%
\begin{tabular}{|c|c|c|c|c|}
\hline 
\multirow{2}{*}{} & {\footnotesize{}Given 25 eigenvals. and} & {\footnotesize{}Given 25 eigenvals. and} & {\footnotesize{}Given 35 eigenvals. and} & {\footnotesize{}Given 35 eigenvals. and}\tabularnewline
 & {\footnotesize{}computing $S_{25}(\rho,L)$} & {\footnotesize{}computing $\mathtt{s}_{23}(\rho,L)$} & {\footnotesize{}computing $S_{35}(\rho,L)$} & {\footnotesize{}computing $\mathtt{s}_{33}(\rho,L)$}\tabularnewline
\hline 
\hline 
{\footnotesize{}Abs. error } & {\footnotesize{}$1.49\times10^{-6}$} & {\footnotesize{}$2.8\times10^{-9}$} & {\footnotesize{}$2.14\times10^{-7}$} & {\footnotesize{}$9.04\times10^{-11}$}\tabularnewline
\hline 
{\footnotesize{}Rel. error } & {\footnotesize{}$2.68\times10^{-8}$} & {\footnotesize{}$4.12\times10^{-11}$} & {\footnotesize{}$2.8\times10^{-9}$} & {\footnotesize{}$9.87\times10^{-13}$}\tabularnewline
\hline 
\end{tabular}{\footnotesize\par}
\par\end{centering}
{\footnotesize{}\caption{\label{tab:Abs.-and-rel.Razavy}Example \ref{Example-2.-Razavy}.
Maximum abs. and rel. errors of the values of $\rho_{k}$ computed
from 5, 15, 25 and 35 eigenvalues.}
}{\footnotesize\par}
\end{table}

The results show a remarkable convergence of the method even when
dealing with complex $\rho_{k}$. Moreover, satisfactory results were
attained with a reduced number of the eigenvalues given. Indeed, even
from 5 given eigenvalues the approximate $\omega$ was found with
an absolute error of $1.63\times10^{-1}$ and a relative error of
$2.99\times10^{-2}$.

\end{example} 

\subsection{Neumann-Dirichlet spectrum\label{subsec:Neumann-Dirichlet-spectrum}}

\begin{example}\textbf{ }\label{ExampleRotor}Consider\textbf{ }the
potential $q(x)=-b\cos(x)$, $x\in[0,2\pi]$ with the parameter $b=5$,
see \cite{Chen}. 

Paper \cite{Chen} is devoted to the computation of energy levels
and analytical wave functions for the Dirichlet-Dirichlet problem
of the rigid planar rotor in an electric field model. The authors
transform the Schrödinger equation into a confluent Heun differential
equation and subsequently compute the eigenvalues using the Maple
package. The accuracy of the algorithm relies on the numerical calculation
of confluent Heun functions and their first-order derivatives by Maple.
This procedure becomes complicated when a large number of eigenvalues
must be computed. Thus, the spectrum completion algorithm can be useful
to obtain more such data accurately. 

Consider the Neumann-Dirichlet problem 
\begin{equation}
\begin{array}{c}
-y''(x)-5\cos(x)y(x)=\rho^{2}y(x),\,0<x<2\pi\\
y'(0)=0,\,y(2\pi)=0.
\end{array}\label{eq:boundarycondExampl}
\end{equation}

Since the solution $\phi(\rho,x)$ satisfies the first initial condition
in (\ref{eq:boundarycondExampl}), the corresponding characteristic
function of the problem is approximated by $\varphi_{N}(\rho,2\pi)$
(or $\Phi_{N}(\rho,2\pi)$), see Section \ref{subsec:Error-estimates}.
Thus, given the eigenvalues $\left\{ \nu_{k}^{2}\right\} _{k=1}^{M}$
(computed by Matslise), the $M$-equations $\varphi_{N}(\nu_{k},2\pi)=0$
(or $\Phi_{N}(\nu_{k},2\pi)=0$) form the system of linear algebraic
equations for the coefficients of the approximate characteristic function.
In this example, we complete the sequence of the first 300 eigenvalues.

In Figures \ref{fig:Abs.-and-rel.DNCompeltion} and \ref{fig:DNCopletion2}
the absolute errors of $\nu_{k}$ computed from given 5 and 15 eigenvalues
are reported, respectively. In Table \ref{tab:Abs.-and-rel.Razavy-1}
the maximum values of the absolute and relative errors of the sequence
$\nu_{k}$ completed are presented.

\begin{figure}[H]
\begin{minipage}[c][1\totalheight][t]{0.45\textwidth}%
\begin{center}
\includegraphics[scale=0.6]{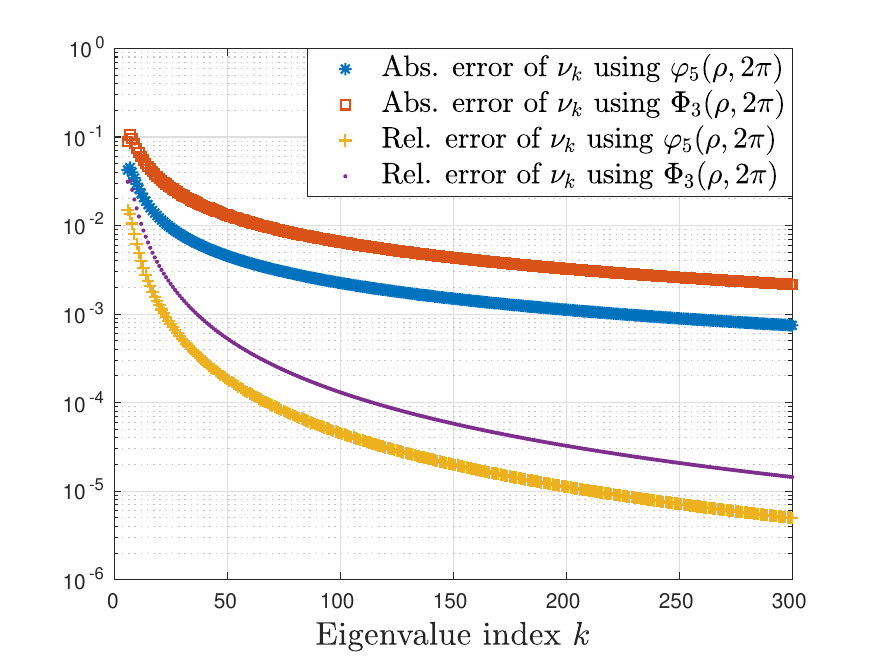}
\par\end{center}
\caption{\label{fig:Abs.-and-rel.DNCompeltion}Example \ref{ExampleRotor}.
Abs. and rel. errors of $\nu_{k}$, $k=6,\ldots,300$ computed from
5 eigenvalues.}
\end{minipage}\hspace{0.5cm}%
\begin{minipage}[c][1\totalheight][t]{0.45\textwidth}%
\begin{center}
\includegraphics[scale=0.6]{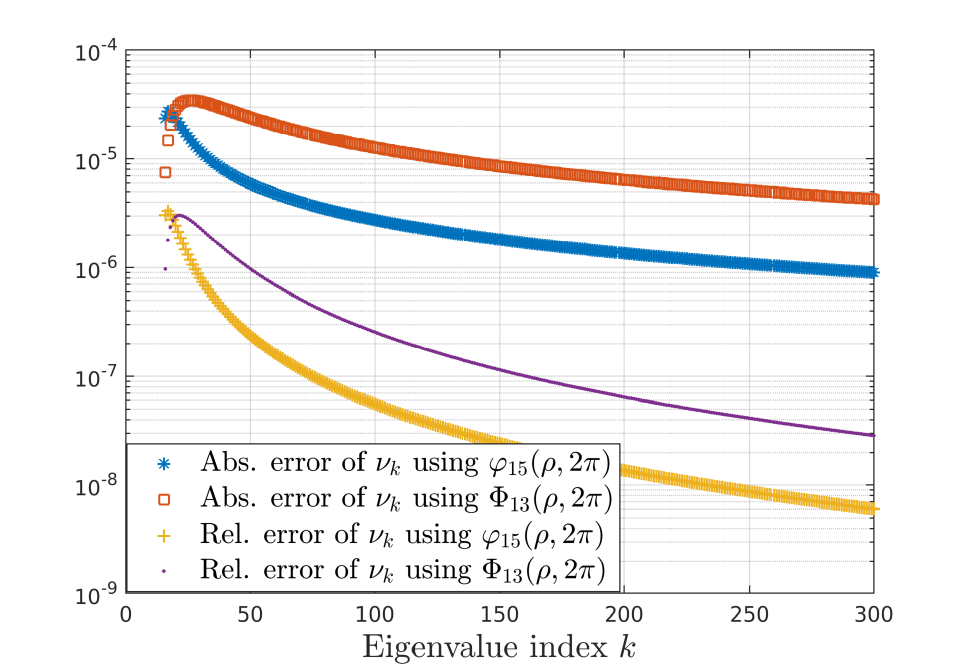}
\par\end{center}
\caption{\label{fig:DNCopletion2}Example \ref{ExampleRotor}. Abs. and rel.
errors of $\nu_{k}$, $k=16,\ldots,300$ computed from 15 eigenvalues.}
\end{minipage}
\end{figure}

\begin{table}[H]
\begin{centering}
{\footnotesize{}}%
\begin{tabular}{|c|c|c|c|c|}
\hline 
\multirow{2}{*}{} & {\footnotesize{}Given 5 eigenvals. and} & {\footnotesize{}Given 5 eigenvals. and} & {\footnotesize{}Given 15 eigenvals. and} & {\footnotesize{}Given 15 eigenvals. and}\tabularnewline
 & {\footnotesize{}computing $\varphi_{5}(\rho,2\pi)$} & {\footnotesize{}computing $\mathtt{\Phi}_{3}(\rho,2\pi)$} & {\footnotesize{}computing $\varphi_{15}(\rho,2\pi)$} & {\footnotesize{}computing $\mathtt{\Phi}_{13}(\rho,2\pi)$}\tabularnewline
\hline 
\hline 
{\footnotesize{}Abs. error } & {\footnotesize{}$4.5\times10^{-2}$} & {\footnotesize{}$1.05\times10^{-1}$} & {\footnotesize{}$2.76\times10^{-5}$} & {\footnotesize{}$3.49\times10^{-5}$}\tabularnewline
\hline 
{\footnotesize{}Rel. error } & {\footnotesize{}$1.37\times10^{-2}$} & {\footnotesize{}$3.56\times10^{-2}$} & {\footnotesize{}$7.5\times10^{-6}$} & {\footnotesize{}$5.64\times10^{-6}$}\tabularnewline
\hline 
\end{tabular}{\footnotesize\par}
\par\end{centering}
{\footnotesize{}\caption{\label{tab:Abs.-and-rel.Razavy-1}Example \ref{ExampleRotor}. Maximum
abs. and rel. errors of the values of $\nu_{k}$ computed from 5 and
15 eigenvalues.}
}{\footnotesize\par}
\end{table}

Additionally, the parameter $\omega$ was computed from given 15 eigenvalues
with an absolute error of $4.05\times10^{-3}$. 

Thus, the results of the spectrum completion are satisfactory even
when the Sturm-Liouville problem is considered on a larger interval
in $x$, and a limited number of the eigenvalues is provided. The
fast convergence of the method is noticeable.

\end{example} 

\subsection{Other spectra}

\begin{example}\label{-CompletiionH}Consider a Sturm-Liouville problem
with a Mathieu potential, see \cite{Mathieu}

\begin{equation}
\begin{array}{c}
-y''(x)+is\cos(2x)y(x)=\rho^{2}y(x),\,0<x<\pi,\\
y(0)=0,\,y'(\pi)+Hy(\pi)=0,
\end{array}\label{eq:SecondSLP-1}
\end{equation}

with $H$ assumed to be unknown. This potential with the parameter
$s$ being a real constant has been studied in numerous papers (see,
e.g., \cite{Brimacombe}, \cite[p. 44]{Flugge}, \cite{Mulholland }
and \cite{Ziener}). 

The solution $\psi(\rho,x)$ fulfills the second boundary condition
of (\ref{eq:SecondSLP-1}). Thus, the characteristic function of the
problem (\ref{eq:SecondSLP-1}) is defined by 
\begin{equation}
\Delta^{\circ}(\lambda):=\psi(\rho,0).\label{eq:DelltaCharfun2}
\end{equation}
Notice also that $\Delta^{\circ}(\lambda)=S(\rho,\pi)+HS'(\rho,\pi)$. 

Let $s=2$. The input data (a finite set of the eigenvalues $\mu_{k}^{2}$
of (\ref{eq:SecondSLP-1})) is computed with the aid of the NSBF representation
(\ref{eq:NsbfS}) and its derivative (see \cite{Vk2017NSBF}). The
spectrum completion is performed by using $\psi_{N}(\rho,0)$ in the
algorithm, see (\ref{eq:PArtialPPSI}). Let $H=i$, given 5 eigenvalues,
65 more were computed with a maximum absolute error of $4.32\times10^{-4}$,
see Figure \ref{fig:Abs.-and-rel.PsiCompeltionFrom5}. Additionally,
from 10 eigenvalues, 60 more were computed with a maximum absolute
error of $1.69\times10^{-5}$, see Figure \ref{fig:Abs.-and-rel.PsiComplFrom10}.
In the later case, the parameter $\omega_{H}$ (see (\ref{eq:omega}))
was obtained with an absolute error of $1.01\times10^{-3}$.

\begin{figure}[H]
\centering{}%
\begin{minipage}[c][1\totalheight][t]{0.45\textwidth}%
\begin{center}
\includegraphics[scale=0.5]{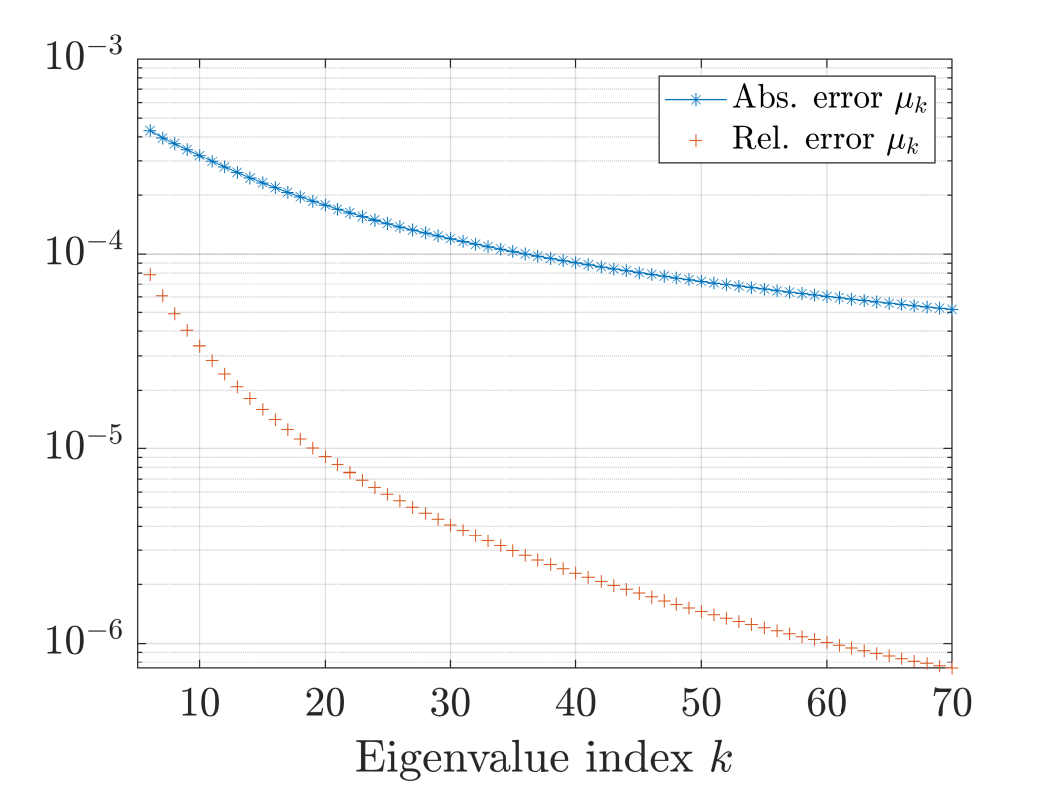}
\par\end{center}
\caption{\label{fig:Abs.-and-rel.PsiCompeltionFrom5}Example \ref{-CompletiionH}.
Abs. and rel. errors of $\mu_{k}$, $k=6,\ldots,70$ computed from
5 eigenvalues.}
\end{minipage}\hspace{0.5cm}%
\begin{minipage}[c][1\totalheight][t]{0.45\textwidth}%
\begin{center}
\includegraphics[scale=0.57]{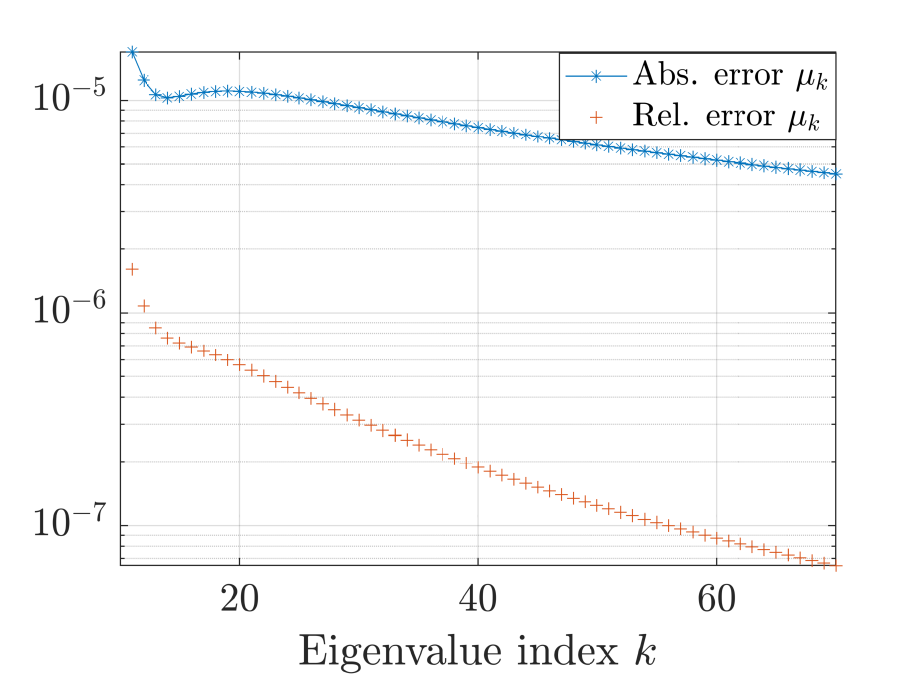}
\par\end{center}
\caption{\label{fig:Abs.-and-rel.PsiComplFrom10}Example \ref{-CompletiionH}.
Abs. and rel. errors of $\mu_{k}$, $k=11,\ldots,70$ computed from
10 eigenvalues.}
\end{minipage}
\end{figure}

Since a good accuracy is attained even with a reduced number of given
eigenvalues, the information obtained may result to be useful for
solving other problems related with the same Sturm-Liouville equation,
e.g., to the spectrum completion of another set of the eigenvalues
corresponding the same Sturm-Liouville equation with other boundary
conditions (see problem (\ref{eq:FirstSLP-1}) below).

Now, assume there is given the set $\left\{ \xi_{k}^{2}\right\} _{k=1}^{M}$
of eigenvalues of the Sturm-Liouville problem

\begin{equation}
\begin{array}{c}
-y''(x)+2i\cos(2x)+y(x)=\rho^{2}y(x),\,0<x<\pi\\
y'(0)\,-hy(0)=0,\,y'(\pi)+Hy(\pi)=0,
\end{array}\label{eq:FirstSLP-1}
\end{equation}
with $h=0.7$ (which is assumed to be unknown). Recall the characteristic
function (\ref{eq:DeltaCharfun}).

The knowledge of the approximate solution $\psi_{N}(\rho,0)$, $N=5,10$
and the value of $\omega_{H}$ (found in the spectrum completion of
the eigenvalues of problem (\ref{eq:SecondSLP-1})) is used as known
variables in the linear system (step 2 in spectrum completion algorithm)
obtained from equations $\Delta(\xi_{k}^{2})=0$. Solving this system
we obtain the coefficients $\left\{ \mathring{\psi}_{n}(0)\right\} _{n=0}^{M-2}$
and the value of $h$, and then $\Delta_{N}\left(\rho^{2}\right)$
for any $\rho\in\mathbb{C}$.

Given 5 roots of the eigenvalues ($M=5$), 65 more were computed with
a maximum absolute error of $1.69\times10^{-3}$, see Figure \ref{fig:Abs.-and-rel.DPsiCompeltionFrom5}.
Additionally, from 10 roots of the eigenvalues ($M=10$) 60 more were
computed with a maximum absolute error of $1.34\times10^{-4}$, see
Figure \ref{fig:Abs.-and-rel.-DpsiCompleFrom10}.

\begin{figure}[H]
\centering{}%
\begin{minipage}[c][1\totalheight][t]{0.45\textwidth}%
\begin{center}
\includegraphics[scale=0.57]{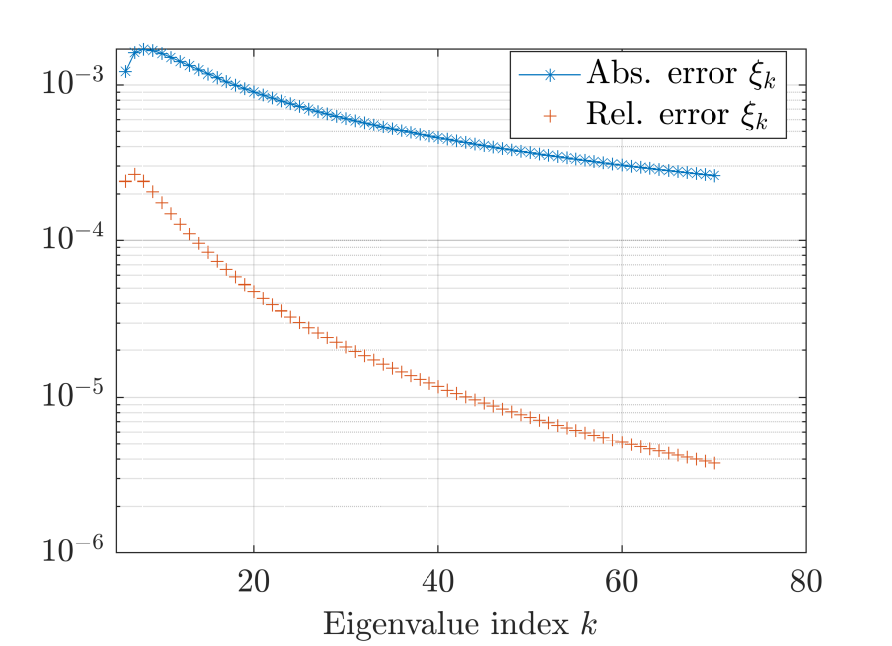}
\par\end{center}
\caption{\label{fig:Abs.-and-rel.DPsiCompeltionFrom5}Example \ref{-CompletiionH}.
Abs. and rel. errors of $\xi_{k}$, $k=6,\ldots,70$ computed from
5 eigenvalues.}
\end{minipage}\hspace{0.5cm}%
\begin{minipage}[c][1\totalheight][t]{0.45\textwidth}%
\begin{center}
\includegraphics[scale=0.6]{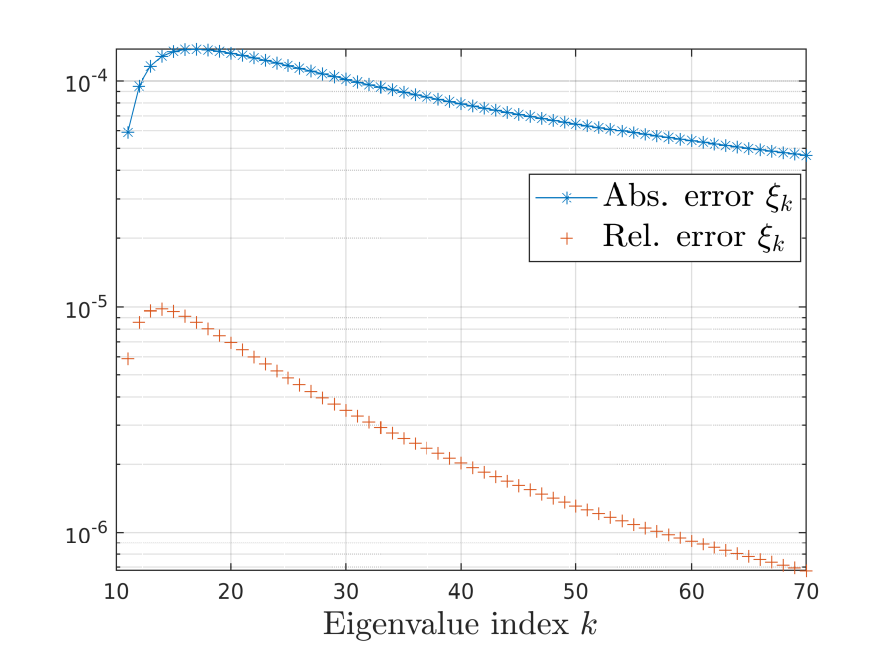}
\par\end{center}
\caption{\label{fig:Abs.-and-rel.-DpsiCompleFrom10}Example \ref{-CompletiionH}.
Abs. and rel. errors of $\xi_{k}$, $k=11,\ldots,70$ computed from
10 eigenvalues.}
\end{minipage}
\end{figure}

Moreover, the parameter $h$ is obtained with an absolute error of
$5.6\times10^{-2}$ from 5 eigenvalues and with an absolute error
of $1\times10^{-2}$ from 10 eigenvalues.

In this case, considerably more noise (corresponding to $\psi_{N}(\rho,0)$
and $\omega_{H}$) was added to the input data forming the linear
system of the spectrum completion algorithm, however, the spectrum
was completed sufficiently well. Indeed, the results obtained here
were applied below in Example \ref{ExaInverseMathi}, for solving
a two-spectra inverse Sturm-Liouville problem for the same potential.

\end{example} 

\section{Computation of a Weyl function from two finite sets of eigenvalues\label{sec:Weyl-Titchmarsh-m-function}}

In this section we discuss the computation of the Weyl function $M(\lambda)$
from two spectra $\left\{ \xi_{k}^{2}\right\} _{k=0}^{\infty}$ and
$\left\{ \mu_{m}^{2}\right\} _{m=0}^{\infty}$ corresponding to the
Sturm-Liouville problems (\ref{eq:SLProblemCompletion}) and
\begin{equation}
\begin{array}{c}
-y''(x)+q(x)y(x)=\rho^{2}y(x),\\
y(0)=0,\,y'(L)+Hy(L)=0,
\end{array}\label{eq:SecondSLP}
\end{equation}
respectively.

There are the well-known formulas (see \cite{Bterin 2007}, \cite{ButerinYurko}),
which in principle give us the Weyl function in terms of these sets
of the eigenvalues. Namely,
\begin{equation}
M(\lambda)=-\frac{\Delta^{\circ}(\lambda)}{\Delta(\lambda)}=-\frac{{\displaystyle \prod_{n=0}^{\infty}\frac{\mu_{n}^{2}-\lambda}{(n+1/2)^{2}}}}{L(\xi_{0}^{2}-\lambda){\displaystyle \prod_{n=1}^{\infty}\frac{\xi_{n}^{2}-\lambda}{n^{2}}}}\label{eq:Infintie Prducts}
\end{equation}
where $\Delta(\lambda)$ (see (\ref{eq:DeltaCharfun})) and $\Delta^{\circ}(\lambda)$
(see (\ref{eq:DelltaCharfun2})) are the characteristic functions
of problems (\ref{eq:SLProblemCompletion}) and (\ref{eq:SecondSLP}),
respectively.

In Section \ref{sec:Spectrum-completion} the algorithm for solving
the problem of spectrum completion is presented, in which the first
3 steps consist in the computation of $\Delta(\lambda)$ (procedure
analogous for the computation of $\Delta^{\circ}(\lambda)$) from
the corresponding finite set of the eigenvalues. 

Thus, in order to discuss the calculation of $M(\lambda)$ from (\ref{eq:Infintie Prducts})
let us illustrate in the following example the approximation of the
characteristic functions $\Delta(\lambda)$ and $\Delta^{\circ}(\lambda)$
from given corresponding sets of the eigenvalues $\left\{ \xi_{k}^{2}\right\} _{k=0}^{M_{1}}$
and $\left\{ \mu_{m}^{2}\right\} _{m=0}^{M_{2}}$ by using the series
representation (\ref{eq:PArtialPPSI}) in comparison with the use
of the products in (\ref{eq:Infintie Prducts}), i.e., by considering
\begin{equation}
\Delta^{\circ}(\lambda)\thickapprox\frac{L^{2}}{\pi^{2}}\prod_{n=0}^{M_{2}}\frac{\mu_{n}^{2}-\lambda}{(n+1/2)^{2}}\label{eq:ApproxDelta2}
\end{equation}
 and 

\begin{equation}
\Delta(\lambda)\thickapprox\frac{L^{3}}{\pi^{2}}(\xi_{0}^{2}-\lambda){\displaystyle \prod_{n=1}^{M_{1}}\frac{\xi_{n}^{2}-\lambda}{n^{2}}}.\label{eq:ApproxDelta1}
\end{equation}

\begin{example}\label{WeylMAthi}Consider\textbf{ }the Mathieu potential
$q(x)=is\cos(2x)$, $x\in[0,\pi]$ with parameter $s=1.468768613785142$,
see \cite{Blanch}, \cite{Mulholland } and \cite{Ziener}. 

For the constants $H=2$ and $h=i$, the characteristic functions
were computed with the aid of Wolfram Mathematica12. They were represented
in terms of the Mathieu functions $C(a,q,v)$ (the even solution of
the Mathieu differential equation), $S(a,q,v)$ (the odd solution
of the Mathieu differential equation) and their derivatives with respect
to the parameter $v$, see \cite[Ch 10.]{Abramo}. 

The exact eigenvalues were computed by the function $FindRoot$ to
obtain the numerical roots of the corresponding characteristic functions.
Figure \ref{fig:ErrorPsiSincGiven14MathieuDoublePoint} presents the
absolute error of $\Delta_{13}^{\circ}(\lambda)=\psi_{13}(\rho,0)$
computed in a strip of the upper complex $\rho$-plane from given
15 eigenvalues. The maximum absolute error is $1.18\times10^{-2}$
achieved at $-4+1.8i$. In Figure \ref{fig:ErrorRelPsiSinDoublePint}
the relative error is presented. The maximum relative error achieved
at $-25.4$ is $1\times10^{-4}$.

\begin{figure}[H]
\begin{minipage}[c][1\totalheight][t]{0.45\textwidth}%
\begin{center}
\includegraphics[scale=0.67]{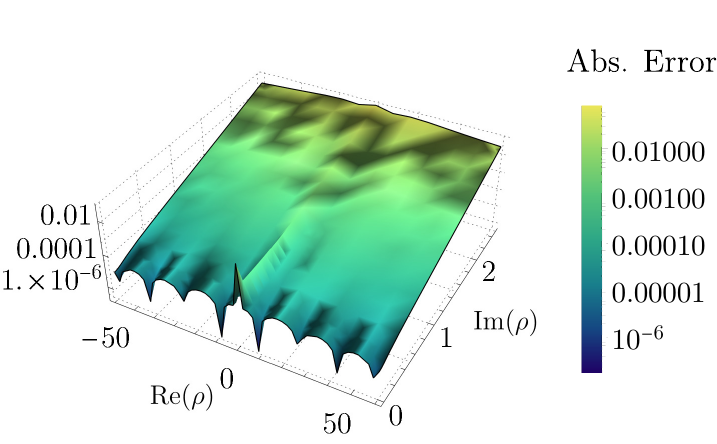}
\par\end{center}
\caption{\label{fig:ErrorPsiSincGiven14MathieuDoublePoint}Example \ref{WeylMAthi}.
Absolute error of $\Delta_{13}^{\circ}(\lambda)$ computed from $\left\{ \mu_{k}\right\} _{k=0}^{14}$
by using (\ref{eq:PArtialPPSI}).}
\end{minipage}\hfill{}%
\begin{minipage}[c][1\totalheight][t]{0.45\textwidth}%
\begin{center}
\includegraphics[scale=0.67]{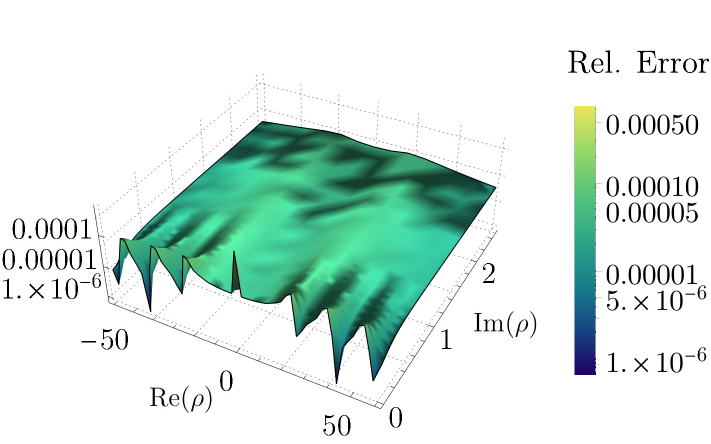}
\par\end{center}
\caption{\label{fig:ErrorRelPsiSinDoublePint}Example \ref{WeylMAthi}. Relative
error of $\Delta_{13}^{\circ}(\lambda)$ computed from $\left\{ \mu_{k}\right\} _{k=0}^{14}$
by using (\ref{eq:PArtialPPSI}).}
\end{minipage}
\end{figure}

Note that the approximation of the characteristic function decreases
as the absolute value of $\rho$ grows in the upper half-plane. In
most cases these results improve if more eigenvalues are given.

Additionally, in Figures \ref{fig:ErrorPsiSincGiven14MathieuDoublePoint-1}
and \ref{fig:ErrorPsiProductsGiven14MathieuDoublePoint-1}, the absolute
and relative errors, respectively, of $\Delta_{13}^{\circ}(\lambda)$
computed by (\ref{eq:ApproxDelta2}).

\begin{figure}[H]
\begin{minipage}[c][1\totalheight][t]{0.45\textwidth}%
\begin{center}
\includegraphics[scale=0.65]{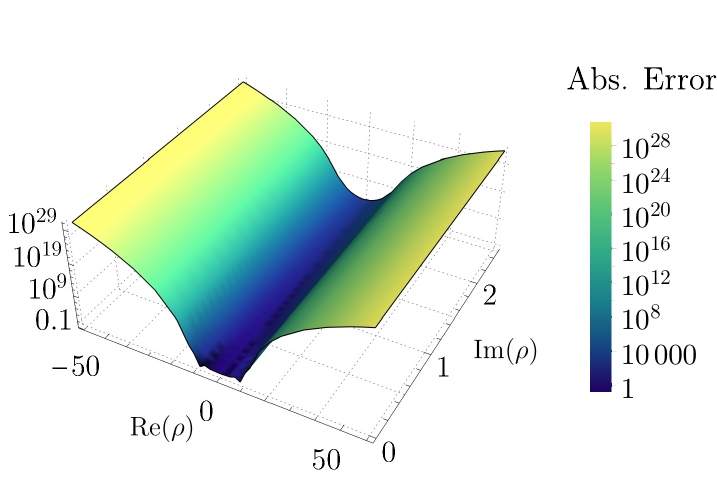}
\par\end{center}
\caption{\label{fig:ErrorPsiSincGiven14MathieuDoublePoint-1}Example \ref{WeylMAthi}.
Absolute error of $\Delta^{\circ}(\lambda)$ computed from $\left\{ \mu_{k}\right\} _{k=0}^{14}$
by using (\ref{eq:ApproxDelta2}).}
\end{minipage}\hfill{}%
\begin{minipage}[c][1\totalheight][t]{0.45\textwidth}%
\begin{center}
\includegraphics[scale=0.65]{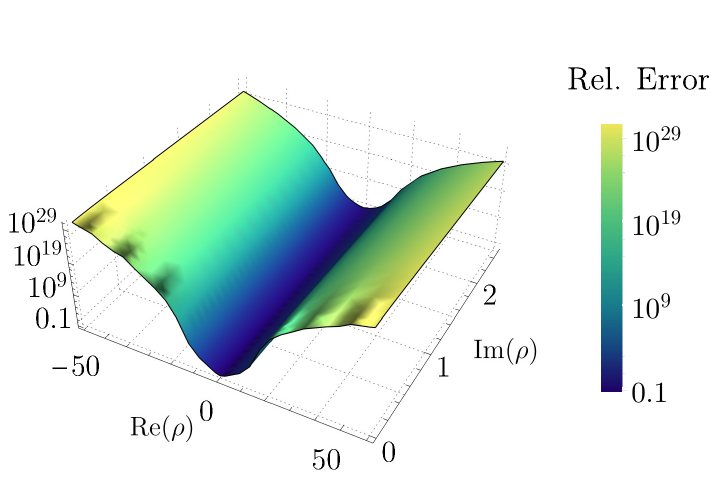}
\par\end{center}
\caption{\label{fig:ErrorPsiProductsGiven14MathieuDoublePoint-1}Example \ref{WeylMAthi}.
Relative error of $\Delta^{\circ}(\lambda)$ computed from $\left\{ \mu_{k}\right\} _{k=0}^{14}$
by using (\ref{eq:ApproxDelta2}).}
\end{minipage}
\end{figure}

Analogously, similar results are obtained for $\Delta(\lambda)$ computed
from given 15 eigenvalues by using both the cardinal series representation
$h\psi_{13}(\rho,0)-\mathring{\psi}_{14}(\rho,0)$ and the product
(\ref{eq:ApproxDelta1}). The knowledge of $\psi_{13}(\rho,0)$, with
absolute error presented in Figure \ref{fig:ErrorPsiSincGiven14MathieuDoublePoint},
is used as input data in the linear system. A detailed explanation
of this procedure was presented in Example \ref{-CompletiionH}, when
dealing with the problem (\ref{eq:FirstSLP-1}).

This simple comparison makes it obvious, that the analytical formulas
(\ref{eq:ApproxDelta2}) and (\ref{eq:ApproxDelta1}) are not of practical
usage. Instead, the series representations developed on the present
work give us accurate results and indeed allow us to reconstruct the
Weyl function from two finite sets of the eigenvalues.

In Figure \ref{fig:ErrorPsiGiven14MathieuDoublePointPRODUCTS-1} the
approximate Weyl function computed with the aid of the cardinal series
representations is presented. The corresponding absolute error of
this approximation is shown in Figure \ref{fig:ErrorRelPsiProductsGiven14MathieuDoublePointPRODUCTS-1}.

\begin{figure}[H]
\begin{minipage}[c][1\totalheight][t]{0.45\textwidth}%
\begin{center}
\includegraphics[scale=0.63]{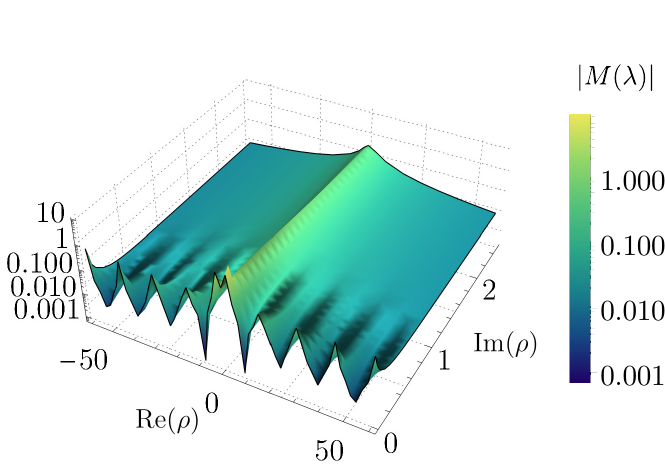}
\par\end{center}
\caption{\label{fig:ErrorPsiGiven14MathieuDoublePointPRODUCTS-1}Example \ref{WeylMAthi}.
Approximate Weyl function computed from 15 eigenvalues, and with the
aid of cardinal series representations.}
\end{minipage}\hfill{}%
\begin{minipage}[c][1\totalheight][t]{0.45\textwidth}%
\begin{center}
\includegraphics[scale=0.62]{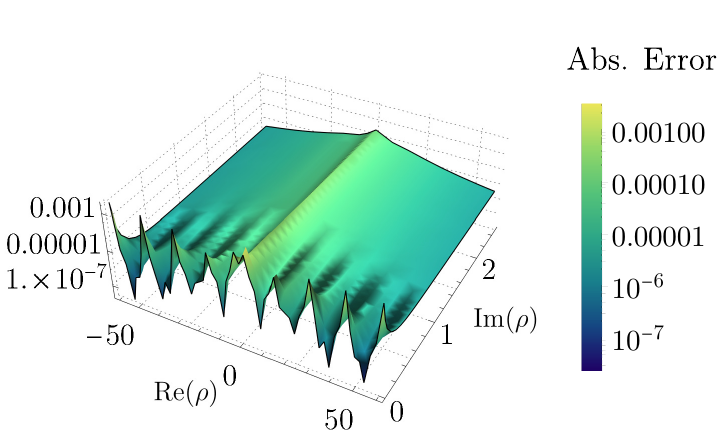}
\par\end{center}
\caption{\label{fig:ErrorRelPsiProductsGiven14MathieuDoublePointPRODUCTS-1}Example
\ref{WeylMAthi}. Abs. error of the recovered Weyl function computed
from 15 eigenvalues with the aid of cardinal series representations.}
\end{minipage}
\end{figure}

The presented results confirm that the method proposed in the first
3 steps of the spectrum completion algorithm to approximate the corresponding
characteristic function has a much higher convergence than that based
on the products in (\ref{eq:Infintie Prducts}). Additionally, the
Weyl function is computed accurately from 15 eigenpairs.

\end{example} 
\begin{rem}
A similar approach is applicable to other Weyl functions, such as
the Weyl-Titchmarsh m-function
\[
m\left(\lambda\right)=\frac{\psi'(\rho,0)}{\psi(\rho,0)},
\]

which is obtained when $h=0$ and $H=0$, see \cite{Hor} and \cite{Simon}. 
\end{rem}

\section{Inverse problem\label{sec:Inverse-problem}}

In this section we discuss a method for solving the two-spectra inverse
problem in which $q$, $h$ and $H$ are complex.

\textbf{Problem (two-spectra inverse problem) }Given two sequences
of the eigenvalues $\left\{ \xi_{k}^{2}\right\} _{k=1}^{\infty}$
and $\left\{ \mu_{m}^{2}\right\} _{m=1}^{\infty}$ of the Sturm-Liouville
problems, respectively, 
\begin{equation}
\begin{array}{c}
-y''(x)+q(x)y(x)=\rho^{2}y(x),\\
y'(0)\,-hy(0)=0,\,y'(L)+Hy(L)=0
\end{array}\label{eq:FirstSLP}
\end{equation}
and
\begin{equation}
\begin{array}{c}
-y''(x)+q(x)y(x)=\rho^{2}y(x),\\
y(0)=0,\,y'(L)+Hy(L)=0.
\end{array}\label{eq:SecondSLPP}
\end{equation}

Find the potential $q(x)$ and the complex constants $h$ and $H$. 

$ $

As mentioned in the previous section, a characteristic function of
problem (\ref{eq:FirstSLP}) is (\ref{eq:DeltaCharfun}), and a characteristic
function of problem (\ref{eq:SecondSLPP}) is (\ref{eq:DelltaCharfun2}). 

Additionally, recall the identity satisfied by these characteristic
functions 
\begin{equation}
\psi(\rho,x)=(\psi'(\rho,0)-h\psi(\rho,0))S(\rho,x)+\psi(\rho,0)\phi(\rho,x),\rho\in\mathbb{C},\label{eq:genEq}
\end{equation}

see \cite[Formula (1.4.9), p. 29 ]{FreYurko}.

The approach developed for solving the two-spectra inverse problem
develops the idea from \cite{VKComplex}, where Neumann series of
Bessel functions representations for solutions of (\ref{eq:PrincipalEq})
were used. Our approach is based on the cardinal series representations
from Section \ref{sec:Cardinal-series-representation}. The potential
is recovered from the first coefficients of the representations. Moreover,
two possibilities are explored. The first one involves the cardinal
series from Theorem \ref{thm:Let-.-TheNewSinc}. Due to (\ref{eq:op2-1}),
no differentiation is required in the final step. The second one deals
with the cardinal series representation from Theorem \ref{thm:Let-.-The OldSIn},
and the potential is recovered from (\ref{eq:Potential First Coeffs}),
which implies a double differentiation in the last step.

The algorithm to recover $q(x)$, $h$ and $H$ consists of the following
steps.

\textbf{Numerical algorithm for solving the} \textbf{two-spectra inverse
problem for potentials from }$W^{2,1}[0,\pi]$

Assume that two finite sets $\left\{ \xi_{k}^{2}\right\} _{k=1}^{K}$
and $\left\{ \mu_{m}^{2}\right\} _{m=1}^{M}$ of the eigenvalues of
problems (\ref{eq:FirstSLP}) and (\ref{eq:SecondSLPP}), respectively,
are given.
\begin{enumerate}
\item Consider the approximate characteristic function $\Delta_{N}^{\circ}(\lambda)=\psi_{N}(\rho,0)$
of problem (\ref{eq:SecondSLPP}) on the set $\left\{ \mu_{m}\right\} _{m=1}^{M}$
, i.e.,
\begin{align}
\psi_{N}(\mu_{m},0) & =\cos(\mu_{m}L)+\frac{\sin(\mu_{m}L)}{\mu_{m}}\omega_{H}-\frac{L}{\mu_{k}}j_{1}(\mu_{m}L)q_{L}^{H}(0)\label{eq:CharPsiNewsinc}\\
 & -\frac{1}{\mu_{m}^{2}}\sum_{n=1}^{N}\tilde{\Psi}_{n}(0)\left(\sinc\left(\frac{\mu_{m}L}{\pi}+n\right)+\sinc\left(\frac{\mu_{m}L}{\pi}-n\right)\right)=0,\nonumber 
\end{align}
and $m=1,\ldots,M,\,\,N+2\leq M.$\\
Solve the finite system of $M$ linear equations derived from the
equalities 
\begin{align*}
 & \frac{\sin(\mu_{m}L)}{\mu_{m}}\omega_{H}-\frac{L}{\mu_{k}}j_{1}(\mu_{m}L)q_{L}^{H}(0)\\
 & -\frac{1}{\mu_{m}^{2}}\sum_{n=1}^{N}\tilde{\Psi}_{n}(0)\left(\sinc\left(\frac{\mu_{m}L}{\pi}+n\right)+\sinc\left(\frac{\mu_{m}L}{\pi}-n\right)\right)=-\cos(\mu_{m}L),
\end{align*}
which are obtained from (\ref{eq:CharPsiNewsinc}). This gives the
coefficients computed
\begin{equation}
\omega_{H},\,q_{L}^{H}(0)\,\,\text{and}\,\,\left\{ \tilde{\Psi}_{n}(0)\right\} _{n=1}^{N}.\label{eq:CoefsNewSincPsi}
\end{equation}
Note that, using the coefficients (\ref{eq:CoefsNewSincPsi}) (independent
on $\rho$), the approximate characteristic function $\psi_{N}\left(\rho,0\right)$
of problem (\ref{eq:SecondSLPP}) can be computed for any $\rho$. 
\item Consider the approximate characteristic function $h\psi_{N}(\rho,0)-\mathring{\psi}_{N}(\rho,0)$
on the set $\left\{ \xi_{k}\right\} _{k=1}^{K}$, i.e.,{\small{}
\begin{align}
 & \mathring{\psi}_{N}(\xi_{k},0)-h\psi_{N}(\xi_{k},0)=\xi_{k}\sin(\xi_{k}L)+\omega_{H}\xi_{k}Lj_{1}(\xi_{k}L)\nonumber \\
 & +\sum_{n=0}^{N-1}\mathring{\psi}_{n}(0)\left(\sinc\left(\frac{\xi_{k}L}{\pi}+n\right)+\sinc\left(\frac{\xi_{k}L}{\pi}-n\right)\right)-h\psi_{N}(\xi_{k},0)=0,\,k=1,\ldots,K.\label{eq:AuxEq}
\end{align}
}Solve the finite system of $K$ linear equations derived from the
equalities{\small{}
\begin{align}
\sum_{n=0}^{N-1}\mathring{\psi}_{n}(0)\left(\sinc\left(\frac{\xi_{k}L}{\pi}+n\right)+\sinc\left(\frac{\xi_{k}L}{\pi}-n\right)\right)-h\psi_{N}(\xi_{k},0)=-\omega_{H}\xi_{k}Lj_{1}(\xi_{k}L)-\xi_{k}\text{sin}(\xi_{k}L),\label{eq:AuxEq-2}
\end{align}
}$k=1,\ldots,K,$ which are obtained from (\ref{eq:AuxEq}). The knowledge
of $\psi_{N}\left(\rho,0\right)$ simplifies the system, and solving
it we obtain the coefficients
\begin{equation}
h\,\,\,\text{and}\,\,\,\left\{ \mathring{\psi}_{n}(0)\right\} _{n=0}^{N-1}.\label{eq:CoefsDSincPsi}
\end{equation}
Note that, by using the coefficients (\ref{eq:CoefsDSincPsi}) (independent
on $\rho$), and $\psi_{N}\left(\rho,0\right)$ obtained above, the
approximate characteristic function $\mathring{\psi}_{N}(\rho,0)-h\psi_{N}(\rho,0)$
can be computed for any $\rho$. 
\item Substitution of (\ref{eq:TruncatedNewCardinalS}), (\ref{eq:TruncatedNewCardinalPhi-1})
and (\ref{eq:PArtialPPSI}), into (\ref{eq:genEq}) leads to the approximate
equality
\begin{equation}
\psi_{N}(\rho,x)=\left(\mathring{\psi}_{N}(\rho,0)-h\psi_{N}(\rho,0)\right)\mathtt{s}_{N}(\rho,x)+\psi_{N}(\rho,0)\Phi_{N}(\rho,x),\label{eq:eq:genEqTrunca}
\end{equation}
where different values of $N$ can be considered in each sum. Equation
(\ref{eq:eq:genEqTrunca}) implies
\begin{align}
 & \cos(\rho(x-L))+\frac{\sin(\rho(L-x))}{\rho}\omega_{H}(x)-\frac{L-x}{\rho}j_{1}(\rho(L-x))q_{L}^{H}(x)\nonumber \\
 & -\frac{1}{\rho^{2}}\sum_{n=1}^{N}\tilde{\Psi}_{n}(x)\left(\sinc\left(\frac{\rho(L-x)}{\pi}+n\right)+\sinc\left(\frac{\rho(L-x)}{\pi}-n\right)\right)\nonumber \\
 & =\left(\mathring{\psi}_{N}(\rho,0)-h\psi_{N}(\rho,0)\right)\frac{\sin(\rho x)}{\rho}-\left(\mathring{\psi}_{N}(\rho,0)-h\psi_{N}(\rho,0)\right)\frac{\cos(\rho x)}{\rho^{2}}\omega(x)\nonumber \\
 & +\left(\mathring{\psi}_{N}(\rho,0)-h\psi_{N}(\rho,0)\right)\frac{\sin(\rho x)}{\rho^{3}}q^{+}(x)\nonumber \\
 & -\frac{\left(\mathring{\psi}_{N}(\rho,0)-h\psi_{N}(\rho,0)\right)}{2\rho^{3}\sin(\rho x)}\sum_{n=0}^{N-1}\mathfrak{s}_{n}(x)\left(\sinc\left(\frac{2\rho x}{\pi}-2n-1)\right)+\sinc\left(\frac{2\rho x}{\pi}+2n+1\right)\right)\nonumber \\
 & +\psi_{N}(\rho,0)\cos(\rho x)+\psi_{N}(\rho,0)\frac{\sin(\rho x)}{\rho}\omega_{h}(x)-\frac{x}{\rho}\psi_{N}(\rho,0)j_{1}(\rho x)q^{h}(x)\nonumber \\
 & -\frac{\psi_{N}(\rho,0)}{\rho^{2}}\sum_{n=1}^{N}\Phi_{n}(x)\left(\sinc\left(\frac{\rho x}{\pi}-n\right)+\sinc\left(\frac{\rho x}{\pi}+n\right)\right),\,\,x\in(0,L),\,\rho\in\mathbb{C\setminus}\{0\}.\label{eq:MainEq}
\end{align}
where $\omega_{h}(x)=\omega(x)+h$. Here, the coefficients 
\begin{equation}
\omega_{H}(x),q_{L}^{H}(x),\left\{ \tilde{\Psi}_{n}(x)\right\} _{n=1}^{N},\omega(x),q^{+}(x),\left\{ \mathfrak{s}_{n}(x)\right\} _{n=0}^{N-1},\omega_{h}(x),q^{h}(x)\,\,\text{and}\,\,\left\{ \Phi_{n}(x)\right\} _{n=1}^{N}\label{eq:unkowns}
\end{equation}
are unknown. Then, for any $x\in(0,L)$ construct a linear system
of $J$ equations for the unknowns (\ref{eq:unkowns}), by considering
(\ref{eq:MainEq}) evaluated at a finite set $\left\{ \tilde{\rho}_{j}\right\} _{j=1}^{J}$,
\begin{align}
 & \frac{L-x}{\tilde{\rho}_{j}}j_{1}(\tilde{\rho}_{j}(L-x))q_{L}^{H}(x)-\frac{\sin(\tilde{\rho}_{j}(L-x))}{\tilde{\rho}_{j}}\omega_{H}(x)\nonumber \\
 & +\frac{1}{\tilde{\rho}_{j}^{2}}\sum_{n=1}^{N}\tilde{\Psi}_{n}(x)\left(\sinc\left(\frac{\tilde{\rho}_{j}(L-x)}{\pi}+n\right)+\sinc\left(\frac{\tilde{\rho}_{j}(L-x)}{\pi}-n\right)\right)\nonumber \\
 & -\left(\mathring{\psi}_{N}(\tilde{\rho}_{j},0)-h\psi_{N}(\tilde{\rho}_{j},0)\right)\left(\frac{\cos(\tilde{\rho}_{j}x)}{\tilde{\rho}_{j}^{2}}\omega(x)-\frac{\sin(\tilde{\rho}_{j}x)}{\tilde{\rho}_{j}^{3}}q^{+}(x)\right)\nonumber \\
 & -\frac{\left(\mathring{\psi}_{N}(\tilde{\rho}_{j},0)-h\psi_{N}(\tilde{\rho}_{j},0)\right)}{2\tilde{\rho}_{j}^{3}\sin(\tilde{\rho}_{j}x)}\sum_{n=0}^{N-1}\mathfrak{s}_{n}(x)\left(\sinc\left(\frac{2\tilde{\rho}_{j}x}{\pi}-2n-1)\right)+\sinc\left(\frac{2\tilde{\rho}_{j}x}{\pi}+2n+1\right)\right)\nonumber \\
 & +\psi_{N}(\tilde{\rho}_{j},0)\left(\frac{\sin(\tilde{\rho}_{j}x)}{\tilde{\rho}_{j}}\omega_{h}(x)-\frac{x}{\tilde{\rho}_{j}}j_{1}(\tilde{\rho}_{j}x)q^{h}(x)\right)\nonumber \\
 & -\frac{\psi_{N}(\tilde{\rho}_{j},0)}{\tilde{\rho}_{j}^{2}}\sum_{n=1}^{N}\Phi_{n}(x)\left(\sinc\left(\frac{\tilde{\rho}_{j}x}{\pi}-n\right)+\sinc\left(\frac{\tilde{\rho}_{j}x}{\pi}+n\right)\right)\label{eq:MainEq-1}\\
 & \boldsymbol{=}\cos(\tilde{\rho}_{j}(x-L))-\left(\mathring{\psi}_{N}(\tilde{\rho}_{j},0)-h\psi_{N}(\tilde{\rho}_{j},0)\right)\frac{\sin(\tilde{\rho}_{j}x)}{\tilde{\rho}_{j}}-\psi_{N}(\tilde{\rho}_{j},0)\cos(\tilde{\rho}_{j}x),\,j=1,\ldots,J.\nonumber 
\end{align}
 
\item For $x=L$ we have $\psi(\rho,L)=1$, so equation (\ref{eq:eq:genEqTrunca})
at $x=L$ evaluated at a finite set $\left\{ \hat{\rho}_{i}\right\} _{i=1}^{I}$
has the form 
\begin{align*}
 & -\left(\mathring{\psi}_{N}(\hat{\rho}_{i},0)-h\psi_{N}(\hat{\rho}_{i},0)\right)\left(\frac{\cos(\hat{\rho}_{i}L)}{\hat{\rho}_{i}^{2}}\omega-\frac{\sin(\hat{\rho}_{i}L)}{\hat{\rho}_{i}^{3}}q^{+}(L)\right)\\
 & -\frac{\left(\mathring{\psi}_{N}(\hat{\rho}_{i},0)-h\psi_{N}(\hat{\rho}_{i},0)\right)}{2\hat{\rho}_{i}^{3}\sin(\hat{\rho}_{i}L)}\sum_{n=0}^{N-1}\mathfrak{s}_{n}(L)\left(\sinc\left(\frac{2\hat{\rho}_{i}L}{\pi}-2n-1)\right)+\sinc\left(\frac{2\hat{\rho}_{i}L}{\pi}+2n+1\right)\right)\\
 & +\psi_{N}(\hat{\rho}_{i},0)\left(\frac{\sin(\hat{\rho}_{i}L)}{\hat{\rho}_{i}}\omega_{h}-\frac{L}{\hat{\rho}_{i}}j_{1}(\hat{\rho}_{i}L)q^{h}(L)\right)\\
 & -\frac{\psi_{N}(\hat{\rho}_{i},0)}{\hat{\rho}_{i}^{2}}\sum_{n=1}^{N}\tilde{\Phi}_{n}(L)\left(\sinc\left(\frac{\hat{\rho}_{i}L}{\pi}-n\right)+\sinc\left(\frac{\hat{\rho}_{i}L}{\pi}+n\right)\right)\\
 & \boldsymbol{=}1-\left(\mathring{\psi}_{N}(\hat{\rho}_{i},0)-h\psi_{N}(\hat{\rho}_{i},0)\right)\frac{\sin(\hat{\rho}_{i}L)}{\hat{\rho}_{i}}-\psi_{N}(\hat{\rho}_{i},0)\cos(\hat{\rho}_{i}L),\,i=1,\ldots,I,
\end{align*}
\textrm{where} $\omega_{h}=\omega+h$. \\
Solving this system we compute
\begin{equation}
\omega,q^{+}(L),\left\{ \mathfrak{s}_{n}(L)\right\} _{n=0}^{N-1},h,q^{h}(L)\,\,\text{and}\,\,\left\{ \Phi_{n}(L)\right\} _{n=1}^{N}.\label{eq:CoeffsAtL}
\end{equation}
\item From (\ref{eq:CoefsNewSincPsi}) and (\ref{eq:CoeffsAtL}) obtain
the complex numbers $h,H$ and $\omega$. Additionally, since 
\[
q(L)=2(q^{+}(L)+q^{h}(L)+\omega^{2}+h\omega)\,\,\,\text{and}\,\,\,q(0)=2(q^{+}(L)-q^{h}(L)-h\omega)
\]
the potential is approximated at the endpoints of the interval.
\item Finally, it is possible to recover the potential in $(0,L)$ following
one of two options. The first one by derivation: 
\begin{equation}
q(x)=2\frac{d}{dx}\omega(x),\label{eq:op1}
\end{equation}
and the second one from coefficients (\ref{eq:unkowns}) and with
the aid of (\ref{eq:op2-1}).
\end{enumerate}
The performance of the algorithm is illustrated by the following examples.

\begin{example} \label{ExaInverseMathi}Consider a Mathieu potential,
see \cite{Mathieu},
\[
q(x)=is\cos(2x),\,x\in[0,\pi].
\]

Let $s=2$, $h=0.7$ and $H=i$. For computing the spectrum of the
Sturm-Liouville problems (\ref{eq:FirstSLP}) and (\ref{eq:SecondSLPP}),
the NSBF representations (\ref{eq:NsbfS}), (\ref{eq:NsbfPhi}) and
their derivatives (see \cite{Vk2017NSBF}) were used to approximate
the corresponding characteristic functions. Additionally, to find
complex zeros of the approximate characteristic function, the algorithm
of the argument principle theorem was applied, see \cite{Vk14}.

The recovery of the potential was performed from given $10$ pairs
of the eigenvalues (corresponding to the Sturm-Liouville problems
(\ref{eq:FirstSLP}) and (\ref{eq:SecondSLPP})). In step 6 of the
algorithm for solving the inverse problem, both options (\ref{eq:op1})
and (\ref{eq:op2-1}) were implemented, resulting in the recovery
of the potential with an absolute error of $1.7\times10^{-1}$ in
both cases, see Figure \ref{fig:Abs.-errors-orMathew10and15}. Using
$15$ eigenpairs the potential was recovered with a maximum absolute
error of $5.35\times10^{-2},$ see Figure \ref{fig:Abs.-errors-orMathew10and15}.
In steps 3 and 4 of the algorithm, the points $\left\{ \tilde{\rho}_{j}\right\} _{j=1}^{700}$
and $\left\{ \hat{\rho_{i}}\right\} _{i=1}^{700}$ were chosen as
logarithmically equally distributed on the segment $[0.1,1500]$,
and the number $N$ in (\ref{eq:MainEq-1}) was chosen as $N=15$.

\begin{figure}[H]
\begin{centering}
\includegraphics[scale=0.55]{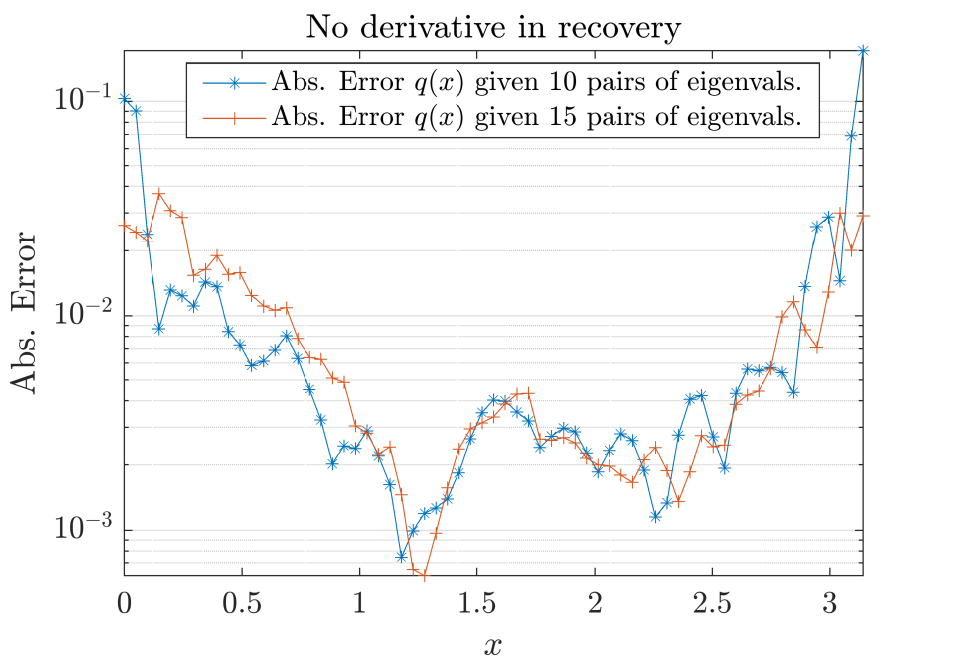}\includegraphics[scale=0.55]{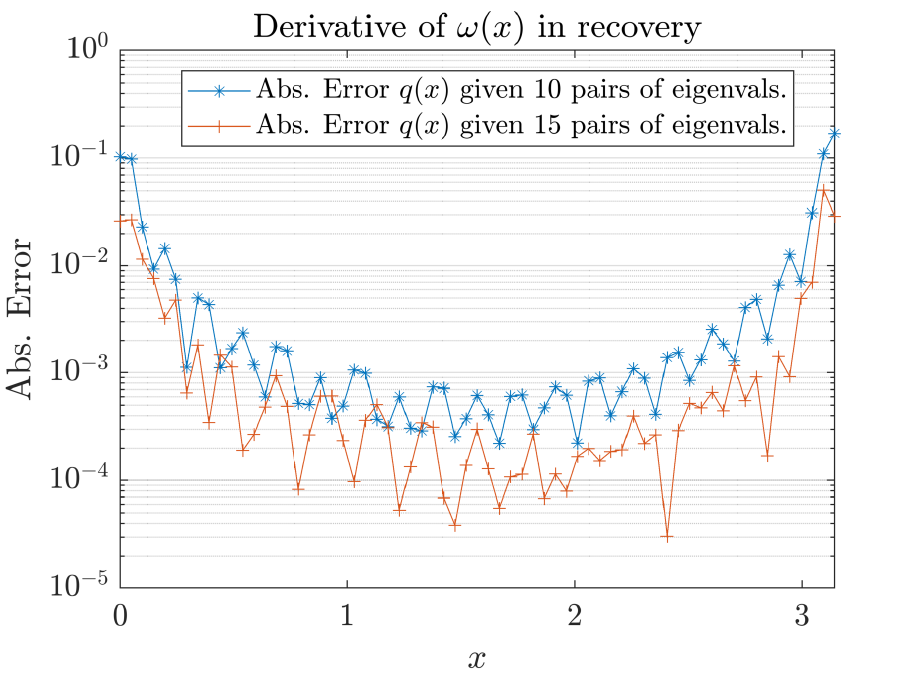}
\par\end{centering}
\caption{\label{fig:Abs.-errors-orMathew10and15}Example \ref{ExaInverseMathi}.
Abs. error of the Mathieu potential recovered from given 10 and 15
eigenpairs, using option (\ref{eq:op2-1}) on the left and option
(\ref{eq:op1}) on the right.}
\end{figure}

Furthermore, for a given set of 10 eigenpairs, the constants $h,H$
and $\omega$ were computed with the absolute errors $1.03\times10^{-2}$,
$1.07\times10^{-2}$ and $8.1\times10^{-3}$, respectively. In the
case of 15 eigenpairs given, the absolute errors of these constants
were $4.16\times10^{-3}$, $4.32\times10^{-3}$ and $1.51\times10^{-3}$,
respectively.

Now, let us consider the points in step 3 as $\left\{ \xi_{k}\right\} _{k=1}^{K}$,
the square roots of the spectrum of the problem (\ref{eq:FirstSLP}),
i.e., $h\psi_{N}(\xi_{k},0)-\mathring{\psi}_{N}(\xi_{k},0)$ is approximately
$0$, and (\ref{eq:eq:genEqTrunca}) is reduced to 
\[
\psi_{N}(\xi_{k},x)=\psi_{N}(\xi_{k},0)\Phi_{N}(\xi_{k},x).
\]

Hence, the number of unknowns in the system of step 3 is reduced to
only those associated with $\psi_{N}(\xi_{k},x)$ and $\Phi_{N}(\xi_{k},x)$.
Despite this reduction, having 10 or 15 eigenvalues is insufficient
to accurately recover the coefficients (\ref{eq:unkowns}). Therefore,
in this case, the spectrum completion procedure presented in Example
\ref{-CompletiionH} results to be very useful for increasing the
number of given eigenvalues as input data.

Associated to the problem (\ref{eq:FirstSLP}), from 10 eigenvalues
65 more eigenvalues found by the spectrum completion algorithm were
used. With this additional information the complex constants $h$,
$H$ and $\omega$ were found with absolute errors $1.57\times10^{-4}$,
$1.29\times10^{-4}$ and $1.34\times10^{-4}$, respectively. Moreover,
the potential was recovered with an absolute error $8.23\times10^{-3}$
with the two options (\ref{eq:op1}) and (\ref{eq:op2-1}), see Figure
\ref{fig:Abs.-errors-orMathew10and15-1}.

\begin{figure}[H]
\begin{centering}
\includegraphics[scale=0.5]{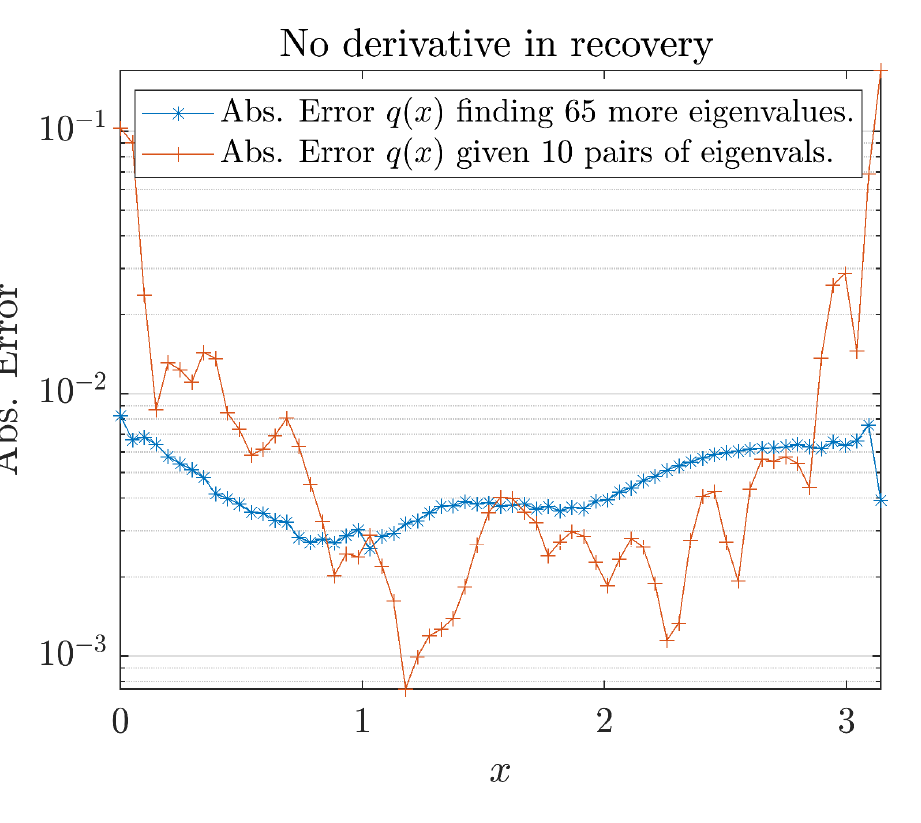}\includegraphics[scale=0.52]{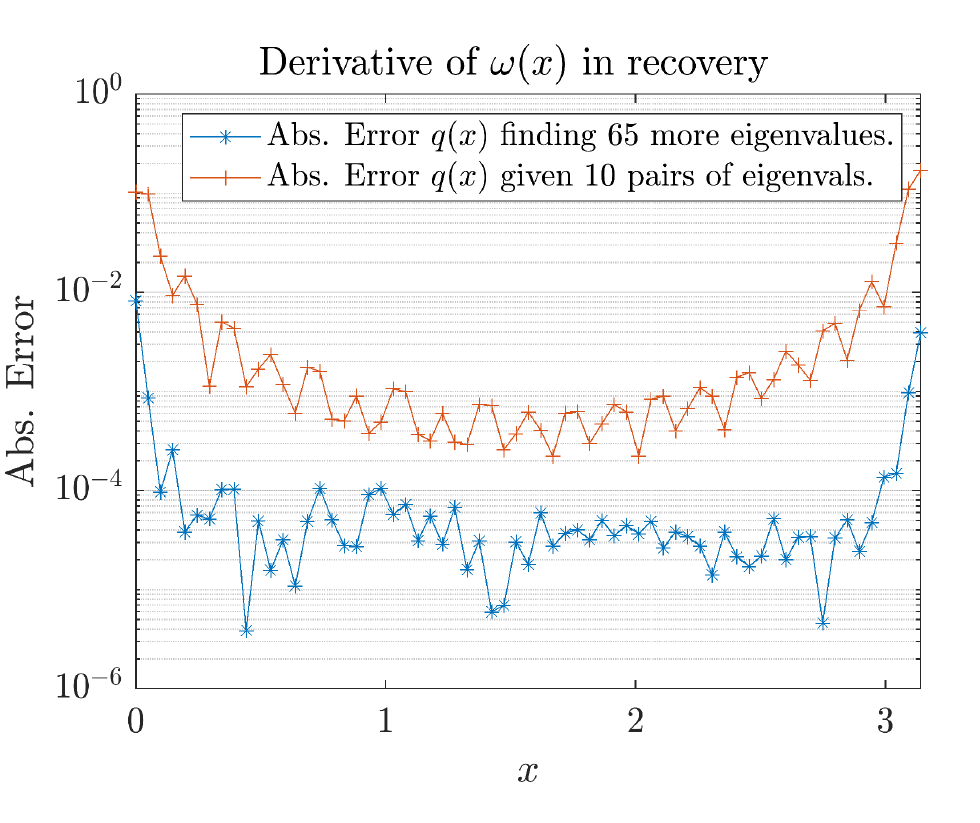}
\par\end{centering}
\caption{\label{fig:Abs.-errors-orMathew10and15-1}Example \ref{ExaInverseMathi}.
Abs. error of the Mathieu potential recovered from given 10 eigenvalues
and one spectrum completed, using option (\ref{eq:op2-1}) on the
left and option (\ref{eq:op1}) on the right.}
\end{figure}

In Figure \ref{fig:Recovered-Mathew-potential10completion}, the recovered
potential (using (\ref{eq:op2-1})), from 10 eigenpairs and 65 more
eigenvalues of problem (\ref{eq:FirstSLP}) completed is presented.
\begin{figure}[H]
\begin{centering}
\includegraphics[scale=0.7]{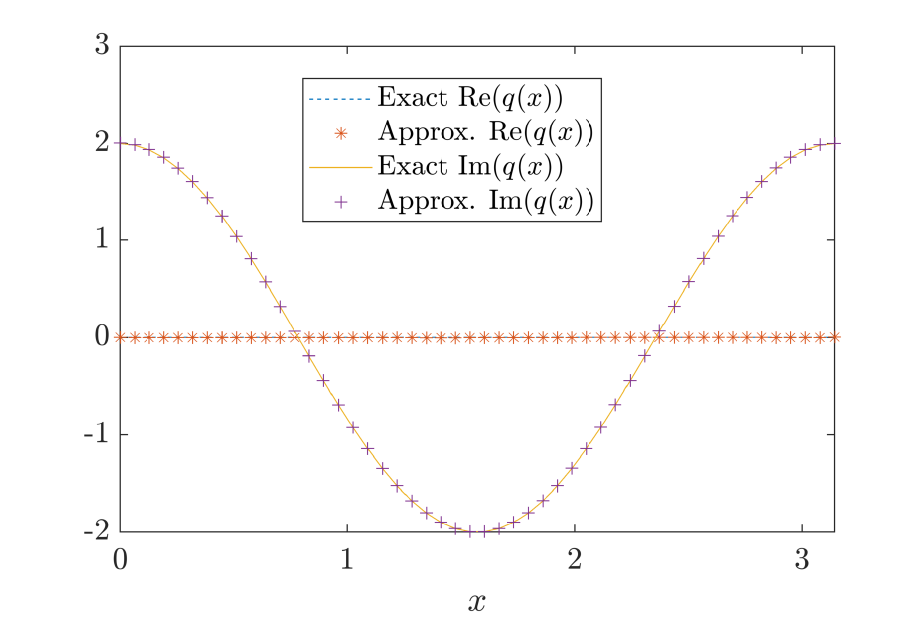}
\par\end{centering}
\caption{\label{fig:Recovered-Mathew-potential10completion}Example \ref{ExaInverseMathi}.
Recovered Mathieu potential $q(x)=2i\cos(2x)$, $x\in[0,\pi]$, with
10 eigenpairs, 65 more eigenvalues completed and no derivatives involved
in the recovery.}
\end{figure}

This example explores the application of the spectrum completion for
solving inverse problems. By using the completed eigenvalues obtained
in Example \ref{-CompletiionH}, the recovery of the potential from
10 eigenpairs was considerably improved.

\end{example}

\begin{example} \label{-Consider-theC1}Consider the potential 
\[
q(x)=\left|x-1\right|(x-1)+\left(\frac{x}{\pi}\right)^{5/3}i,\,x\in[0,\pi].
\]

Ten eigenvalues of each problem (\ref{eq:FirstSLP}) and (\ref{eq:SecondSLPP})
with the parameters $h=0.8$ and $H=0.5$ are given. The constants
$h$, $H$ and $\omega$ are found with the absolute errors $1.35\times10^{-2}$,
$1.42\times10^{-2}$ and $1.06\times10^{-2}$, respectively, while
the potential was recovered with an absolute error of $2.09\times10^{-1}$
(by using both options (\ref{eq:op1}) and (\ref{eq:op2-1})), see
Figure \ref{fig:Abs.-errors-orMathew10and15-1-1}. Now, considering
18 eigenpairs, the constants $h$, $H$ and $\omega$ are found with
the absolute errors $3.43\times10^{-3}$, $3.57\times10^{-3}$ and
$3.25\times10^{-3}$, while the potential was recovered with an absolute
error $7.92\times10^{-2}$, by using both options (\ref{eq:op1})
and (\ref{eq:op2-1}), see Figure \ref{fig:Recovered-C1-potential}.
It is worth to mention that in this case, $\omega(x)$ was recovered
with an absolute error of $2.26\times10^{-3}$. In both cases, given
10 and 18 eigenvalues, the points $\left\{ \tilde{\rho}_{j}\right\} _{j=1}^{1200}$
were chosen as logarithmically equally distributed on the segment
$[0.1,1800]$, and the number $N$ in (\ref{eq:MainEq-1}) was chosen
as $N=10$.

\begin{figure}[H]
\begin{centering}
\includegraphics[scale=0.6]{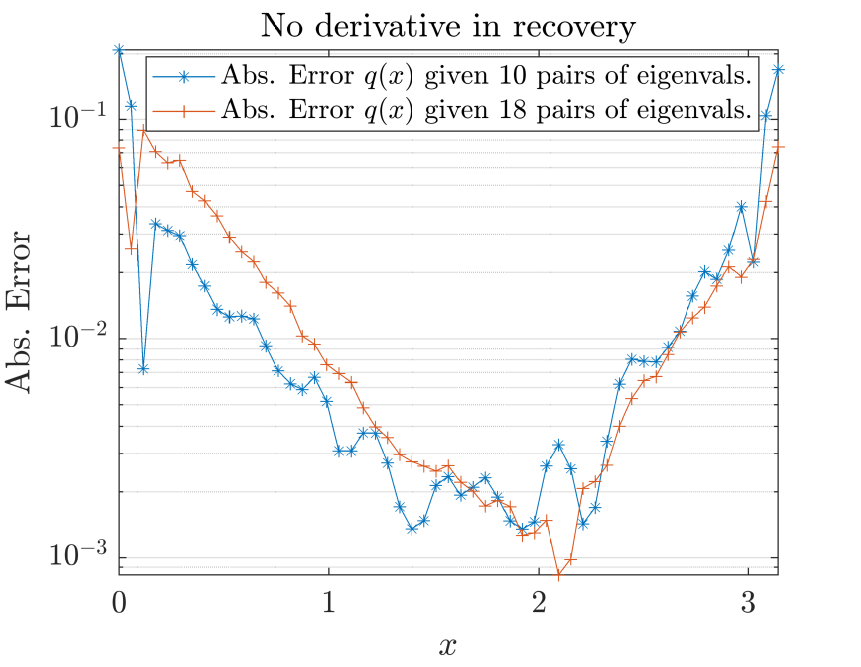}\includegraphics[scale=0.6]{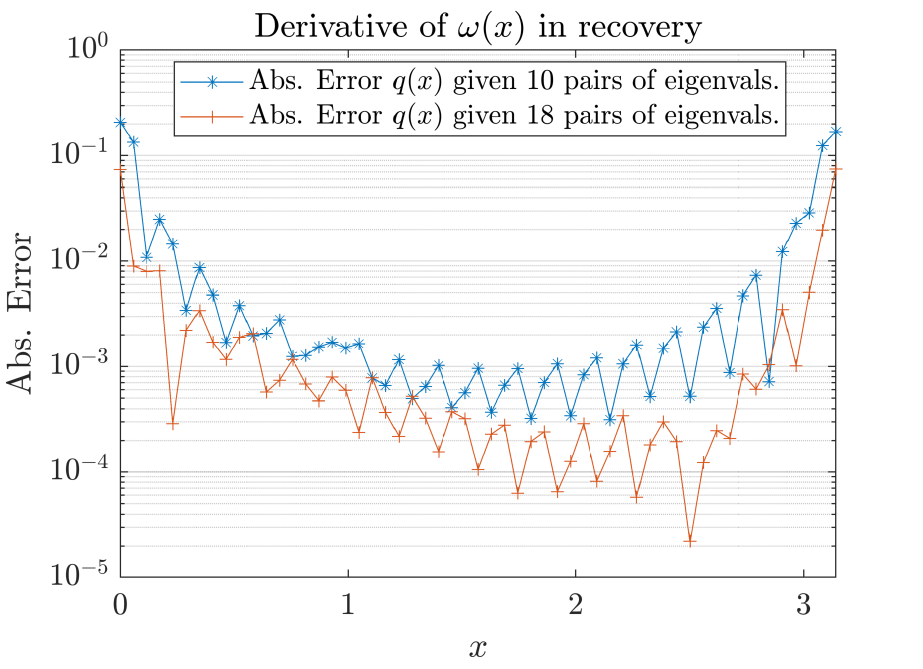}
\par\end{centering}
\caption{\label{fig:Abs.-errors-orMathew10and15-1-1}Abs. error of the potential
of Example \ref{-Consider-theC1} recovered from given 10 and 18 eigenpairs,
using option (\ref{eq:op2-1}) on the left and option (\ref{eq:op1})
on the right.}
\end{figure}

\begin{figure}[H]
\begin{centering}
\includegraphics[scale=0.7]{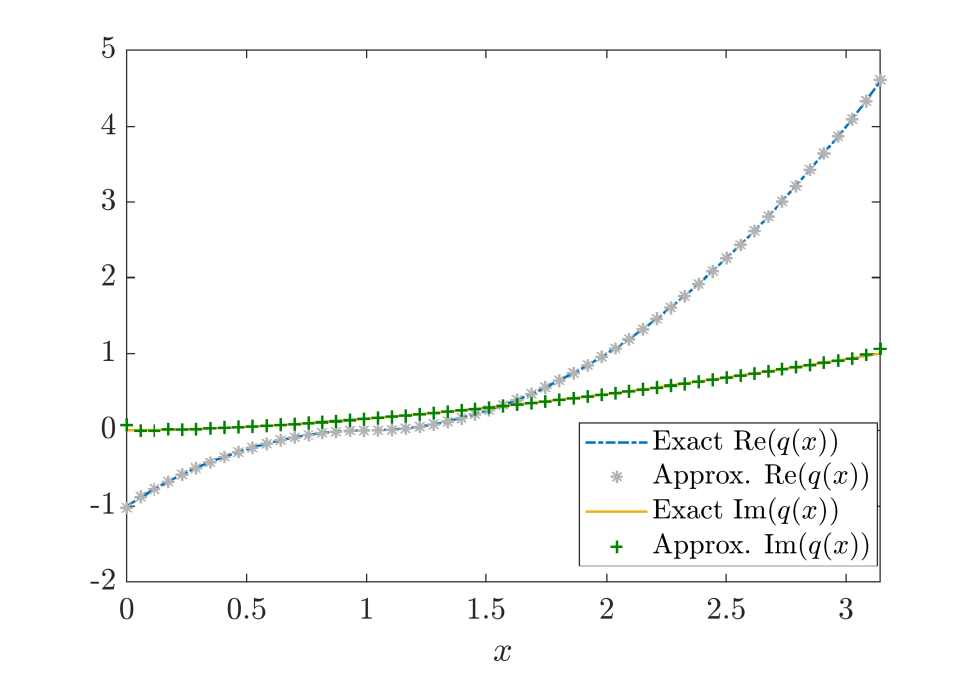}
\par\end{centering}
\caption{\label{fig:Recovered-C1-potential}Potential from Example \ref{-Consider-theC1}
recovered from 18 eigenpairs and taking one derivative of $\omega(x)$.}
\end{figure}

Note that the recovery of this complex valued and only once differentiable
potential required only slightly more input data to attain an accuracy
comparable with previous real-valued and smooth examples.

\end{example} 

\textbf{Numerical algorithm to solving} \textbf{the two-spectra inverse
problem for potentials from $\mathcal{L}_{2}[0,L]$}

Assume that two finite sets $\left\{ \xi_{k}^{2}\right\} _{k=1}^{K}$
and $\left\{ \mu_{m}^{2}\right\} _{m=1}^{M}$ of the eigenvalues of
problem (\ref{eq:FirstSLP}) and (\ref{eq:SecondSLPP}), respectively,
are given. The procedure is analogous to that for potentials in $W^{2,1}[0,L]$.
The difference consists in using the cardinal series representations
introduced in Section \ref{subsec:Cardinal-series-representation}
for $\psi(\rho,x)$, $\phi(\rho,x)$ and $S(\rho,x)$. This implies
a double differentiation in the last step for the recovery of the
potential. 

The algorithm to recover $q(x)$, $h$ and $H$ consists of approximating
(from the given eigenvalues) the characteristic functions $\Delta_{N}^{\circ}(\lambda)$
and $\Delta_{N}(\lambda)$ by $\Psi_{N}(\rho,0)$ and $h\Psi_{N}(\rho,0)-\mathring{\psi}_{N}(\rho,0)$.
Construct for any $x\in(0,L)$ the linear system of $J$ equations
by considering 
\begin{equation}
\Psi_{N}(\rho,x)=(\mathring{\psi}_{N}(\rho,0)-h\Psi_{N}(\rho,0))S_{N}(\rho,x)+\Psi_{N}(\rho,0)\varphi_{N}(\rho,x)\cdot\label{eq:eq:genEqTrunca-1}
\end{equation}
evaluated at a finite set $\left\{ \tilde{\rho}_{j}\right\} _{j=1}^{J}$,
i.e.,
\begin{align*}
 & \sum_{n=0}^{N-1}\psi_{n}(x)\left(\sinc\left(\frac{\tilde{\rho}_{j}(L-x)}{\pi}-n\right)+\sinc\left(\frac{\tilde{\rho}_{j}(L-x)}{\pi}+n\right)\right)\\
 & -\frac{\mathring{\psi}_{N}(\tilde{\rho}_{j},0)-h\Psi_{N}(\tilde{\rho}_{j},0)}{2\tilde{\rho}_{j}\sin(\tilde{\rho}_{j}x)}\sum_{n=0}^{N-1}s_{n}(x)\left(\sinc\left(\frac{2\tilde{\rho}_{j}x}{\pi}-2n-1\right)+\sinc\left(\frac{2\tilde{\rho}_{j}x}{\pi}+2n+1\right)\right)\\
 & -\Psi_{N}(\tilde{\rho}_{j},0)\sum_{n=0}^{N-1}\phi_{n}(x)\left(\sinc\left(\frac{\tilde{\rho}_{j}x}{\pi}-n\right)+\sinc\left(\frac{\tilde{\rho}_{j}x}{\pi}+n\right)\right)\\
 & =\left(\mathring{\psi}_{N}(\tilde{\rho}_{j},0)-h\Psi_{N}(\tilde{\rho}_{j},0)\right)\frac{\sin(\tilde{\rho}_{j}x)}{\tilde{\rho}_{j}}-\Psi_{N}(\tilde{\rho}_{j},0)\tilde{\rho}_{j}xj_{1}(\tilde{\rho}_{j}x)+\tilde{\rho}_{j}(L-x)j_{1}(\tilde{\rho}_{j}(L-x)).
\end{align*}
Solving this system of equations, we obtain the coefficients
\[
\left\{ \psi_{n}(x)\right\} _{n=0}^{N-1},\left\{ s_{n}(x)\right\} _{n=0}^{N-1}\,\,\text{and}\,\,\left\{ \phi_{n}(x)\right\} _{n=0}^{N-1}.
\]

For $x=L$ we have $\psi(\rho,L)=1$, then equation (\ref{eq:eq:genEqTrunca-1})
at $x=L$ evaluated at some chosen points $\left\{ \hat{\rho}_{i}\right\} _{i=1}^{I}$
has the form
\begin{align*}
 & \frac{\left(\mathring{\psi}_{N}(\hat{\rho}_{i},0)-h\Psi_{N}(\hat{\rho}_{i},0)\right)}{2\hat{\rho}_{i}\sin(\hat{\rho}_{i}L)}\sum_{n=0}^{N-1}s_{n}(L)\left(\sinc\left(\frac{2\hat{\rho}_{i}L}{\pi}-2n-1)\right)+\sinc\left(\frac{2\hat{\rho}_{i}L}{\pi}+2n+1\right)\right)\\
 & +\Psi_{N}(\hat{\rho}_{i},0)\sum_{n=0}^{N-1}\phi_{n}(L)\left(\sinc\left(\frac{\hat{\rho}_{i}L}{\pi}-n\right)+\sinc\left(\frac{\hat{\rho}_{i}L}{\pi}+n\right)\right)\\
 & =1-\left(\mathring{\psi}_{N}(\hat{\rho}_{i},0)-h\Psi_{N}(\hat{\rho}_{i},0)\right)\frac{\sin(\hat{\rho}_{i}L)}{\hat{\rho}_{i}}+\hat{\rho}_{i}Lj_{1}(\hat{\rho}_{i}L).
\end{align*}
Solving this system gives us the possibility to obtain the approximate
coefficients
\[
\left\{ s_{n}(L)\right\} _{n=0}^{N-1}\,\,\text{and}\,\,\left\{ \phi_{n}(L)\right\} _{n=0}^{N-1}.
\]

Finally, when the first coefficient $\phi_{0}(x)$ (or $s_{0}(x)$)
$x\in[0,L]$ has been computed, it is possible to recover the potential for $x\in[0,L]$
from (\ref{eq:Potential First Coeffs}).

The performance of the algorithm is illustrated by the following example.

\begin{example} \label{-Consider-theC1-1}Consider the potential
\[
q(x)=\frac{1}{(x-0.5)^{1/3}}+0.4+0.5i,\,x\in[0,1]
\]

and parameters $h=0$, $H=0$ in (\ref{eq:FirstSLP}) and (\ref{eq:SecondSLPP}),
i.e., the Neumann-Neumann and Dirichlet-Neumann Sturm-Liouville problems.
Fifteen pairs of eigenvalues of these two problems are given. They
were computed for the real part of the potential by Matslise, and
a shift of them by $0.5i$ generate the input data of the inverse
problem. It is worth mentioning that these eigenvalues possess a considerable
level of noise. The following simple test illustrates the situation.
The first 15 eigenvalues of the Neumann-Neumann Sturm-Liouville problem
with the potentials $q_{1}(x)=\frac{1}{(x-0.5)^{1/3}}+0.4$ and $q_{2}(x)=\frac{1}{(x-0.5)^{1/3}}$
were computed by Matslise. However, the eigenvalues obtained for $q_{1}(x)$
and a shift by $0.4$ of the eigenvalues for $q_{2}(x)$ resulted
to differ by $9\times10^{-5}$. This test reveals that the ``exact''
input data contain errors, possibly in the fifth digit. Nonetheless,
the algorithm for solving the inverse problem allows us to recover
the potential from the given noisy data. Figure \ref{fig:Abs.-errors-oOldSin15Eigen-1}
presents the recovered potential. 
\begin{figure}[H]
\begin{centering}
\includegraphics[scale=0.62]{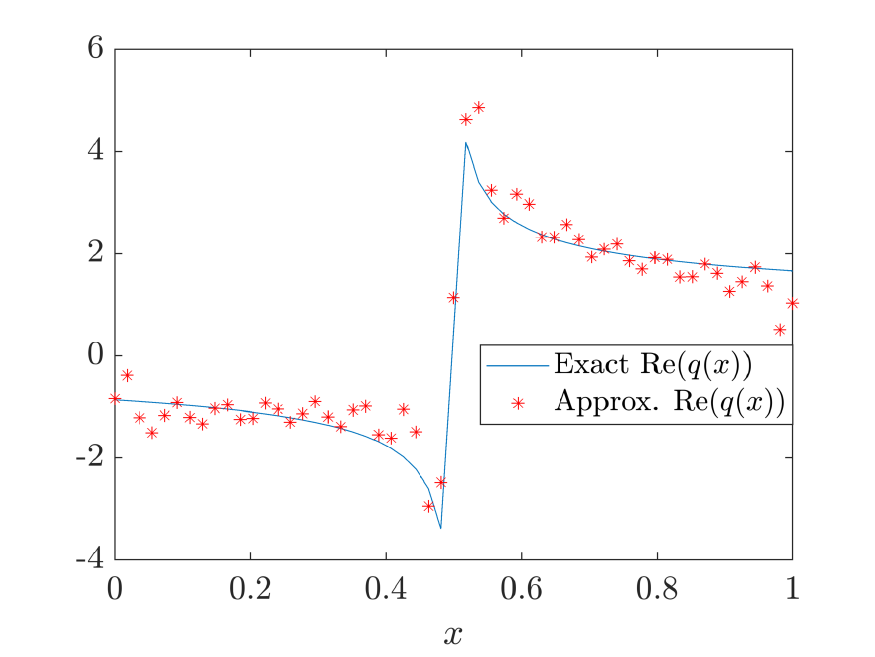}\includegraphics[scale=0.62]{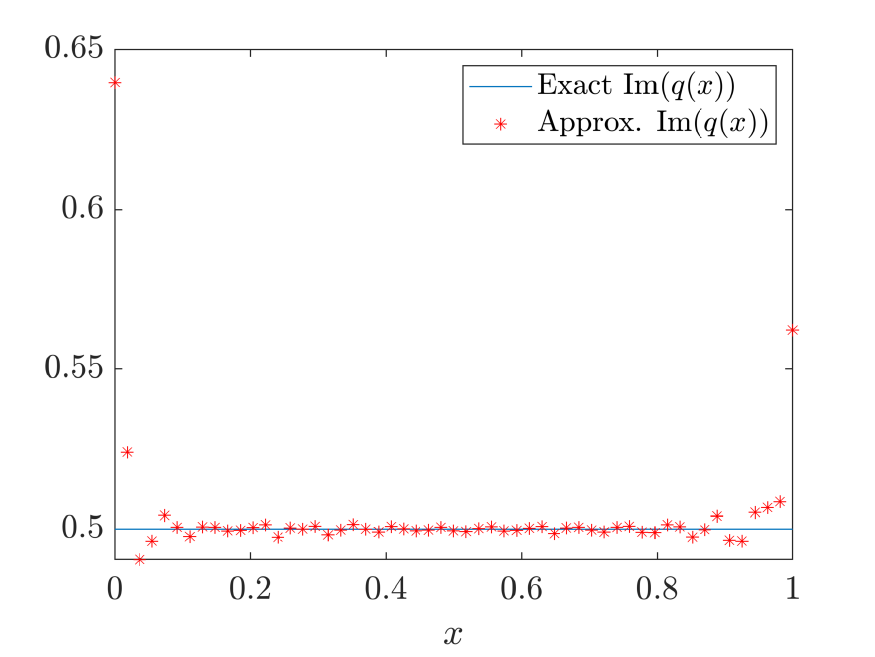}
\par\end{centering}
\caption{\label{fig:Abs.-errors-oOldSin15Eigen-1}Example \ref{-Consider-theC1-1}.
Real (left) and imaginary (right) parts of the potential recovered
from given 15 eigenpairs.}
\end{figure}

The maximum absolute error of the recovered potential is $9.5\times10^{-1},$
see Figure \ref{fig:Abs.-errors-oOldSin15Eigen}.

\begin{figure}[H]
\begin{centering}
\includegraphics[scale=0.7]{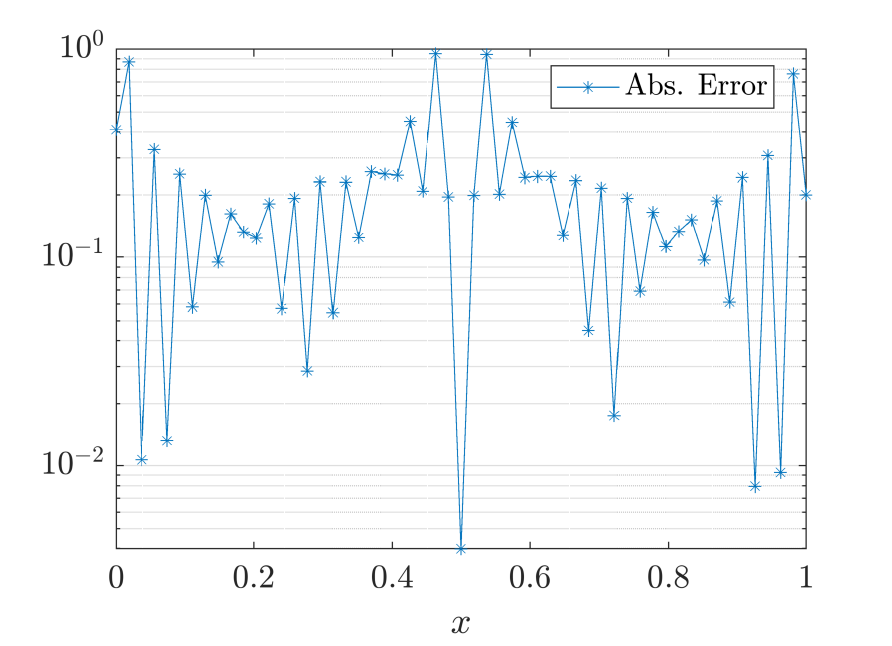}
\par\end{centering}
\caption{\label{fig:Abs.-errors-oOldSin15Eigen}Abs. error of the potential
of Example \ref{-Consider-theC1-1} recovered from given 15 eigenpairs.}
\end{figure}

This example shows the possibility of recovering a complex valued
singular potential from a limited set of noisy eigenvalues using the
proposed algorithm for potentials in \textbf{$\mathcal{L}_{2}[0,L]$}.

\end{example}

\section{Conclusions}

Cardinal series representations for the solutions of the Sturm-Liouville
equation $-y''+q(x)y=\rho^{2}y$ with a complex valued potential $q(x)$
are obtained. They enjoy remarkable convergence properties, which
make them a very convenient tool for solving a variety of spectral
problems. In particular, with their aid the spectrum completion problem
can be efficiently solved, allowing one to compute with a good accuracy
large sets of the eigenvalues from a reduced number of the given ones
and without any information on the potential and on the constants
from the boundary conditions. This idea leads to an efficient computation
of a Weyl function from two reduced finite sets of the eigenvalues.
This application of the cardinal series representations is developed
in the present paper and illustrated by numerical tests. Finally,
a method for solving inverse two-spectra Sturm-Liouville problems
for complex potentials is developed, based on the cardinal series
representations. Its numerical efficiency is discussed and illustrated
by numerical tests.

\textbf{Acknowledgments} Research was supported by CONAHCYT, Mexico
via the project 284470. The authors thank the helpful comments of
Sergii M. Torba.

\textbf{Data availability} The data that support the findings of this
study are available upon reasonable request. 

\textbf{Conflict of interest} This work does not have any conflict
of interest.


\begin{thebibliography}{10}
\bibitem{Whittaker} E. Whittaker, On the functions which are represented
by the expansion of the interpolation theory. Proc. Roy. Soc. Edinburgh,
vol. 35, pp. 181-194, (1915).
\end{thebibliography}
\end{document}